\documentclass[leqno]{article}
\Large
\usepackage{graphicx} 
\usepackage{footmisc}
\usepackage[all,cmtip]{xy}
\usepackage{amsmath, amssymb, amscd, amsthm, mathtools, amsfonts, mathrsfs, tikz-cd, bbm, enumitem}
\usetikzlibrary{positioning}
\usepackage[mathscr]{euscript}
\usepackage{hyperref}
\usetikzlibrary {matrix}
\usetikzlibrary {graphs}
\usetikzlibrary{decorations.markings}
\usepackage[utf8]{inputenc}
\usepackage[english]{babel}
\usepackage{stmaryrd}
\usepackage{xcolor}
\usepackage{marginnote}
\usetikzlibrary {shapes.geometric}
\theoremstyle{definition}
\newtheorem{Def}{Definition}[section]

\newtheorem{Rem}[Def]{Remark}
\newtheorem{Rem*}[]{Remark}
\newtheorem{Exm}[Def]{Example}

\theoremstyle{plain}
\newtheorem{Prop}[Def]{Proposition}
\newtheorem{Thm}[Def]{Theorem}
\newtheorem*{Thm*}{Theorem}
\newtheorem{Lem}[Def]{Lemma}
\newtheorem{Cor}[Def]{Corollary}

\title{Arithmetic field theory via pro-p duality groups}
\author{Oren Ben-Bassat, Nadav Gropper}
\DeclareMathOperator{\Hom}{Hom}
\DeclareMathOperator{\Aut}{Aut}

\DeclareMathOperator{\Fun}{Fun}
\DeclareMathOperator{\colim}{colim}
\DeclareMathOperator{\hcolim}{hcolim}
\DeclareMathOperator{\id}{id}
\DeclareMathOperator{\ob}{ob}

\DeclareMathOperator{\Irr}{Irr}
\DeclareMathOperator{\rk}{rk}
\DeclareMathOperator{\gr}{gr}
\DeclareMathOperator{\cd}{cd}
\DeclareMathOperator{\PD}{PD}
\DeclareMathOperator{\Cob}{Cob}
\DeclareMathOperator{\Gal}{Gal}
\DeclareMathOperator{\Grpd}{Grpd}
\DeclareMathOperator{\fGrpd}{fGrpd}
\DeclareMathOperator{\pGrpd}{p-Grpd}
\DeclareMathOperator{\pGrp}{p-Grp}
\DeclareMathOperator{\Span}{span}
\DeclareMathOperator{\Cospan}{cospan}

\DeclareMathOperator{\Proj}{Proj}
\DeclareMathOperator{\TQFT}{TQFT}
\DeclareMathOperator{\Frob}{Frob}
\DeclareMathOperator{\wHom}{\widetilde{\Hom}}

\begin{document}
\maketitle{}
\makeatletter
\tikzset{nomorepostaction/.code={\let\tikz@postactions\pgfutil@empty}}
\makeatother
\newsavebox{\TrBox} 
\savebox{\TrBox}{  \begin{tikzpicture} 
\draw(0.3,0.65) to[out=0,in=210] (1,1) to[out=30,in=150] (3,1) to[out=-30,in=180] (3.7,0.65); \draw(3.7,-0.65) to[out=180,in=30] (3,-1) to[out=210,in=-30] (1,-1) to[out=150,in=0] (0.3,-0.65);
\draw[smooth] (1.4,0.1) .. controls (1.8,-0.25) and (2.2,-0.25) .. (2.6,0.1);
\draw (1.5,0.015) .. controls (1.8,0.2) and (2.2,0.2) .. (2.5,0.015); 
\draw (0.3,0.65) arc(90:450:0.3 and 0.65); 
\draw[dashed] (3.7,0.65) arc(90:270:0.3 and 0.65); 
\draw (3.7,0.65) arc(450:270:0.3 and 0.65); 
\filldraw[blue] (1,0.65) circle (2pt) node[align=right,  below] {$p^r$} ;
\end{tikzpicture} }
\newcommand{\Tr}{\usebox{\TrBox}}
\newsavebox{\TorBox} 
\savebox{\TorBox}{  \begin{tikzpicture} 
\draw(0.3,0.65) to[out=0,in=210] (1,1) to[out=30,in=150] (3,1) to[out=-30,in=180] (3.7,0.65); \draw(3.7,-0.65) to[out=180,in=30] (3,-1) to[out=210,in=-30] (1,-1) to[out=150,in=0] (0.3,-0.65);
\draw[smooth] (1.4,0.1) .. controls (1.8,-0.25) and (2.2,-0.25) .. (2.6,0.1);
\draw (1.5,0.015) .. controls (1.8,0.2) and (2.2,0.2) .. (2.5,0.015); 
\draw (0.3,0.65) arc(90:450:0.3 and 0.65);
\draw[dashed] (3.7,0.65) arc(90:270:0.3 and 0.65); 
\draw (3.7,0.65) arc(450:270:0.3 and 0.65); 
\end{tikzpicture} }
\newcommand{\Torb}{\usebox{\TorBox}}

\newsavebox{\PantsLBox} 
\savebox{\PantsLBox}{  \begin{tikzpicture} 
\draw [smooth] (0,1.25) to[out=0,in=140] (1,1) to[out=-40,in=180] (1.7,0.65) ;
\draw [smooth] (1.7,-0.65) to[out=180,in=40] (1,-1) to[out=220,in=0] (0,-1.25);
\draw[smooth] (0,0.25) arc(270:450:0.35 and -0.25);
\draw (0,0.25) arc(90:450:0.3 and -0.5);
\draw (0,-0.25) arc(90:450:0.3 and 0.5);
\draw[dashed] (1.7,-0.65) arc(90:270:0.25 and -0.65);
\draw (1.7,-0.65) arc(450:270:0.25 and -0.65);
\end{tikzpicture} }
\newcommand{\midarrow}{\tikz \draw[-triangle 90] (0,0) -- +(.1,0);}
\newcommand{\PantsL}{\usebox{\PantsLBox}}

\newsavebox{\PantsRBox} 
\savebox{\PantsRBox}{  \begin{tikzpicture} 
\draw [smooth] (4.4,1.25) to[out=180,in=40] (3.4,1) to[out=-140,in=0] (2.7,0.65) ;
\draw [smooth] (2.7,-0.65) to[out=0,in=140] (3.4,-1) to[out=320,in=180] (4.4,-1.25);
\draw[smooth] (4.4,0.25) arc(270:450:-0.35 and -0.25);
\draw[dashed] (4.4,0.25) arc(90:270:0.3 and -0.5);
\draw (4.4,0.25) arc(450:270:0.3 and -0.5);
\draw[dashed] (4.4,-0.25) arc(90:270:0.3 and 0.5);
\draw (4.4,-0.25) arc(450:270:0.3 and 0.5);
\draw (2.7,-0.65) arc(90:450:0.25 and -0.65);
\end{tikzpicture} }
\newcommand{\PantsR}{\usebox{\PantsRBox}}

\newsavebox{\capBox} 
\savebox{\capBox}{  \begin{tikzpicture} 
\draw (0,0.25) arc(270:450:0.7 and 0.5);
\draw (0,0.25) arc(90:450:0.3 and -0.5);
\end{tikzpicture} }
\newcommand{\capB}{\usebox{\capBox}}

\newsavebox{\cupBox} 
\savebox{\cupBox}{  \begin{tikzpicture} 
\draw (2.5,0.25) arc(270:450:-0.7 and 0.5);
\draw (2.5,0.25) arc(450:270:0.3 and -0.5);
\draw[dashed] (2.5,0.25) arc(90:270:0.3 and -0.5);
\end{tikzpicture} }
\newcommand{\cupB}{\usebox{\cupBox}}

\newsavebox{\cylBox} 
\savebox{\cylBox}{  \begin{tikzpicture} 
\draw[smooth] (0,0.25) to (1.5,0.25);
\draw[smooth] (0,1.25) to (1.5,1.25);
\draw (0,0.25) arc(90:450:0.3 and -0.5);
\draw[dashed] (1.5,0.25) arc(90:270:0.3 and -0.5);
\draw (1.5,0.25) arc(450:270:0.3 and -0.5);
\end{tikzpicture} }
\newcommand{\cyl}{\usebox{\cylBox}}

\newsavebox{\cylBoxAlpha} 
\savebox{\cylBoxAlpha}{\begin{tikzpicture}  \draw[smooth] (0,0.25) to (1.5,0.25); \draw[smooth] (0,1.25) to (1.5,1.25); \draw (0,0.25) arc(90:450:0.3 and -0.5); \draw[dashed] (1.5,0.25) arc(90:270:0.3 and -0.5); \draw (1.5,0.25) arc(450:270:0.3 and -0.5);
 \node [isosceles triangle, fill=black, shape border rotate=90, inner xsep=0.09pt] at (0.297,0.7){}; \node [isosceles triangle, fill=black, shape border rotate=270, inner xsep=0.09pt] at (1.793,0.7){}; \draw (2,0.7) node {$\alpha$}; \draw (0.5,0.75) node {\small $1$}; \end{tikzpicture}}
 \newcommand{\cylAlpha}{\usebox{\cylBoxAlpha}}

 \newsavebox{\swapBox}
 \savebox{\swapBox}{\begin{tikzpicture}  \draw (0,1.75) arc(90:450:0.3 and -0.5); \draw[dashed] (2,1.75) arc(90:270:0.3 and -0.5); \draw (2,1.75) arc(450:270:0.3 and -0.5);
\draw (0,0.25) arc(90:450:0.3 and -0.5); \draw[dashed] (2,0.25) arc(90:270:0.3 and -0.5); \draw (2,0.25) arc(450:270:0.3 and -0.5);
\draw[smooth] (0,1.75) .. controls(0.3,1.85)and (1.7,0.15).. (2,0.25); \draw[smooth] (0,2.75) .. controls(0.3,2.85)and (1.7,1.15).. (2,1.25);
\draw[smooth] (0,0.25) .. controls(0.33,0.32).. (0.99,1); \draw[dashed] (0.99,1) to (1.54,1.5); \draw[smooth](1.54,1.5) .. controls (1.72,1.63)..(2,1.75);
\draw[smooth] (0,1.25) .. controls(0.3,1.3) .. (0.48,1.48); \draw[dashed] (0.48,1.48) to (0.99,2); \draw[smooth] (0.99,2) .. controls(1.6,2.68).. (2,2.75);
 \end{tikzpicture}}
 \newcommand{\swap}{\usebox{\swapBox}}
 
\newsavebox{\PairingBox} 
\savebox{\PairingBox}{  \begin{tikzpicture} 
\draw [smooth] (0,1.25) arc(270:450:1.3 and -1.25);
\draw[smooth] (0,0.25) arc(270:450:0.35 and -0.25);
\draw (0,0.25) arc(90:450:0.3 and -0.5);
\draw (0,-0.25) arc(90:450:0.3 and 0.5);
\end{tikzpicture} }
\newcommand{\Pairing}{\usebox{\PairingBox}}

\newsavebox{\CoPairingBox} 
\savebox{\CoPairingBox}{  \begin{tikzpicture} [rotate=180]
\draw [smooth] (0,1.25) arc(270:450:1.3 and -1.25);
\draw[smooth] (0,0.25) arc(270:450:0.35 and -0.25);
\draw (0,0.25) arc(90:270:0.3 and -0.5);
\draw (0,-0.25) arc(90:270:0.3 and 0.5);
\draw[dashed] (0,0.25) arc(450:270:0.3 and -0.5);
\draw[dashed] (0,-0.25) arc(450:270:0.3 and 0.5);
\end{tikzpicture} }
\newcommand{\CoPairing}{\usebox{\CoPairingBox}}
 
\begin{abstract}
Using the theory of pro-p groups and relative Poincar\'{e} duality, we define a type of cobordism category well suited to arithmetic topology. We completely classify topological quantum field theories on these two-dimensional versions of our cobordism categories. This classification uses Frobenius algebras with extra operations corresponding to automorphisms of the p-adic integers. We look in more detail at the example of arithmetic Dijkgraff--Witten theory for a finite gauge p-group in this setting. This allows us to deduce formulae counting Galois extensions of local p-adic fields whose Galois groups are the given gauge group.
\end{abstract}
\section{Introduction}

Quantum field theory (QFT) is a framework for organizing principles for many structures in mathematics and physics (for some recent mathematical perspectives on QFTs see \cite{QFT1, QFT2, QFT3, QFT4}).
When the theory depends only on the topology of the space, it is called a Topological Quantum Field Theory (TQFT). Axiomatic frameworks for TQFTs were first studied by Atiyah  \cite{atiyah1988topological} and Segal \cite{segal1988definition}.

Arithmetic topology (sometimes known as the knots primes analogy) draws analogies between number theory and low dimensional topology \cite{mazur1973notes, morishita2024knots}. Work on this analogy is due to Mazur, Mumford, Kapranov, Reznikov, Morishita, and others. Following this philosophy,  Minhyong Kim in his Arithmetic Chern-Simons Theory \cite{MR4166930, MR3857144}, proposed one can try and find arithmetic versions of quantum field theoretic ideas.
Since then, there has been a surge of interest in arithmetic analogues of TQFTs and the role they play in number theory and arithmetic geometry, some examples of arithmetic QFTs can be found in \cite{ACSII,chung2019abelian,Hirano:2023hqu, Pappas21, MR4648746, MR4586262, MR4474125, MR3888342}, and more recently in connections to the Langland's program and its generalizations \cite{ben2024relative}. \\
So far the various existing arithmetic QFTs and TQFTs are as explicit constructions of analogues of specific theories from physics, and there has yet to be a general framework for arithmetic cobordisms and TQFTs. 

In this article, we introduce a new general framework for $(d+1)$-dimensional arithmetic (pro-$p$) TQFT, this framework generalizes classical axiomatic frameworks of TQFTs such as ones given by Atiyah, Segal, and by Turner-Turaev.
Vaguely speaking, in this article, we consider and reformulate TQFT and cobordisms in a completely group theoretic way as opposed to the classical way of looking at topological or geometric objects.
To do this we consider pro-$p$ groups with certain duality properties, this enables us to study pro-$p$ cobordisms and TQFTs for arithmetic objects (such as $p$-adic fields in dimension $2$) and manifolds (spacetime) at the same time.
This approach allows one to include certain objects and morphisms of arithmetic nature, and to decompose them into objects of more geometric nature.

More precisely, closed manifolds are replaced with (pro-$p$) groups with a Poincar\'{e} duality on their group cohomology.
Manifolds with boundary are replaced by groups with a relative version of Poincar\'{e} duality. That is by a pair $(G,\mathcal{S})$, where $\mathcal{S}$ is a collection of closed subgroups of $G$, such that the relative group cohomology of the pair (in the sense of \cite{wilkes2019relative}) has a Poincar\'{e} duality.

Our work not only provides an axiomatic framework for considering arithmetic and classical cobordisms at the same time but also as far as we know this is the first purely group theoretic version of cobordism categories and TQFTs. Group theoretic cobordism rings were considered by Bieri and Eckmann in \cite{MR0393197} (i.e. $\pi_0$ of the category we define here). 

In dimension $2$, it is known that $\PD^2_{\mathbb{Z}}$ (groups with $\PD$ over $\mathbb{Z}$ are exactly surface groups and it is conjectured that in dimension $3$ such duality groups correspond to $3$-manifolds.
In the pro-$p$ world, they also have a very nice form.

The objects of such cobordisms are finite collections of the cyclic pro-$p$ group $\mathbb{Z}_{p}$ (so a “collection of circles”) and morphisms between such objects $\mathcal{U}$ and $\mathcal{N}$ are pairs of groups of the form
\[G=\langle x_{1},y_{1},...,x_{n},y_{n},b_{1},...b_{s} \rangle\]
a free pro-$p$ group on $2n+s$ generators, and boundary components \[\mathcal{S}=\{B_{0},B_{1},...,B_{s}\}\] and \[b_{0}=x_{1}^{p^{r}}[x_{1},y_{1}]\cdots[x_{n},y_{n}]b_{1}\cdots b_{s},\] $B_{i}$ is the closed subgroup generated by $b_{i}$, such that $\mathcal{S}=\mathcal{N}\coprod\mathcal{U}$ (together with two more types of morphisms corresponding to caps and cups we formally add in the cobordisms).

Since the objects we study are not always "orientable", we follow ideas introduced by Turner and Turaev in \cite{turaev2006unoriented} for studying unoriented cobordisms.
We note here that there are other formalisms for unoriented cobordisms and TQFTs, for example \cite{karimipour1997lattice, young2020orientation}.
Turner and Turaev show that equivalence classes of $2$-dimensional unoriented TQFTs are in bijection with extended Frobenius algebras.
That is a commutative Frobenius algebra together with the usual multiplication $m$, co-multiplication $\Delta$, unit $\iota$ and co-unit $\epsilon$, as well as new maps $\theta: R\to V$ and $\phi:V \to V$ which satisfy $\phi \circ \phi = \text{id}$ as well as 
\[\phi \circ m \circ ( \theta \otimes \text{id}) = m\circ ( \theta \otimes \text{id})\]
and
\[
m \circ (\phi \otimes \text{id}) \circ \Delta \circ \iota =m\circ (\theta \otimes \theta).
\]

One of the main differences in our setting is that rather than just orientation preserving or non-preserving automorphisms coming from $\Aut(\mathbb{Z})=\{\pm 1\}$, in the pro-p world, the orientations are controlled by the larger group $\Aut(\mathbb{Z}_{p})$.
In Definition \ref{Def:ExFrAlg} we define a pro-p version of extended Frobenius algebra. 
The axioms of these extended Frobenius algebras are a direct generalizations of the ones in \cite{turaev2006unoriented}.

Our main theorem gives a classification of pro-p TQFTs of dimension $1+1$, in terms of these extended Frobenius algebras:

\begin{Thm}\label{Thm:MainIntro}
    Isomorphism classes of $(1+1)$ pro-$p$ TQFTs at $p$ (which are mod $p$ orientation compatible) are in a bijective correspondence with isomorphism classes of $\Aut(\mathbb{Z}_{p})$-extended $R$-Frobenius algebras for any ring $R$.
\end{Thm}
To obtain the above theorem, we use the analysis of splittings of pro-p $PD^2$ groups, from \cite{GropperArxiv}.
In a sense, we get a sort of “pair of pants” decomposition to every pro-$p$ $\PD^{2}$ pair.
The power of having such classification is that it allows the study of certain invariants for a $p$-adic field, by first studying them for pairs of pants, and then gluing them together to a $p$-adic field.
Such pairs of pants can be glued to one another along their boundaries in nonstandard ways, corresponding to the large group of automorphisms of $\mathbb{Z}_{p}$.

A practical consequence of this work is that using a specific example of TQFTs in our theory one can re-derive Yamagishi’s formula \cite{Yama} for the number of extensions of a $p$-adic field with a given Galois group. 
As opposed to Yamagishi's algebraic arguments, we use more ``geometric” arguments (i.e. cutting into pairs of pants).
These ``geometric” arguments generalize the arguments used by Mednykh \cite{MR490616} to get a formula for counting $\Gamma$-covers of a manifold S.\\

\begin{Cor}\label{Cor:MainCorIntro}
    
    Let $K$ be a $p$-adic field of degree $[K:\mathbb{Q}_{p}]=n$, such that $p\neq2$ and K contains the $p^{r}$-th roots of unity, but not the $p^{r+1}$-th roots of unity. Let $G_{K}$ be the absolute Galois group of $K$, and let $\Gamma$ be a p-group, we have the following: \[|\Hom(G_{K},\Gamma)|=\frac{1}{|\Gamma|^{n}}\underset{\chi\in \Irr(\Gamma)}{\sum}\chi(1)^{-n}\underset{\gamma\in\Gamma}{\sum}\chi(\gamma^{p^{r}-1})\chi(\gamma).\]

Here $\Irr(\Gamma)$ is the set of traces of equivalences classes of finite dimensional complex irreducible representations of $\Gamma$. Thus we get that the number of extensions L of a field K, with $\Gal(L/K)\cong\Gamma$ is\[\frac{1}{|\Aut(\Gamma)|}\sum_{H\leq\Gamma}\mu(H)|\Hom(G_{K},H)|.\]
\end{Cor}
The last equation comes from Hall's M\"{o}bius inversion formula on groups \cite{hall1936eulerian}.
This corollary comes from our main example of a pro-$p$ TQFT, which looks at complex functions on stacks of $\Gamma-$bundles for some finite gauge group $\Gamma$. Such a TQFT can be seen as an arithmetic version of Dijkgraaf-Witten theory \cite{MR1048699}. 
We explain how to compute these invariants when one replaces $G_K$ with a general pro-$p$ $\PD^2$ group, and $\Gamma$ with a  general finite $p$-group.\\
In the future, we hope to put other arithmetic TQFTs into our formalism. For example in an ongoing sequel \cite{Part2} to this paper, we have begun studying more general Dijkgraaf-Witten theories, coming from twisting by a class in $H^{2}(\Gamma,\mathbb{C}^{\times})$, and deformations of such theories. The resulting TQFT will produce formulae similar to those in \cite{turaev2007dijkgraaf}. 
\subsection*{Outline of the paper}
We summarize the relevant background on pro-p graphs of groups and relative Poincare duality groups in Section \ref{sect:prelim}.\\
The main definitions of our paper are in Section \ref{sect:MainDefs}. We start by giving the general formalism for pro-$p$ TQFTs in Subsection \ref{subsect:TQFT}. We then focus on the case of $(1+1)$ dimensional TQFTs, for which there is an equivalent notion of a TQFT coming from symmetric monoidal functors from a certain category of cospans of groupoids, $Cob^2_p$, to projective $R$-modules, this is discussed and defined in Subsection \ref{subsect:Cob}.
In Subsection \ref{subsect:Frob} we define $\mathbb{U}_p$-extended Frobenius algebras and give a first example of such algebras.\\
We study the structure of the category $Cob_p^2$ in Section \ref{sect:Structure}, and describe a set of generators for it in subsection \ref{subsect:Gens}.
We then find relations between said generators, and show that these relations are all the relations of $Cob_p^2$ in Subsection \ref{subsect:Rels}.\\
Finally, in Section \ref{sect:Equiv}, we prove our main Theorem \ref{Thm:MainIntro} by using the structure we described in Section \ref{sect:Structure}.

We define a pro-$p$ version of a $(1+1)$ dimensional Dijkgraaf-Witten theory in Section \ref{sect:DW}. We show it is indeed a TQFT, and provide explicit computations of the image of the generators of $Cob_p^2$ under this TQFT.
We use these computations to get Corollary \ref{Cor:MainCorIntro}.

We also have two appendices, the first Appendix \ref{App:Cats} provides some background on pro-$p$ groupoids and their amalgams. Appendix \ref{App:Cospan} contains most of the technical details for why a $(1+1)$ dimensional pro-$p$ TQFT as defined in Definition \ref{Def:TQFT-TurnerTaraev} is equivalent to a symmetric monoidal functor from $Cob_p^2$.

\subsection*{Notation and conventions}
In this article, $p$ is a fixed odd prime number. We use 
 \[\mathbb{U}_p=\ker [U(\mathbb{Z}_p)\to U(\mathbb{Z}/p)]
 \]
 for the subgroup of units of $\mathbb{Z}_p$ which are trivial mod $p$. Given a pro-p group $G$ and elements $\lambda\in\mathbb{Z}_p$, $x\in G$, we denote by $x^\lambda=\underset{i\rightarrow\infty}\lim x^{n_i}$ for some sequence of integers $n_i$ converging to $\lambda$ (this is independent of the choice of sequence). We use the pro-p version of generators and relations notation, so for example $<x> = \{x^\alpha \ | \ \alpha \in \mathbb{Z}_p \}\cong \mathbb{Z}_p$ and $<x,y>$ is the pro-$p$ completion of the free group on two generators.
 
 We use $\pGrp$ for the category of finitely presented pro-$p$ groups. We denote the category of groupoids  with finite $\pi_0$ and all automorphism groups finite by $\fGrpd$. We denote the category of pro-$p$-groupoids with finite $\pi_0$ and all automorphism groups finitely presented pro-$p$ groups by $\pGrpd$.
 Amalgams in this paper (unless explicitly stated otherwise) are considered in the category of pro-p groups and the amalgamation will be denoted by $*$.
 Whenever we have finite families of objects $\{A_i\}_{i\in I}$, $\{B_j\}_{j\in J}$, an isomorphism between them, is a family of isomorphisms $\phi_i:A_i\rightarrow B_j$ for some $j\in J$, such that this induces a bijection between $I$ and $J$. 
In this article, groups are usually denoted as $G, H, \dots$, collections of subgroups as $\mathcal{S}, \mathcal{N}, \dots$, and groupoids as $\mathfrak{G}, \mathfrak{H}, \dots$.

Given two groupoids $\mathfrak{G}$ and $\mathfrak{H}$ we use $\Hom(\mathfrak{G},\mathfrak{H})$ to denote the natural equivalence classes of functors from $\mathfrak{G}$ to $\mathfrak{H}$, and use $\Fun(\mathfrak{G},\mathfrak{H})$ to denote the groupoid whose objects are functors from  $\mathfrak{G}$ to $\mathfrak{H}$ and whose morphisms are invertible natural transformations between functors.

Given symmetric monoidal categories $C$ and $D$, we write $\Fun^\otimes(C,D)$ for the groupoid whose objects are (strong not lax) symmetric monoidal functors from $C$ to $D$ and whose morphisms are natural equivalences. For an arbitrary commutative unital ring, we usually use $R$, and $\Proj(R)$ is the symmetric monoidal category of finitely presented projective $R$-modules with the usual tensor product. We use $\Gamma$ for a finite $p$-group.
\subsection*{\textbf{Acknowledgments}}
O.B. (University of Haifa) and N.G. (University of Pennsylvania and University of Haifa) would like to acknowledge Tony Pantev for helpful conversations during the writing of this article. We would also like to thank Pantev for funding from his NSF grant number 2021717 (part of the NSF-BSF program in collaboration with O.B. associated to Pantev's NSF award 2200914 with project Title: NSF-BSF: Derived and quantum corrected structures on arithmetic and geometric moduli). This funding, together with supporting funds from the Universities of Pennsylvania and Haifa supported N.G. as a postdoc. We would also like to thank the Bloom Graduate School of the University of Haifa for awarding a Bloom Scholarship Program Postdoctoral Research Grant to N.G.. We would like to thank Jonathan Fruchter, Minhyong Kim, Kobi Kremnizer, Masanori Morishita, Gabriel Navarro, Nick Rozenblyum, and Matthew Young for helpful conversations.

\newpage
\section{Preliminaries on Poincar\'{e} duality groups and pairs}\label{sect:prelim}
In this section, we will give a quick overview of the framework of Poincar\'{e} duality groups. These are some of the main objects which we will later use to replace manifolds.
The theory of relative group cohomology and  relative Poincar\'{e} duality groups was developed in \cite{Bierie-Eckmann-relative} for discrete groups, and in the profinite setting in \cite{wilkes2019relative}.
We give here the main definitions from \cite{wilkes2019relative} which we will use in this paper.

Given a profinite group $G$, the $\mathbb{Z}_p$-algebra $\mathbb{Z}_p[[G]]$ is defined by 
the completion of the group algebra $\mathbb{Z}_p[G]$ equipped with the topology where a system of open neighborhoods of $0$ is given by the kernels of the maps 
\[
\mathbb{Z}_p[G] \longrightarrow (\mathbb{Z}_p / p^n \mathbb{Z}_p)[G/U]
\]
where $U$ is an open normal subgroup of $G$ and $n$ is an integer.
If $M$ is a $G$-module in the category of profinite groups, then $\mathbb{Z}_p[[M]]$ becomes a module for $\mathbb{Z}_p[[G]]$ in the category of $\mathbb{Z}_p$-modules. Such a module will usually just be called a $G$-module if the meaning is clear.
\begin{Def}\label{Def:GroupPairs}
    We define a profinite group pair to be $(G,\mathcal{S})$, where $G$ is a profinite group, and $\mathcal{S}=\{S_i\}_{i\in I}$ is a finite family of (not necessarily distinct) closed subgroups of $G$.
    A morphism between group pairs $(G,\mathcal{S})\rightarrow(H,\mathcal{F})$ consists of a homomorphism $\phi:G\rightarrow H$ and a function between the index sets $f:J\rightarrow I$, such that $\phi(F_j)\subseteq S_{f(j)}$ for all $j\in J$
\end{Def}
\begin{Rem}
    More generally one can index the closed subgroups by a profinite space $X$ (see \cite{wilkes2019relative}), but for all purposes in this paper we will only need a finite indexing set.
\end{Rem}
Let $(G,\mathcal{S})$ be a profinite group pair.  The $G$-module $\mathbb{Z}_p[[G/\mathcal{S}]]$ is defined by a direct sum of $G$-modules:
\[\mathbb{Z}_p[[G/\mathcal{S}]] = \underset{S \in \mathcal{S}}\bigoplus \mathbb{Z}_p[[G/S]]\]
where each $\mathbb{Z}_p[[G/S]]$ is the $G$-module given by the completion of $\mathbb{Z}_p[G/S]$ with respect to the topology whose system of open neighborhoods of $0$ is defined by the kernels of the maps 
\[
\mathbb{Z}_p[G/S]\to (\mathbb{Z}_p / p^n \mathbb{Z}_p)[G/U]
\]
where $U$ is an open normal subgroup of $G$ containing $S$ and $n$ is an integer. The augmentation gives a short exact sequence of $G$-modules:
\[
0 \to \triangle_{(G, \mathcal{S})} \to \mathbb{Z}_p[[G/\mathcal{S}]] \to \mathbb{Z}_p \to 0.
\]
where $G$ acts trivially on  $\mathbb{Z}_p$.
\begin{Def}The category $\mathfrak{D}_p(G)$ is defined as the full subcategory of topological abelian groups whose objects are colimits of finite $p$-primary $G$-modules.
\end{Def}
\begin{Def} The pair $(G, \mathcal{S})$ is of type $pFP_\infty$ when $\triangle_{(G, \mathcal{S})}$ is of type $pFP_\infty$ as a $G$-module, meaning that it has a resolution by finitely generated projective $\mathbb{Z}_p[[G]]$-modules.
\end{Def}
\begin{Def}The pair $(G, \mathcal{S})$ is assigned a number $\cd_p(G, \mathcal{S})$ defined by 
\[
\cd_p(G, \mathcal{S}) = \max \{n \ \ | \ \ \exists A \in \mathfrak{D}_p(G) \ \ \text{such that} \ \ H^n(G, \mathcal{S}; A) \neq 0 \}
\]
\end{Def}
\begin{Def}\label{Def:pdnpair} The pair $(G, \mathcal{S})$ is a Poincar\'{e} duality pair of dimension $n$ at the prime $p$ (called a $\PD^n$ pair at $p$) if $(G, \mathcal{S})$ is of type $pFP_\infty$, $\cd_p(G, \mathcal{S})=n>0$, $H^k(G, \mathcal{S}; \mathbb{Z}_p[[G]])$ is trivial for all $k\neq n$ and that $H^n(G, \mathcal{S}; \mathbb{Z}_p[[G]])\cong \mathbb{Z}_{p}$ as a $\mathbb{Z}_{p}$-module. 
\end{Def}
    
For $n=2$ we have the following classification theorem by Wilkes \cite{wilkes2020classification} (and by Demushkin when $\mathcal{S}=\emptyset$):
\begin{Thm}\label{Thm:WilkesPDn}
    Let $G$ be a pro-$p$ group, and $\mathcal{S}$ a collection of subgroups of $G$ such that $(G,\mathcal{S})$ is a $\PD^2$ pair.
    \begin{enumerate}
        \item If $\mathcal{S}=\emptyset$ then $G$ can be presented as a pro-$p$ group by $2n$ generators $x_1,y_1,...,x_n,y_n$ and a single relation $x_1^{p^r}[x_1,y_1]\cdots[x_n,y_n].$
        \item Otherwise, $\mathcal{S}=\{S_0,S_1,...,S_b\}$ and $G$ is a free pro-$p$ group on $2n+b$ generators, such that it has a generating set $x_1,y_1,...,x_n,y_n,s_1,...,s_b$ and an element \[s_0:=x_1^{p^r}[x_1,y_1]\cdots[x_n,y_n]s_1 \cdots s_b,\] such that each $s_i$ generates some conjugate of $S_i$ for $i=0, \dots, n$.
    \end{enumerate}
\end{Thm}
We will later split the boundary groups of these two-dimensional relative Poincar\'{e} duality groups into ``ingoing" and ``outgoing" boundaries (i.e. cobordisms).
\subsection{Orientation Character}
To a Poincar\'{e} duality group, one can associate an orientation character.
In dimension $2$, the orientation character together with the rank of the group determines the group.
\begin{Def}
    To a pro-$p$ $\PD^n$ pair $(G,\mathcal{S})$ we associate an orientation character $\chi:G\rightarrow \Aut(\mathbb{Z}_p)$ coming from the action of $G$ on $H^n(G, \mathcal{S}; \mathbb{Z}_p[[G]])\cong \mathbb{Z}_{p}$.
    If the map $\chi$ is trivial, we say that the pair $(G,\mathcal{S})$ is orientable.
     Denote by $r$, the largest integer for which the composition \[G \stackrel{\chi_r}\longrightarrow \Aut(\mathbb{Z}_p)\longrightarrow \Aut(\mathbb{Z}/p^r\mathbb{Z})
    \] is trivial. We refer to this $r$ as the orientability level of $(G,\mathcal{S})$.
\end{Def}
In dimension $2$, one has that a $\PD^2$ group $G$ is determined by its rank and $Im(\chi)$ (in fact just by the orientability level of $G$).
As seen in Theorem 4 of \cite{labute1967classification} (or by Proposition 1.5 and Theorem 2.2 in \cite{wilkes2020classification} to extend to the relative case), we have the orientation character of $G$ is the map \[\chi:G\rightarrow \mathbb{U}_p\]  defined by \[\chi(y_1)=(1-p^r)^{-1}\]  and \[\chi(y_i)=\chi(x_j)=\chi(s_k)=1\]
for all $i\neq 1$ and all $j,k$.

Let $K$ be a p-adic field, with maximal p Galois group $G_K(p)$, and $N=[K:\mathbb{Q}_p]$. 
Suppose that there are $p$-th roots of unity in $K$, one has that the orientation character is the p-part of the cyclotomic character, and the orientability level $r$ is the highest power of $p$, such that $p^r$ roots of unity $\mu_{p^r}$ are contained in $K$. The rank of $G_K(p)$ is $N+2$.
If there are no $p$-th roots of unity in $K$, then $G_K(p)$ is a free pro-$p$ group of rank $N+1$.

\subsection{Gluing relative Poincar\'{e} duality groups}\label{subsect:gluing}
In the classical setting of manifolds and TQFTs, a central tool for studying them, is cutting along sub-manifolds, and gluing boundaries.
We would like to have a way to look at such topological operations, from an algebraic point of view.
Van Kampen’s theorem gives us this desired connection.
\begin{Thm}[Van Kampen]\label{Thm:VanKampen}
Let $X_1$ and $X_2$ be subspaces of a topological space $X$ whose interiors cover $X$. Then the obvious functor gives an equivalence of categories
\[
\pi_{\leq 1}(X_1)\coprod_{\pi_{\leq 1}(X_1 \cap X_2)}^{h} \pi_{\leq 1}(X_2) \to \pi_{\leq 1}(X)
\]
from the homotopy pushout of the fundamental groupoids to the fundamental groupoid of $X$. If $X_1 \cap X_2$ is path connected and $c \in X_1 \cap X_2$ is a basepoint then the natural map
\[
\pi_{1}(X_1,c)*_{\pi_1(X_1 \cap X_2,c)} \pi_{1}(X_2,c) \to \pi_{1}(X,c)
.\]
is an isomorphism, where here $*$ refers to a discrete amalgamation.
\end{Thm}

In the purely group theoretic world, a framework for studying such gluing/cutting is Bass-Serre theory and graphs of groups (see \cite{serre2002trees}). In the pro-p/profinite world this theory can be found in \cite{ribes2017profinite}.
As explained in \cite{ribes2017profinite}, any such cutting can be seen as an iterated process (by standard arguments of collapsing subgraphs of a single edge and its end vertex/vertices, into a single vertex) of doing one of the following two operations:
\begin{Def}\label{Def:pro-p amalgamted}
Let $A$, $B$ and $C$ be pro-$p$ groups and $\varphi_{1}:C\rightarrow A$, $\varphi_{2}:C\rightarrow B$, be (continuous) monomorphisms of pro-$p$ groups.

The amalgamated free pro-$p$ product of $A$ and $B$ with amalgamated subgroup $C$ is a pushout in the category of pro-$p$ groups 
\[
\xymatrix{C\ar[d]_{\varphi_{2}}\ar[r]^{\varphi_{1}} & A\ar[d]_{f_{1}}\\
B\ar[r]^{f_{1}} & A*_{C}B
}
\]
In other words, a pro-$p$ group $G=A*_{C}B$, satisfying the following universal property: 
 
For any pair of homomorphisms $\psi_{1}:A\rightarrow K$, $\psi_{2}:B\rightarrow K$ into a pro-$p$ group $K$ with $\psi_{1}\circ \varphi_{1}=\psi_{2}\circ \varphi_{2}$, there exists a unique homomorphism $\psi:G\rightarrow K$ such that the following diagram is commutative:
\[
\xymatrix{C\ar[d]_{\varphi_{2}}\ar[r]^{\varphi_{1}} & G_{1}\ar[d]_{f_{1}}\ar@/^{1pc}/[ddr]^{\psi_{1}}\\
G_{2}\ar[r]^{f_{2}}\ar@/_{1pc}/[drr]_{\psi_{2}} & G\ar@{-->}[dr]^{\psi}\\
 &  & K
}
\]
\end{Def}

\begin{Def}\label{Def: pro-p HNN}
Let $A, C$ be pro-$p$ groups and $\varphi_1,\varphi_2:C\rightarrow A$ be two continuous injections from $C$ to a pro-$p$ group $A$, with images pro-$p$ closed subgroups of $A$.  The pro-$p$ HNN extension HNN (G. Higman, B. H. Neumann, H. Neumann) of $A$ with respect to $(\varphi_i,C)$, is a pro-$p$ group $G=HNN(A,\varphi_1,\varphi_2,C)=A*_{C}$, with an element $t\in G$ and a continuous homomorphism $f:A\rightarrow G$, with the following universal property: 

For any pro-$p$ group $K$, any $k\in K$ and any continuous homomorphism $\psi:A\rightarrow K$ satisfying $k\psi((\varphi_1(c)))k^{-1}=\psi(\varphi_2(c))$ for all $c\in C'$, there is a unique continuous homomorphism $\omega:G\rightarrow K$ such that the diagram
\[
\xymatrix{G\ar@{-->}[dr]^{\omega}\\
H\ar[u]^{f}\ar[r]^{\psi} & K
}
\]
 is commutative and $\omega(t)=k$.
\end{Def}

\begin{Rem}\label{Rem:amalgams}
    A few very important remarks:\begin{enumerate}
        \item Recall that for discrete groups with presentations $A=\langle \  S \ | \ R \ \rangle $, $B=\langle \ S' \ | \ R' \ \rangle $ one has that the discrete amalgamated free product $A*_{C}B$ has presentation \[\langle \ S, \ S' \ | \ R, \ R', \{\varphi_{1}(c)\varphi_{2}(c)^{-1} \ | \ \forall c\in C\} \ \rangle. \]
        Similarly for a discrete group $A$ with a presentation $A=\langle \ S \ | \ R \ \rangle$, the discrete HNN extension $A*_{C}$ has presentation
        \[\langle \ S, \ t \ | \ R, \ \{t\varphi_{1}(c)t^{-1}\varphi_{2}(c)^{-1} \ | \ \forall c\in C\} \rangle .\] The new generator $t$ is called the stable letter. 
        \item In the pro-p setting, the maps $f:A\rightarrow HNN(A,\varphi_1,\varphi_2,C)$, and $f_1:A\rightarrow A*_CB$, $f_2:B\rightarrow A*_CB$ need not be monomorphisms! 
        In the case that they are, we say that the pro-$p$ amalgamated product/HNN extension is \textbf{proper}.
        \item In general a pro-$p$ amalgamated product/HNN extension does not have a nice presentation as in the discrete case!
    \end{enumerate}
\end{Rem}

We have put some information on pushouts and homotopy pushouts of groupoids in the Appendix \ref{App:Cats} where the HNN extension can be defined more abstractly. 

In the case of surfaces and curves on them, we actually have that all such algebraic splittings come from topological ones: 
\begin{Thm} (\cite{zieschang2006surfaces} Theorem 4.12.1)
\label{Thm:sccsplit}Let $S$ be a connected surface,
then all splittings over $\mathbb{Z}$ of $\pi_{1}(S)$, into HNN
extensions or amalgamated free products, arise from cutting along a simple closed curve on the surface, and
we have a bijection between isotopy classes of simple closed curves on $S$,
and such splittings of $\pi_{1}(S)$ over $\mathbb{Z}$.
\end{Thm}

We have the following gluing theorem \cite{wilkes2019relative} of Wilkes, which shows that proper gluings/cutting of $PD^n$ groups results in $PD^n$ groups:
\begin{Thm}\label{Thm:amalLES}
Let $G_1,G_2$ be pro-$p$ groups, $\mathcal{S}_1,\mathcal{S}_2$ possibly empty
finite families of subgroups of $G_1$, $G_2$ respectively.

\begin{enumerate}
    \item For a pro-$p$ $\PD^{n-1}$ group $L\neq G_1,G_2$, with continuous monomorphisms $\varphi_i:L\rightarrow G_i$, suppose that  $G_1*_L G_2$ is a proper amalgamated free product, then we have that $(G_1*_L G_2, \mathcal{S}_1\coprod \mathcal{S}_2)$ is a $\PD^n$ pair if and only if $(G_1,\mathcal{S}_1\coprod\varphi_1(L))$ and $(G_2,\mathcal{S}_2\coprod\varphi_2(L))$ are both $\PD^n$.

    \item For pro-$p$ $\PD^{n-1}$ subgroups $L,L'$ of $G_1$, with an isomorphism $\psi:L\rightarrow L',$  and suppose that $G_1*_{L,\psi}$ is a proper pro-$p$ HNN extension, then we have that $(G_1*_{L,\psi}, \mathcal{S}_1)$ is a $\PD^n$ pair if and only if $(G_1,\mathcal{S}_1\coprod L\coprod L')$ is $\PD^n$.
\end{enumerate}
\end{Thm}

For $PD^2$ groups, these splittings were further studies in \cite{GropperArxiv}. We give a summarized version of the main results in \cite{GropperArxiv} which is needed for this paper:
\begin{Thm}\label{Thm:GropperSplit}
    Let $G$ be a pro-p $PD^2$ group (or a pro-p group in a $PD^2$ pair $(G,\mathcal{S})$). 
    Up to an automorphism of $G$, proper $\mathbb{Z}_p$ splitting of $G$, are in bijection with pro-p completion of a discrete $\mathbb{Z}$ splitting of certain explicit discrete groups.\\
    These discrete groups and splittings can be explicitly described in terms of generators and relations, and in particular, any mod $p$ compatible splitting of $G$ is proper, and any proper splitting can be made compatible mod $p$.
\end{Thm}
\begin{Rem}
    By mod $p$ compatible splitting/graph of pro-p relative $PD^2$ groups, we mean the following:
    Given a choice of standard form basis (in the sense of \cite{wilkes2020classification}) of $G_k, H_l$, we get a choice of generators $u_i, m_j$ and we have that $\phi(u_i)=(m_j)^{\alpha_i}$  for some $\alpha_i \in \mathbb{U}_{p}$ .
    We say such gluing is mod $p$ orientation compatible, if there are standard form basis of $G_k,H_l$, and $c\in U(\mathbb{Z}/p)$  such that $\alpha_i
    \equiv c$ mod $p$  for all $i$.

   One reason to think of this condition as natural is that all pro-$p$ $\PD^2$ pairs are oriented mod $p$ (that is their orientation character is always trivial mod $p$), and so we only allow gluings which are``orientation preserving" mod $p$.
\end{Rem}
\newpage
\section{Pro-p Topological Quantum Field Theories}\label{sect:MainDefs}
\subsection{Group Triples and pro-p TQFTs}\label{subsect:TQFT}
In this section with a definition for pro-$p$ TQFTs, which follows the spirit of the definition given in \cite{turaev2006unoriented}. We replace their ``unoriented cobordisms" with $PD^n$ group triples.
\begin{Def}\label{Def:GrpTriple}
    Define a pro-p $PD^n$ group triple to be a $(G,\mathcal{N}, \mathcal{U})$, where $G$ is a pro-p group, and $\mathcal{N}=\{N_i\}_{i\in I}, \mathcal{U}=\{U_j\}_{j\in J}$ are finite collections of pro-p groups (which we refer to as the "in boundary of $G$, and the "out boundary of $G$ respectively), such that one of the following occurs:
    \begin{enumerate}
        
    \item  $(G,\mathcal{S})$ is a $\PD^n$ pair where $\mathcal{S}:=\mathcal{N}\coprod\mathcal{U}$ (and so all $N_i$, $U_j$ are closed subgroups of $G$).
    
   \item $n>1$ and the tuple is $(\{1\}; \emptyset,\emptyset)$ this can be thought of as an $n$-sphere.
    \item     $n>2$ and the tuple is either $(\{1\}; \{1\},\emptyset)$ or $(\{1\}; \emptyset, \{1\})$ 
    which can be thought of as an $n$-ball, with an in (resp. out) boundary the $(n-1)$-sphere.
    \item $n=2$ and the tuple is $(\{1\}; \{\mathbb{Z}_{p}\},\emptyset)$, $(\{1\}; \emptyset, \{\mathbb{Z}_{p}\})$ which can be thought of as a disk, with in (resp. out) boundary a circle.
   \end{enumerate}
   An isomorphism of triples $(G,\mathcal{N},\mathcal{U})$, $(H,\mathcal{M},\mathcal{V})$ are a triple of isomorphisms $\phi:G\rightarrow H$, $\phi_{in}:\mathcal{N}\rightarrow\mathcal{M}$, $\phi_{out}:\mathcal{U}\rightarrow\mathcal{V}$, which are compatible with the maps from the boundaries to $G$.
   These maps are simply the inclusion map in case (1), and the map to the trivial group in the other cases.

   This is equivalent to saying that isomorphisms are only between group triples of the same type.
   For type (1), an isomorphism of triples $(G,\mathcal{N},\mathcal{U})$, $(H,\mathcal{M},\mathcal{V})$ is an isomorphism $\phi:G\rightarrow H$, such that it induces an isomorphism from each $N_i\in\mathcal{N}$ to some $M_k\in\mathcal{M}$, and from each $U_j\in\mathcal{U}$ to some $V_l\in\mathcal{V}$, in a way that gives a bijection from $\mathcal{N}$ to $\mathcal{M}$ and from $\mathcal{U}$ to $\mathcal{V}$.
   Types (2) and (3) only have the trivial isomorphism.
   And isomorphism of a type (4) triple $(\{1\}; \{\mathbb{Z}_{p}\},\emptyset)$ is simply one induced from an isomorphism of two infinite cyclic pro-$p$ groups, and similarly for a type (4) triple $(\{1\};\emptyset,\{\mathbb{Z}_{p}\})$ (but not between the two).
   \end{Def}

We have the following slight generalization of Theorem $\ref{Thm:amalLES}$, which shows that pro-p group triples glue nicely:
   
\begin{Prop}\label{Prop:Gluing}
    Given two collections of pro-p $PD^n$ group triples, \[\{(G_k; \mathcal{N}_{k},\mathcal{U}_{k})\}_{k \in K}, \{(H_l; \mathcal{M}_{l}, \mathcal{V}_{l})\}_{l \in L},\] denote by  $\mathcal{U}=\underset{k\in K}\coprod\mathcal{U}_k=\{U_i\}_{i\in I}$, $\mathcal{M}=\underset{l\in L}\coprod\mathcal{M}_l$.
    For an isomorphism between families: $\{\psi_i\}_{i\in I}=\psi:\mathcal{U}\rightarrow\mathcal{M}$, if the graph of pro-$p$ groups coming from gluing along the isomorphisms in the isomorphism of families $\psi$ is proper, then we get a new family of group triple \[\{(E_j; \mathcal{N}_{j},\mathcal{V}_{j})\}_{j \in J}:=\{(H_l; \mathcal{M}_{l}, \mathcal{V}_{l})\}_{l \in L}\circ_\psi \{(G_k; \mathcal{N}_{k}, \  \mathcal{U}_{k})\}_{k \in K}\] (where for some partition between $\mathcal{N}$ and $\mathcal{V}$). 

\end{Prop}

\begin{proof}
We will form a graph of groups, with vertices $\{G_k\}_{k \in K} \coprod \{H_l\}_{l \in L}$, and an edge between them for every $\psi_i\in\psi$. Let $J$ be an index set, of size equal to the number of connected components of this graph.\\
For each $j\in J$, we denote by $E_m$ the fundamental group of the $j$-th connected component. We get a triple $(E_j; \mathcal{N}_{j}, \mathcal{V}_{j})$ where $\mathcal{N}_{j}$ is the union of the $\mathcal{N}_{k}$ for all the vertices $G_k$ in the $j$-th connected component, and similarly for $\mathcal{V}_{j}$.
We claim that the triples $(E_j, \mathcal{N}_{j},\mathcal{V}_{j})$ are pro-p $PD^n$ triples for all $j$.

 We start by gluing triples in the family, which have only copies of the trivial group.
 By definition \ref{Def:GrpTriple}, we have that when $n=2$, the triple must be $(\{1\};\emptyset,\emptyset)$ and so there is no subgroup to glue, and we will just get that the resulting family of triples will have some component of the form $(\{1\};\emptyset,\emptyset)$.
 If $n>2$ then we either have the same situation as above or that the component is either $(\{1\};\emptyset,\{1\})$ or $(\{1\};\{1\} ,\emptyset)$.
 Since $\psi$ is an isomorphism we either have that $(\{1\};\emptyset,\{1\})$ glues to $(\{1\};\{1\} ,\emptyset)$ and the glued triple there will be again $(\{1\};\emptyset,\emptyset)$, or that we glue  $(\{1\};\{1\} ,\emptyset)$ to nothing resulting in $(\{1\};\{1\} ,\emptyset)$, or we glue nothing to $(\{1\};\emptyset,\{1\})$ resulting in $(\{1\};\emptyset,\{1\})$.\\

 We now deal with the case for gluing $(\{1\};\mathbb{Z}_p,\emptyset)$. 
 We will work by gluing in such caps one at a time. Gluing a $\PD^n$ pair $(G,\mathcal{S})$ to a pair $(\{1\},\mathbb{Z}_p)$ corresponds to taking a quotient by $S_i$ where $S_i\in \mathcal{S}$ is the subgroup which $\psi$ glues along. If $|\mathcal{S}|\geq 2$ and the normal subgroup generated by $S_i$ intersects all other $S_{i'}$ trivially (this is true if $|\mathcal{S}|>2$ or $|\mathcal{S}|=2$ and $S_{i'}\not\subseteq \langle \langle S_i \rangle \rangle$), we get that the quotient is again a $\PD^2$ pair, and if $|\mathcal{S}|=1$ then the quotient is a $\PD^2$ by the classification of Wilkes (Theorem \ref{Thm:WilkesPDn}). 
 We are left with the case of $(G;\{S_1\},\{S_2\})$.
 If $G$ is generated by more than one element, then again the quotient is a $\PD^2$ pair in the usual sense. Otherwise we get $G=\mathbb{Z}_p$ and the quotient will be $(\{1\},\mathbb{Z}_p)$, or $G=\{1\}$ and the quotient is $(\{1\},\emptyset$). So for any connected component with $G=\{1\}$ we got a connected cobordism.

We are now left with gluings such that all $G\neq\{1\}$.
In these cases, since all vertex groups in the connected component were non trivial, we get that $(E_j, \mathcal{N}_{j}\coprod\mathcal{V}_{j})$ is formed as a gluing in the sense of classical graphs of groups, and so by \cite{wilkes2019relative} Theorem 5.18 is a pro-p $\PD^n$ pair (by first collapsing all non loop edges who's edge group equals the vertex group to get a reduced graph of groups, as explained in \cite{wilkes2019relative} before Theorem 5.18). 

Overall we indeed get that the resulting family is a family of pro-p $PD^n$ triples. we get the new cobordism of the claim: \[\{(E_j; \mathcal{N}_{j}, \mathcal{V}_{j})\}_{j \in J}.\] 
\end{proof}
\begin{Rem}
    We will usually remove the index set and simply write the gluing as:
    \[(E_J; \mathcal{N}_{J},\mathcal{V}_{J})=(H_L; \mathcal{M}_{L}, \mathcal{V}_{L})\circ_\psi (G_K; \mathcal{N}_{K}, \  \mathcal{U}_{K})\]
\end{Rem}
Now that we know how to glue pro-p $PD^n$ group triples, we turn to define a TQFT on families of such triples:
\begin{Def} \label{Def:TQFT-TurnerTaraev}
    Given a commutative ring with unity $R$, we define a $(n+1)$ $R$-TQFT (over $p$) as the following assignment: 
    To any family of pro-$p$ $\PD^n$  groups $\mathcal{M}$ assigns a projective $R$–module of finite type $A_M$. For any family of  pro-$p$ $PD^{n+1}$ triples, $\{(G_k; \mathcal{N}_{k}, \mathcal{U}_{k})\}_{k \in K}$, an $R$-homomorphism \[\zeta((G_k)_{k\in K}):A_{\mathcal{N}}\rightarrow A_{\mathcal{U}}\]
    
    To any isomorphism of pro-$p$ $\PD^n$ groups $f:M\rightarrow M'$ we give, an $R$–isomorphism $f_{\#}:A_M\rightarrow A_{M'}$ such that the following seven conditions hold.
 \begin{enumerate}\label{TT}
    \item[T1.] 
    For the empty collection of $\PD^n$ groups $\emptyset$, we have $A_{\emptyset}=R$.
     \item[T2.] 
     For collections of pro-$p$ $\PD^n$ groups $\mathcal{N}$ and $\mathcal{M}$, we have an isomorphism \newline $A_{\mathcal{N}\coprod\mathcal{M}} \cong A_{\mathcal{N}}\otimes_RA_{\mathcal{M}}$ which is natural with respect to isomorphisms of collections of $\PD^n$ groups and compatible with the labelings.
     \item[T3.] 
     For any two isomorphisms of pro-$p$ $\PD^n$ groups $f:M\rightarrow M'$, $f':M'\rightarrow M''$, we have:
     $(f'\circ f)_{\#}=f_{\#}'\circ f_{\#}$.
      \item[T4.] 
      The homomorphism $\tau$ is natural with respect to isomorphisms of group triples.
In other words, given an isomorphism \[\psi:\{(G_k, \mathcal{N}_{k}, \mathcal{U}_{k})\}_{k \in K}\rightarrow\{(H_l; \mathcal{M}_{l}, \mathcal{V}_{l})\}_{l \in L}\] we have that the following diagram commutes:
\[
\xymatrix{A_{\mathcal{N}^{k}}\ar[d]_{(\psi|_{\mathcal{N}^{k}})_{\#}}\ar[r]^{\zeta(G_k)} & A_{\mathcal{U}^{k}}\ar[d]^{(\psi|_{\mathcal{U}^{k}})_{\#}}\\
A_{\mathcal{M}^{m}}\ar[r]_{\zeta(H_k)} & A_{\mathcal{V}^{m}}}
\]
      \item[T5.] 
      For a family of  pro-$p$ $PD^{n+1}$ triples $\{(G_k, \mathcal{N}_{k}, \mathcal{U}_{k})\}_{k \in K}$ we have $\zeta((G_k)_{k\in K})=\otimes_{k\in K} \zeta(G_k)$, under the isomorphism of condition T2.
      \item[T6.]  
      Given a pro-$p$ $\PD^n$ group $M$, we can also consider it as the triple $(M;\{M\},\{M\})$. For this a pro-p $PD^{n+1}$ triple, and we have that: \newline
      \[\zeta(M;\{M\},\{M\})=\id :A_M\rightarrow A_M\]
      \item[T7.]
       Suppose we have two families of pro-p $PD^{n+1}$ triples $\{(G_k; \mathcal{N}_{k}, \mathcal{U}_{k})\}_{k \in K}$, $\{(H_l; \mathcal{M}_{l}, \mathcal{V}_{l})\}_{l \in L}$,
    which we glue along an isomorphism $f:\mathcal{U} \to \mathcal{M}$ to get a new family of pro-p $PD^{n+1}$ triples  $\{(E_j; \mathcal{N}_{j}, \mathcal{V}_{j})\}_{j \in J}$. \newline 
    We have: \[\zeta((E_j)_{j\in J})=\zeta((H_l)_{l\in L}) \circ f_{\#} \circ \zeta((G_k)_{k\in K}):A_\mathcal{N}\rightarrow A_\mathcal{V}\] 
 \end{enumerate}  

 Two TQFTs $(A,\zeta)$ and $(A',\zeta')$ are isomorphic if for any pro-$p$ $\PD^n$ group $M$ we have an isomorphism $\eta_M:A_M\rightarrow A'_M$ such that:\begin{enumerate}
     \item $\eta_\emptyset=\id_R$
     \item $\eta_{M\coprod N}=\eta_M\otimes\eta_N$
     \item for any isomorhpism $f:M\rightarrow M'$ we have $\eta_{M'}\circ f_\#=f_\# \circ \eta_M$
     \item for any family of pro-p $PD^{n+1}$ triples
 $\{(G_k; \mathcal{N}_{k}, \mathcal{U}_{k})\}_{k \in K}$, we have \[\zeta'(G) \circ \eta_{\mathcal{N}_{k}}=\eta_{\mathcal{U}_{k}} \circ \zeta(G).\]
 \end{enumerate}
Let $\TQFT^{n}_{p}(R)$ be the category of TQFTs defined as above.
\end{Def} 
\begin{Rem}\label{Rem:TensorCompt}
    In axiom (T2) above we mean the following:
    We have that $\mathcal{N}\coprod\mathcal{M}\neq\mathcal{M}\coprod\mathcal{N}$ but rather there is a canonical swapping morphism, and similarly for $\otimes_R$, and so by compatible with the labeling we mean that one has the following:
    \begin{enumerate}
        \item For any two collections of pro-p $PD^n$ groups, the diagram:
        \[
\xymatrix{A_{\mathcal{N}\coprod\mathcal{M}}\ar[d]_{}\ar[r]^{\sim} & A_{\mathcal{N}}\otimes_R A_{\mathcal{M}}\ar[d]^{\textbf{Perm}}\\
A_{\mathcal{M}\coprod\mathcal{N}}\ar[r]^{\sim} & A_{\mathcal{M}}\otimes_R A_{\mathcal{N}}}
\]
        where Perm is the map $x\otimes_R y\mapsto y\otimes_R x$, is commutative.
        \item For any three pro-p $PD^n$ groups $M,N,U$ denote $\mathcal{M}=\{M\},\mathcal{N}=\{N\},\mathcal{U}=\{U\}$ one has that:
            \begin{equation*}
            \begin{split}
            (A_{\mathcal{M}}\otimes_R A_{\mathcal{N}})\otimes_R A_{\mathcal{U}} & \cong A_{\mathcal{M}}\otimes_R A_{\mathcal{N}\coprod\mathcal{U}}\cong A_{\mathcal{M}\coprod\mathcal{N}\coprod\mathcal{U}} \\ & \cong A_{\mathcal{M}\coprod\mathcal{N}}\otimes_R A_{\mathcal{U}}\cong A_{\mathcal{M}}\otimes_R (A_{\mathcal{N}}\otimes_R A_{\mathcal{U}})
            \end{split}
            \end{equation*}
        is the identification $(m\otimes n)\otimes u=m\otimes(n\otimes u)$.
        \item For any two isomorphisms between families of pro-p $PD^n$ groups, $f:\mathcal{N}\rightarrow\mathcal{N}'$, $g:\mathcal{M}\rightarrow\mathcal{M}'$ the diagram:
        \[
\xymatrix{A_{\mathcal{N}\coprod\mathcal{M}}\ar[d]_{(f\coprod g)_\#}\ar[r]^{\sim} & A_{\mathcal{N}}\otimes_R A_{\mathcal{M}}\ar[d]^{f_\#\otimes_R g_\#}\\
A_{\mathcal{N}'\coprod\mathcal{M}'}\ar[r]^{\sim} & A_{\mathcal{M}'}\otimes_R A_{\mathcal{N}'}}
\]
        is commutative.
    \end{enumerate}
\end{Rem}
\begin{Rem}
Later we will explain how in the $(1+1)$ dimensional case this is equivalent to studying symmetric monoidal functors from the category of cobordisms in the sense of cospans of groups.    
\end{Rem}

From now on we will focus our study on $(1+1)$-dimensional TQFTs and group triples.

We finish this section with the following classification of group triples for the $(1+1)$ case.
\begin{Lem} \label{Lem:invariants}
To each triple of pro-p $PD^2$ groups, we can associate an invariant $(g,r,n,u)$, where $g,n,u\in\mathbb{Z}_{\geq 0}$,  $r\in\mathbb{Z}_{> 0}\cup\{\infty\}$ if $g\neq 0$, and $r=\infty$ when $g=0$.\\
    If two cobordisms have the same invariants $(g,r,n,u)$ then they are isomorphic.
    Conversely for any tuple $(g,r,n,u)$, with $g,n,u\in\mathbb{Z}_{\geq 0}$,  $r\in\mathbb{Z}_{> 0}\cup\{\infty\}$ if $g\neq 0$, and $r=\infty$ when $g=0$, there is a cobordism with such an invariant. 
\end{Lem}
\begin{proof}
Two $PD^2$ triples of pro-p groups $(G;\mathcal{N},\mathcal{U})$, $(G';\mathcal{N}',\mathcal{U}')$ are isomorphic, if and only if there is an isomorphism of the groups $G$ to $G'$ sending isomorphically each ``in" boundary in $\mathcal{N}$ to (a conjugate of) a unique ``in" boundary in $\mathcal{N}'$ (inducing a bijection between the sets) and similarly for the ``out" boundaries. 
In other words, we get that $|\mathcal{N}|=n$ and $|\mathcal{U}|=u$ are invariants of an isomorphism class of a triple.\newline
Let $\mathcal{S}=\mathcal{N}\coprod \mathcal{U}$. By Wilkes' classification of pro-$p$ $\PD^2$ pairs (Theorem \ref{Thm:WilkesPDn}), we know that if $(G,\mathcal{S})$ is a $
\PD^2$ pair of pro-$p$ groups then for some positive integer $g$ we have \[\rk(G)=2g+|\mathcal{S}|-1.\]
We also get that there exists some $\chi$, the orientation character  of $(G,\mathcal{S})$,  and some integer $q(\chi)=p^r$, which is the maximal power of $p$ such that $\chi$ is the identity mod $p^r$. We call \[g:=\frac{1}{2}(\rk(G)-(|\mathcal{S}|-1))\] the genus of the pair, and we say that the pair is orientable up to level $r$ where $r$ is the power of $p$ in $q(\chi)$: $r=\log_p q(\chi)$.

It is easy to see by \cite{wilkes2020classification} Theorem 3.3, that any two cobordisms with the same invariant $(g,r,n,u)$ are isomorphic.
Suppose that we are given a tuple $(g,r,n,u)$ with $g,n,u\in\mathbb{N}_{\geq 0}$,  $r\in\mathbb{N}_{\geq 0}\cup\{\infty\}$ if $g\neq 0$, and $r=\infty$ when $g=0$. The cobordism whose invariants are this tuple is given by defining $G$ to be the free pro-$p$ group on the $2g+b$ generators $\{x_1,y_1,\dots,x_g,y_g,s_1,\dots,s_{b}\}$ where \[b:=n+u-1,\] and taking
\[\mathcal{N}=\{\langle s_0 \rangle,\dots,\langle s_{n-1} \rangle \}\] \[\mathcal{U}=\{\langle s_n \rangle,\dots,\langle s_{b} \rangle \},\]  where \[s_0=x_1^{p^r}[x_1,y_1]\cdots [x_g,y_g]s_1\cdots s_{b}.\]
\end{proof}

\subsection{Category of pro-p cobordisms}\label{subsect:Cob}
In order to be able to look at our TQFTs from a more functorial point of view, we now turn to describe a framework for cobordisms in terms of cospans of groups (we later show this is equivalent to a cospans of groupoid framework).
\begin{Rem}
We restrict ourselves here to the $(1+1)$ dimensional case. In order to have the categorical perspective for higher dimensions, one would require a better understanding of which gluings are proper (in the sense of remark \ref{Rem:amalgams}), e.g. to have a nice condition such as mod $p$ compatible, which assures all gluings with that property are proper.
\end{Rem}
\begin{Def}\label{Def:groupcobordism}
    We define a non-empty connected pro-$p$ $2$-cobordism to be a tuple \[(G,\mathcal{N},\mathcal{U}, \phi,\tau)\] such that:
    \begin{itemize}
   \item  $G$ is a pro-$p$ group 
   \item  $\mathcal{N}= \{N_i\}_{i\in I}$ and $\mathcal{U}=\{ U_j \}_{j\in J}$ are (possibly empty) collections of $\PD^{1}$ pro-$p$ groups and copies of the trivial group, with fixed generators $n_i$, $u_j$ mod $p$.
   \item $\phi$ and $\tau$ are two collections of group homomorphisms:
   \[ \xymatrix@C=.57em{
  & {G}  \\
 \{N_i\}_{i\in I} \ar[ur]^{\{\phi_i\}_{i\in I}}
 & &\{U_j\}_{j\in J} \ar[ul]_{\{\tau_j\}_{j\in J}}
 }\]
 such that the following two hold
 \begin{enumerate}
     \item The images in $G$ form a $PD^2$ triple, $(G,\{\phi_i(N_i)\}_{i\in I},\{\tau_j(U_j)\}_{j\in J}$
     \item Let $\{S_k\}_{k\in K}=\mathcal{S}=\{\phi_i(N_i)\}_{i\in I}\coprod\{\tau_j(U_j)\}_{j\in J}$. The $PD^2$ pair $(G,\mathcal{S})$ has a standard basis with $<s_k>=S_k$, such that mod $p$ the maps $\phi_i$, and $\tau_j$ send $n_i$, $u_j$ to the elements $s_k$.
 \end{enumerate}
\end{itemize}
\end{Def}

\begin{Def}
  An isomorphism between two connected $n$-cobordisms $(G,\mathcal{N},\mathcal{U}, \phi,\tau)$ and  $(G',\mathcal{N},\mathcal{U}, \phi',\tau')$
   is given by an  isomorphism of groups $\psi:G\rightarrow G'$ such that there are elements $g'_i,g'_j\in G'$ for which \[
   \begin{split}
    \psi(\phi_i)=g'_i(\phi'_i)(g'_i)^{-1}\\
       \psi(\tau_j)=g'_j(\tau'_j)(g'_j)^{-1}
   \end{split}
   \] 
   In other words the isomorphism $\psi:G\rightarrow G'$ fits into a commutative diagram of the form:\[ \xymatrix@C=.57em{
  & {G}\ar[dd]^\psi  \\
 \{N_i\}_{i\in I} \ar[ur]^{\{\phi_i\}_{i\in I}}\ar[dr]_{\{(\phi'_i)^{g'_i}\}_{i\in I}}
 & &\{U_j\}_{j\in J} \ar[ul]_{\{\phi_j\}_{j\in J}}\ar[dl]^{\{(\phi'_j)^{g'_j}\}_{j\in J}}\\
  & {G'}}\]
\end{Def}
\begin{Rem}
    From the point of view of cospans, we only allow an isomorphism between cobordisms having the same boundary groups. 
    If one instead wants to relate this to the notion of isomorphisms of group pairs, (i.e. an isomorphism between $(G,\mathcal{S})$, $(G',\mathcal{S}')$) we decompose this isomorphism into three cospans.
    Namely, the one having $\mathcal{S}$ as the boundary group, followed by a mapping cylinder between $\mathcal{S}$ to $\mathcal{S}'$ and then the one having $\mathcal{S}'$ as its boundary:
    \[\xymatrix@C=.57em{
  {G} & & {\mathcal{S}'}  && {G'}\\
  &{\mathcal{S}} \ar[ul]^{\id}\ar[ur]^{\psi}&&{\mathcal{S}'} \ar[ul]_{\id}\ar[ur]_{\id}\\
  }\] 
\end{Rem}
\begin{Def}
   We define a general $2$-cobordism to be a (finite, possibly empty) collection of $2$-cobordisms, $\{(G_k; \mathcal{N}_{k}, \mathcal{U}_{k},\phi_k,\tau_k)\}_{k \in K}$ and we define an isomorphism between non-connected $n$-cobordism to be a collection of isomorphisms of connected $n$-cobordisms: \[\psi_k:(G_k; \mathcal{N}_{k}, \mathcal{U}_{k})\rightarrow (G'_{\eta(k)}; \mathcal{N}'_{\eta(k)}, \mathcal{U}'_{\eta(k)}).\]
   which induces a bijection $\eta:K \to K'$.\\
 Given two collections of pro-$p$ $\PD^{1}$ groups $\mathcal{N}=\{N_i\}_{i\in I}$, $\mathcal{U}=\{U_j\}_{j\in J}$, a cobordism between them is a collection of connected cobordisms \[\{(G_k,\mathcal{N}_k,\mathcal{U}_k, \phi_k,\tau_k)\}_{k\in K}\] is an $n$-cobordism, such that $\mathcal{N}=\coprod_{k\in K}\mathcal{N}_k$, $\mathcal{U}=\coprod_{k\in K}\mathcal{U}_k$.
 In other words, some  partitions $I=\coprod_{k\in K} I_k$, $J=\coprod_{k\in K} J_k$ and connected cobordisms between from each $\mathcal{N}_k:=\{N_{i_k}\}_{i_k\in I_k}$ to $\mathcal{U}_k:=\{U_{i_k}\}_{i_k\in J_k}$.
   We call $\mathcal{N}$ the in boundary, $\mathcal{U}$ the out boundary, and say that the collections are cobordant to each other.  
\end{Def}

\begin{Rem}
    As explained in \cite{wilkes2019relative} Proposition 2.9, cohomology of a group $G$, relative to a family of subgroups $\mathcal{S}$, (and the notion of the pair having a relative Poincar\'{e} duality) is `invariant' under conjugation of the subgroups in $\mathcal{S}$ by elements of $G$. That is if one conjugates each $S_i\in\mathcal{S}$ by some element $g_i\in G$, the resulting relative cohomologies will be isomorphic. 
\end{Rem}

\begin{Exm}
    Any homomorphism $\psi:G \to G$ has a mapping cylinder cobordism associated to it: $(G,\{\{G\},\{G\}\},\psi,\id)$.
\end{Exm}

\begin{Rem}
    One can also consider a triple  \[(G,\mathcal{S},\zeta_\mathcal{S})\] where $G$ is a pro-$p$ group, $\mathcal{S}$ is a collection of pro-$p$ groups and $\zeta_{\mathcal{S}}$ a collection of homomorphisms from the groups in $\mathcal{S}$ to $G$ as non-directed cobordism, if there is a cobordism $(G,\mathcal{N}_I,\mathcal{U}_J, \phi_I,\tau_J)$ such $\mathcal{S}=\mathcal{N}_I\coprod\mathcal{U}_J$ and $\phi_I,\tau_J$ are the maps induced from $\zeta_{\mathcal{S}}$
\end{Rem}

\begin{Prop}\label{Prop:cobproperties}
Any connected cobordism has:
\begin{enumerate}
    \item $|I|, |J|< \infty$.
    \item Except in the case of $G=\{1\}$, all maps $\phi_i$ and $\tau_j$ are injective.
     \item none of the groups in $\mathcal{N}_I\coprod\mathcal{U}_J$ are trivial.
\end{enumerate}
Every connected cobordism is equivalent (in a non canonical way) to a $\PD^2$ pair $(G,\mathcal{S})$ together with a family of automorphisms $\psi$ of $\mathcal{S}$ and a partition of $\mathcal{S}$ into two disjoint families. 
\end{Prop}

\begin{proof}

    If the images in $G$ of $N_i,U_j$ form an exceptional $PD^2$ triple (i.e. cases (2) or (4) in Definition \ref{Def:GrpTriple}) the proposition is obvious.\\
    Otherwise, given a connected cobordism $(G,\mathcal{N}_I,\mathcal{U}_J, \phi_I,\tau_J)$ we by definition that have that $(G,\{\phi_i(N_i)\}\coprod\{\tau_j(U_j)\})$ is a $\PD^2$ pair at p.
    First by \cite{wilkes2019relative} Proposition 5.4 we get that $|I|,|J|< \infty$.
    Since all groups are pro-$p$, we have by \cite{wilkes2019relative} Proposition 1.32 and Proposition 5.9, that each $\phi_i(N_i)$ and $\tau_j(U_j)$ are $\PD^{1}$ at p.
    So we get that none of the $N_i,U_j$ are not the trivial group, and that $\phi_i$, $\tau_j$ are surjective maps between pro-$p$ $\PD^{1}$ groups.
    Thus by \cite{neukirch2013cohomology}  Theorem 3.7.4 we get that $\cd_p(\ker(\phi_i))=0$ but since these are pro-$p$ groups, we have that $\ker(\phi_i)=\{1\}$.

     For the final statement, since all the maps $\phi_i$ and $\tau_j$ are isomorphisms, if we fix some identification of each isomorphism classes between pro-$p$ $\PD^{1}$ groups, then the corresponding maps from $N_i, U_j$ to subgroups of $(G,\mathcal{S})$ will be equivalent to the cobordism with identity automorphisms, and then any other cobordism between them will be given by families of automorphisms from each subgroup of $\mathcal{S}$ to itself.
\end{proof}
 \begin{Def}\label{Def:CobCat}
    Let $\Cob_p^2$ be the category whose objects are finite collections of the trivial group and pro-$p$ $\PD^{1}$ groups.
    Morphisms between objects are equivalence classes of cobordisms between the collections of groups, and compositions of morphisms is defined by gluings as in Proposition \ref{Prop:Gluing}.
\end{Def}
\begin{Rem}\label{Rem:gluinginstages}
One can instead think of the gluing in Proposition \ref{Prop:Gluing} as push-outs in the category of pro-p groups. Thinking of it like this gives us that the above category $\Cob_p^2$ is equivalent to a full subcategory of the category $\Cospan(\pGrpd)$ of cospans of pro-$p$ groupoids, which is further explained in Appendix \ref{App:Cospan}.
\end{Rem}

\begin{Prop}
    The TQFTs in definition \ref{Def:TQFT-TurnerTaraev} are equivalent to symmetric monoidal functors from $\Cob_p^2$ to $\Proj(R)$. More precisely, the natural functor 
    \[\Fun^{\otimes}(\Cob_p^2, \Proj(R)) \to \TQFT^{2}_{p}(R)\] is an equivalence of groupoids.
\end{Prop}

\begin{proof}
Given a symmetric monoidal functor from $\Cob_p^2$ to $R$-modules, by the discussion in Appendix \ref{App:Cospan} and Remark \ref{Rem:gluinginstages}, we get a TQFT.
On the other hand, by Proposition \ref{Prop:cobproperties} we can think of any cobordism as one where the boundaries are subgroups of $G$ together with a family of automorphisms of the boundaries.
    Now by Remark \ref{Rem:gluinginstages} we get that any TQFT will indeed give us a symmetric monoidal functor.
\end{proof}
\begin{Rem}
    A reformulation a more functorial formalism of the TQFTs of \cite{turaev2006unoriented} in the setting of manifolds can be found in \cite{sweet2013equivariant}. 
\end{Rem}

\subsection{Extended Frobenius algebras}\label{subsect:Frob}
We would now like to have a classification for $(1+1)$ pro-p TQFTs, in terms of some extended notion of Frobenius algebras, which we call $\mathbb{U}_{p}$-extended Frobenius algebras.
In this article we only consider commutative and co-commutative Frobenius algebras. 
Recall the following definition of a Frobenius algebra:
\begin{Def}\label{Def:Frob}
    A Frobenius algebra over a commutative ring $R$, is a projective $R$-module $V$, together with module homomorphisms
    $\iota:R\rightarrow V$, $\epsilon:V\rightarrow R$, $m:V\otimes_{R} V\rightarrow V$ and $\Delta:V\rightarrow V\otimes_{R} V$, called the unit, co-unit, multiplication and co-multiplication respectively, and which satisfy the following:\begin{enumerate}
    \item[F1.]  \[m \circ (\id_V\otimes \iota)=\id_V=m \circ (\iota \otimes \id_V)\]
    \item[F2.] \[(\epsilon\otimes \id_V) \circ \Delta=\id_V=(\id_V \otimes\epsilon) \circ \Delta\]
    \item[F3.]  \[m \circ (m\otimes \id_V)=m \circ (\id_V\otimes m)\]
    \item[F4.]    \[(\Delta\otimes \id_V) \circ \Delta=(\id_V\otimes\Delta) \circ \Delta\]
    \item[F5.]   \[(m \otimes \id_V) \circ (\id_V\otimes \Delta)= \Delta \circ m=(\id_V\otimes m) \circ (\Delta \otimes \id_V)\]
    \item[FS.] 
  \[ m \circ \sigma = m\]
  \[\sigma \circ \Delta = \Delta .\]
       \end{enumerate}
       Where \[\sigma:V \otimes_{R} V \to V \otimes_{R}V\] is the morphism determined by $v\otimes w \mapsto w\otimes v$.
\end{Def}
In the spirit of \cite{turaev2006unoriented}, but with the extra orientation levels considered we define the following:
\begin{Def}\label{Def:ExFrAlg}
    A $\mathbb{U}_{p}-$extended Frobenius algebra over a ring $R$ is a Frobenius algebra $(V, \iota, m, \epsilon, \Delta)$ over a ring $R$ together with a family of automorphisms $\phi_\alpha:V\rightarrow V$ for all $\alpha \in \mathbb{U}_{p}$ which are compatible with respect to multiplication in $\mathbb{U}_{p}$ (so a homomorphism $\rho: \mathbb{U}_{p} \to \Aut_{R}(V)$) 
    such that the following axioms are met:
    \begin{enumerate}
    \item[F6.] 
    \[\phi_\alpha \circ \iota = \iota \]
    \item[F7.] 
    \[
    \epsilon \circ \phi_\alpha = \epsilon
    \]
    \item[F8.] 
    \[
    \Delta \circ \phi_\alpha = (\phi_{\alpha} \otimes \phi_{\alpha}) \circ \Delta
    \]
    \item[F9.] 
    \[
    \phi_{\alpha} \circ m = m \circ (\phi_\alpha \otimes \phi_\alpha)
    \]
    
    \item[F10.] 
        For any $r\in\mathbb{Z}_{> 0}$, the maps from $R \to V$ 
        \[m \circ (\phi_\alpha \otimes \id_V) \circ \Delta \circ \iota\] agree for any $\alpha$ of level $r$.
        In other words, for any unit $\alpha$ with $\alpha\in1+p^r\mathbb{Z}_p$, but $\alpha\not\in1+p^{r+1}\mathbb{Z}_p$.
        We denote this map by $\kappa_r$.
    \item[F11.] As maps $R\to V$ we have 
      \[m \circ (\kappa_r \otimes \kappa_{r'})= m \circ \Delta \circ \kappa_{\min(r,r')}.\]

    \item[F12.] For $\alpha\equiv 1\pmod{p^{r}}$ we have \[\phi_\alpha \circ m\circ (\kappa_r\otimes \id_V) =m\circ (\kappa_r\otimes \id_V) .\]
    \end{enumerate}
    Let $V$, $V'$ be Frobenius algebras over $R$, with $\mathbb{U}_{p}$ structures: $\rho:\mathbb{U}_{p}\rightarrow \Aut_{R}(V)$ and $\rho':\mathbb{U}_{p}\rightarrow \Aut_{R}(V')$.
    An isomorphism of  extended Frobenius algebras is an isomorphism of Frobenius algebras $\psi:V\rightarrow V'$ such that \[\psi(\rho(\mu)(v))=\rho'(\mu)(\psi(v))\] for all $\mu\in\mathbb{U}_{p},v\in V$.
The groupoid of $\mathbb{U}_{p}$-extended Frobenius algebras will be denoted by $\Frob_{p}(R)$.
\end{Def}
\begin{Exm} The following example is based on Khovanov's universal (commutative co-commutative) rank $2$ aspherical Frobenius algebra \cite{MR1740682, MR2232858} and its extension to a universal $U(\mathbb{Z})$-extended (commutative co-commutative) rank $2$ aspherical Frobenius algebra by Turner and Turaev \cite{turaev2006unoriented}. We give a two-parameter family here, not doing the most general case.
For a unit $\alpha\in\mathbb{U}_p$, we call the level of $\alpha$, the largest integer $r$, such that $\alpha\in1+p^r\mathbb{Z}_p$
Let \[R=\mathbb{Z}[h,t] \mathbb{U}_{p}/I\] where 
    $I$ is the ideal of the group ring $\mathbb{Z}[h,t] \mathbb{U}_{p}$ generated by the sets \[
    \{[\alpha][\alpha']-h([\alpha]+[\alpha']-[\min \{\alpha, \alpha'\}]) \ | \ \alpha, \alpha' \in \mathbb{U}_{p}\} \]
    \[\{[\alpha_1]-[\alpha_2] \ | \ \alpha_1, \alpha_2 \in \mathbb{U}_{p} \  \text{have the same level r}\} \]\[ \{ 2[\alpha] \ | \ \alpha \in \mathbb{U}_{p} \}.
    \]
     Here $\min \{\alpha, \alpha'\}$ is some element of $\mathbb{U}_{p}$ whose $r$ level is the minimum of the $r$-levels of $\alpha$ and $\alpha'$. Notice that by taking $\alpha = \alpha'$ we see that $[\alpha]^2-h[\alpha]=0$ in $R$ for any $\alpha \in \mathbb{U}_{p}$. Let $V=R\oplus Rx$. 
    Define $\iota(1)=1,$ and define
    \[
    m(A+Bx, C+Dx)=AC+BDt+(AD+BC+BDh)x
    \]

    \[
    \Delta(1)=1\otimes x+ x\otimes 1-h 1\otimes 1
    \]
    \[
    \Delta(x)=x\otimes x+t1\otimes 1
    \]
    \[
    \epsilon(A+Bx)=B
    \]
    \[
    \phi_{\alpha}(A+Bx)=A+B[\alpha]+Bx.
    \]
    Notice $m(\Delta(x))= x^2+t= 2t+hx$ and $m(\Delta(1))=2x-h$. Then we have 
   \begin{equation*}
   \begin{split}
    & m(\phi_{\alpha}(A+Bx), \phi_{\alpha}(C+Dx)) \\
    & =(A+B[\alpha])(C+D[\alpha])+BDt+((A+B[\alpha])D+B(C+D[\alpha])+BDh)x  \\  
    \end{split}
   \end{equation*} 
whereas 
\begin{equation*}
   \begin{split}
    & \phi_{\alpha}(m(A+Bx, C+Dx)) \\
    &=AC+BDt +(AD+BC+BDh)[\alpha]+(AD+BC+BDh)x.
     \end{split}
   \end{equation*}
Therefore 
\begin{equation*}
\begin{split}
& m(\phi_{\alpha}(A+Bx), \phi_{\alpha}(C+Dx))-\phi_{\alpha}(m(A+Bx, C+Dx)) \\ &=BD([\alpha]^2+2[\alpha]x-h[\alpha])
\end{split}
\end{equation*}
and this vanishes using the relations in the ideal.
Now notice that 
\[
(\phi_\alpha \otimes \id) (\Delta(1)) = 1\otimes x-h1\otimes 1+([\alpha]+x)\otimes 1
\]
and so
\[
m(\phi_\alpha \otimes \id (\Delta(1)))=2x-h+[\alpha]
\]
depends only on the $r$-level of $\alpha$ using the second type of relation in the ideal. So let \[\kappa_r = 2x-h+[\alpha]\] where $\alpha$ has is some (any) unit of level $r$. Then notice \begin{equation*}
\begin{split}
\kappa_r \kappa_{r'} & = 4(t+hx)+[\alpha][\alpha']+h^2+(2x-h)([\alpha]+[\alpha'])-4xh \\
& = 4t+[\alpha][\alpha']+h^2 -h([\alpha]+[\alpha'])
\end{split}
\end{equation*}
We have 
\begin{equation*}
\begin{split}  
m \circ \Delta \circ \kappa_{\min(r,r')} & = m(2\Delta(x)+(-h+[\min \{ \alpha, \alpha'\}])\Delta(1)) \\
& = 2hx+4t+(-h+[\min \{ \alpha, \alpha'\}])(2x-h) \\
& = 4t+h^2 -h[ \min \{ \alpha, \alpha'\}]
.
\end{split}
\end{equation*}
Therefore
\begin{equation*}
\begin{split}
\kappa_r \kappa_{r'}-m \circ \Delta \circ \kappa_{\min(r,r')} &= [\alpha][\alpha'] +h([\min \{ \alpha, \alpha'\} ]-[\alpha]-[\alpha']) \\
&= 0
\end{split}
\end{equation*}

Finally, for $\alpha\equiv 1\pmod{p^{r}}$
\[
\phi_\alpha(\kappa_r)=[\alpha]-h+2x+2[\alpha] = \kappa_r
\]
and
\[
\phi_\alpha(\kappa_r x)=\phi_\alpha (([\alpha] -h)x+2t)=2t+([\alpha] -h)x = \kappa_r x.
\]
This concludes the demonstration that $V$ is a $\mathbb{U}_p$-extended Frobenius algebra. It is easy to see that with respect to the ordered basis $1,x$, $m\circ \Delta$ is given by the matrix 
\begin{equation*}
m\circ \Delta=\begin{pmatrix}
-h & 2t \\
2 & h 
\end{pmatrix}
\end{equation*}
and so $(m\circ \Delta)^2=\begin{pmatrix}
h^2+4t & 0 \\
0 & h^2+4t 
\end{pmatrix}$ and so $(m\circ \Delta)^3=(h^2+4t)\begin{pmatrix}
-h & 2t \\
2 & h 
\end{pmatrix}$ which means that the genus $3$ invariant is given by $\epsilon\circ (m\circ \Delta)^3 \circ \iota=2h^2+8t$. Meanwhile, considering the involutions $\phi_\alpha=\begin{pmatrix}
1 & [\alpha] \\
0 & 1 
\end{pmatrix}$ we can see that the extended structure and orientability level does not affect the surface invariants here. Nevertheless, these examples are not equivalent to the example given by the same underlying Frobenius algebra with all $\phi_\alpha = \id$. A more general family of $\mathbb{U}_p$-extended Frobenius algebra can be achieved by introducing two new scaling parameters for $\Delta$ and $\epsilon$.
\end{Exm}
Our main theorem which will be proven later in the article will be:
\begin{Thm} \label{Thm:Main}
    Isomorphism classes of $(1+1)-R-$TQFTs at $p$ which are mod $p$ orientation compatible are in a bijective correspondence with isomorphism classes of $\mathbb{U}_{p}$-extended $R$-Frobenius algebras. More precisely, there is an equivalence of groupoids $\TQFT^{2}_{p}(R) \cong \Frob_{p}(R)$.
\end{Thm}

We finish this section with some consequences of the axioms of $\mathbb{U}_p$-extended Frobenius algebras, these identities will be useful in the next section. 

First note that for any $f \in V$,
\begin{equation}\label{eqn:FrobCons1}
\Delta(f)=\Delta(m(1_m,f))=(\id \otimes m)(\Delta(1_m) \otimes f)
\end{equation}
so $\Delta$ is determined by $\Delta(1_m)$.
When we will want to study values of TQFTs on $T_{1,1}$ we need to study $m\circ\Delta$.
By the above equation we have that for any $f \in V$: 
\begin{equation}\label{eqn:FrobCons2}
m(\Delta(f))=m((\id \otimes m)(\Delta(1_m) \otimes f))= m( m(\Delta(1_m)) \otimes f).
\end{equation}
Finally, we have:
\begin{Prop}
    Let $r$ is the highest power of $p$ such that $\alpha \beta^{-1}$ is in $1+p^r\mathbb{Z}_p$. Then for a fixed such $r$, the maps $R \to V$  
        \[m \circ (\phi_\alpha \otimes \phi_\beta) \circ \Delta \circ \iota\]
        agree and are equal to $\kappa_r$.
\end{Prop}
\begin{proof}
    By axiom (F8) we have  \[m \circ (\phi_\alpha \otimes \phi_\beta) \circ \Delta \circ \iota= m \circ (\phi_{\alpha \beta^{-1}} \otimes \phi_{\beta \beta^{-1}}) \circ \Delta \circ\phi_\beta \circ \iota\]
    but then by axiom (F6) we get that this is
\[ m \circ (\phi_{\alpha \beta^{-1}} \otimes \text{id}_V) \circ \Delta \circ \iota\]
which is just $\kappa_r$.
\end{proof}
\newpage
\section{Structure of \texorpdfstring{$Cob_p^2$}{Cob\_p\^2}}\label{sect:Structure}
 To set the stage to prove Theorem \ref{Thm:Main}, we start by first better understanding the objects of the two dimensional cobordism category $Cob_p^2$.
 \subsection{Skeleton and symmetric monoidal structure}
 Objects in $\Cob_p^2$ are finite labeled collections of pro-$p$ $\PD^1$ groups and pro-$p$ $\PD^1$ groups are all isomorphic to $\mathbb{Z}_p$.
 So if for each non negative integer $n$ we fix a canonical ordered collection $\{(\mathbb{Z}_p)_i\}_{i \in \{1,..,n\}}$, we have that any object is isomorphic to a collection $\{(\mathbb{Z}_p)_i\}_{i \in \{1,..,n\}}$  for some nonnegative integer $n$.
 The isomorphism is not unique, but for any isomorphisms $a,b$ from $\mathcal{U}_{k\in K}$ to $\{(\mathbb{Z}_p)_i\}_{i \in \{1,..,n\}}$ there is an automorphism $c$ of $\{(\mathbb{Z}_p)_i\}_{i \in \{1,..,n\}}$ such that $c\circ a=b$.
From this, we get that we can instead consider a skeleton of our category, of objects of the form $\{(\mathbb{Z}_p)_i\}_{i \in\{1,..,n\}}$, where the monoidal structure will be by the identification from $\{(\mathbb{Z}_p)_i\}_{i \in\{1,..,n\}}\coprod\{(\mathbb{Z}_p)_j\}_{j \in\{1,..,m\}}$ to $\{(\mathbb{Z}_p)_k\}_{k \in\{1,..,n+m\}}$ sending by the identity automorphism of $\mathbb{Z}_p$ to the first $n$ components (by their order) the copies labeled by $i$, and then from the $n+1$ to the $n+m$ (by their order) the copies labeled by $j$.
The symmetric structure is given by swapping two labeled objects, i.e. the map taking $\mathcal{M}\coprod\mathcal{N}$ to $\mathcal{N}\coprod\mathcal{M}$ which allows us to switch between the copies of $\mathbb{Z}_p$.
We would now like to understand the generators and relations of the monoidal category $Cob^2_p$ (in the sense of section 1.4 of \cite{kock2004frobenius}).

\subsection{Generators}\label{subsect:Gens}

We claim that any cobordism in $Cob^2_p$ can be constructed from certain group triples.\\
These correspond to ``pairs of pants", ``cups", ``caps", ``cylinders" and ``a twice punctured torus, orientable up to level $r$" for each $r$.
By these we mean:
\begin{enumerate}
    \item[G1.]
A pair of pants is the cobordism \[P_{2,1}:=(\langle a,b \rangle; \{\langle a \rangle,\langle b \rangle \},\{\langle ab \rangle\})\] 
or
\[P_{1,2}:=(\langle a,b \rangle; \{\langle ab \rangle\},\{\langle a \rangle,\langle b \rangle\}).\] In other words, the group is the free pro-$p$ group on two generators, and we have either two in and one out boundary components, or one in and two out boundary components.
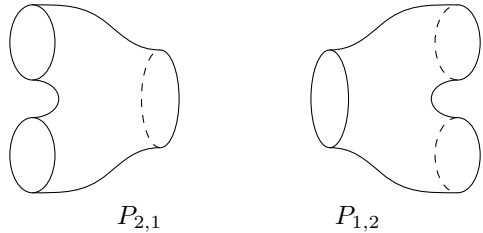
\begin{center} 
\begin{tikzpicture}
\draw (0,0) node{\PantsL};
\draw (0.6,-1.6) node {$P_{2,1}$};
\draw (4,0) node{\PantsR};
\draw (3.5,-1.6) node {$P_{1,2}$};
\end{tikzpicture} 
\end{center}

\item[G2.] A cup is \[C_{1,0}:=(\{1\}; \{\mathbb{Z}_p\},\emptyset)\] \newline
\item[G3.] A cap is \[C_{0,1}:=(\{1\}; \emptyset, \{\mathbb{Z}_p\})\]
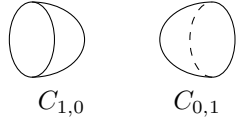
\begin{center} 
\begin{tikzpicture}

\draw (0,0.75) node{\capB};
\draw (0.2,-0.1) node {$C_{1,0}$};
\draw (2,0.75) node{\cupB};
\draw (2,-0.1) node {$C_{0,1}$};

\end{tikzpicture} 
\end{center}

\item[G4.] For all $r\in\mathbb{N}_{>0}\cup\infty$, a twice punctured torus, orientable up to level $r$,  \[T^r_{1,1}:=(\langle x,y,s_1 \rangle ; \{\langle s_1 \rangle\},\{\langle s_0 \rangle\})\]  where \[s_0=s_1x^{p^r}[x,y].\] This is equivalent to having a basis where $s_0$ is in the generators and $s_1$ satisfies a similar relation to that above.
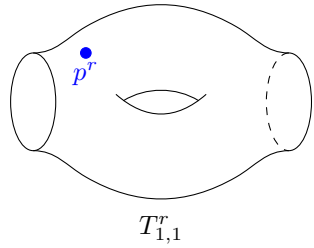
\begin{center} 
\begin{tikzpicture} 

\draw (0,0) node{\Tr};
\node at (0,-1.7) {$ T^r_{1,1}$};

\end{tikzpicture}
\end{center}

\item[G5.] By cylinders we mean \[C_{1,1}:=(\langle c \rangle;\{\langle c \rangle\} , \{\langle c \rangle\})\]
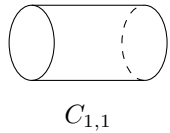
\begin{center} 
\begin{tikzpicture}
\draw (0,0) node{\cyl};
\draw (0,-1) node {$C_{1,1}$};
\end{tikzpicture} 
\end{center}

To also understand cobordisms which don't come from groups and their subgroups we also have the following two generators:
\item[G6.]
Twisting a boundary by $\alpha$:
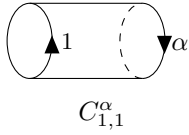
\begin{center} 
\begin{tikzpicture}
\draw (0,0) node{\cylAlpha};
\draw (0,-1) node {$C^{\alpha}_{1,1}$};
\end{tikzpicture} 
\end{center}
\begin{Rem}\label{Rem:twist}
    From the point of view of definition $C^{\alpha}_{1,1}$ is the map $(f_\alpha)_{\#}$ coming from the isomorphism $f_\alpha:\mathbb{Z}_p\rightarrow\mathbb{Z}_p$ taking the generator $1$ to the generator $\alpha$.
    From the point of view of cospans this is the cospan:
     
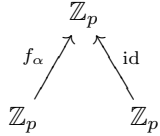
\[ \xymatrix@C=.57em{
   & {\mathbb{Z}_p} \\
  {\mathbb{Z}_p} \ar[ur]^{f_\alpha} & & {\mathbb{Z}_p} \ar[ul]_{\id}} \] 
\end{Rem}
\item[G7.]
Swapping two boundary components, this can be thought of as the switching of the ordering of the two groups in $\{(\mathbb{Z}_p)_i\}_{i \in \{1,2\}}$, or by the family of two group triples $(G_1,\{(\mathbb{Z}_p)_1\},\{(\mathbb{Z}_p)_2\})$ and $(G_2,\{(\mathbb{Z}_p)_2\},\{(\mathbb{Z}_p)_1\})$, where $G_1$ and $G_2$ isomorphic to $\mathbb{Z}_p$. 
Geometrically this is represented by:
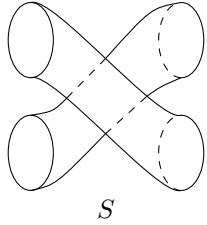
\begin{center} 
\begin{tikzpicture}
\draw (0,0) node{\swap};
\draw (0,-1.5) node {$S$};
\end{tikzpicture} 
\end{center}
\end{enumerate}

\begin{Thm}\label{Thm:generators}
    Any connected cobordism on the skeleton of $Cob^2_p$, can be obtained by gluings copies of the 6 cobordisms $C_{1,1},C_{0,1},C_{1,0},T^r_{1,1},P_{2,1},P_{1,2}$.
\end{Thm}

\begin{proof}
    We will show that we can glue the objects above to get a connected cobordism with invariants $(g,r,n,u)$ for every such tuple.
We start with the cobordism which has one pair of pants (more precisely $P_{2,1}$) and $(n-2)$ cylinders, so this is the cobordism with $n$ connected components, one of which is  $P_{2,1}$ and the rest are $C_{1,1}$. We glue to it a cobordism with $(n-3)$ connected components, one of which is a $P_{2,1}$ and the rest again are $C_{1,1}$, in the following way\footnote{\label{note1}the upper indexing here is used to which cobordism the generator belongs to and not to indicate a power of an element}: 
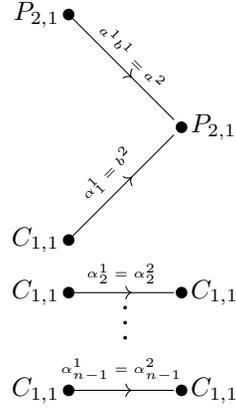
\begin{center} \begin{tikzpicture}[decoration={ markings, mark=at position 0.6 with {\arrow{>}}}]
\filldraw
(0,0) circle (2pt) node[left] {$P_{2,1}$};
\filldraw
(1.5,-1.5) circle (2pt) node[right] {$P_{2,1}$} ;
\filldraw
(0,-3) circle (2pt) node[left] {$C_{1,1}$};
\draw [postaction={decorate}] 
(0,0) -- node[midway, above, sloped] {\tiny $a^1b^1=a^2$} (1.4,-1.4);
\draw [postaction={decorate}] (0,-3)
-- node[midway, above, sloped] {\tiny $\alpha^1_1=b^2$} (1.4,-1.6);
\filldraw (0,-3.7) circle (2pt) node[left] {$C_{1,1}$};
\filldraw(1.5,-3.7) circle (2pt) node[right] {$C_{1,1}$}; 
\draw [postaction={decorate}] (0,-3.7)
--node[midway, above] {\tiny $\alpha^1_2=\alpha_2^2$} 
(1.4,-3.7); 
\filldraw (0.75,-3.9) circle (0.5pt);
\filldraw (0.75,-4.1) circle (0.5pt);
\filldraw (0.75,-4.3) circle (0.5pt);
\filldraw (0,-5) circle (2pt) node[left] {$C_{1,1}$};
\filldraw (1.5,-5) circle (2pt) node[right] {$C_{1,1}$};
\draw [postaction={decorate}] (0,-5)
-- node[midway, above] {\tiny $\alpha^1_{n-1}=\alpha_{n-1}^2$}  
(1.4,-5);
\end{tikzpicture}\end{center}
the result will have one connected component of invariant $(0,0,3,1)$ and $(n-3)$ $C_{1,1}$.
We glue again to it a cobordism with one pair of pants, and the rest of the components $C_{1,1}$.
After $(n-1)$ times we get the following\footref{note1} diagram:
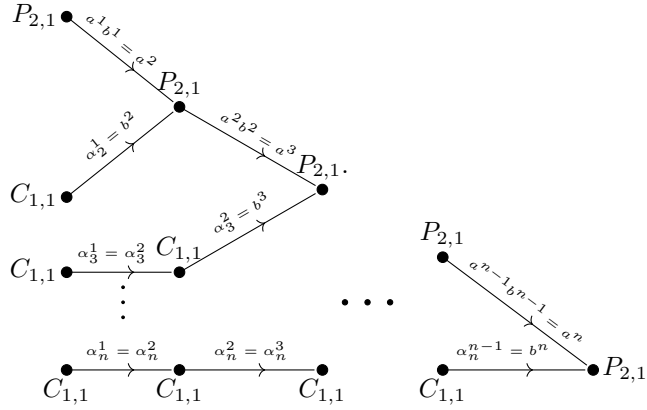
\begin{center} \begin{tikzpicture}[shorten >=3pt, decoration={ markings, mark=at position 0.6 with {\arrow{>}}}]
\filldraw 
(0,-0.3) circle (2pt) node[left] {$P_{2,1}$} ;
\filldraw(1.5,-1.5) circle (2pt) node[above] {$P_{2,1}$};
\filldraw(0,-2.7) circle (2pt) node[left] {$C_{1,1}$};
\draw[postaction={decorate}]
(0,-0.3)
-- node[pos=0.4, above, sloped] {\tiny $a^1b^1=a^2$} 
(1.5,-1.5);
\draw[postaction={decorate}] (0,-2.7)-- node[midway, above, sloped] {\tiny $\alpha_2^1=b^2$} 
(1.5,-1.5);
\filldraw (0,-3.7) circle (2pt) node[left] {$C_{1,1}$};
\filldraw(1.5,-3.7) circle (2pt) node[above] {$C_{1,1}$};
\draw[postaction={decorate}] (0,-3.7) 
--node[pos=0.4, above] {\tiny $\alpha_3^1=\alpha_3^2$} 
(1.5,-3.7);
\draw[postaction={decorate}] 
(1.5,-1.5)
-- node[midway, above, sloped] {\tiny $a^2b^2=a^3$} (3.4,-2.6);
\filldraw(3.4,-2.6) circle (2pt) node[above] {$P_{2,1}.$};
\draw[postaction={decorate}](1.5,-3.7)
-- node[midway, above, sloped] {\tiny $\alpha_3^2=b^3$} 
(3.4,-2.6);
\filldraw (0.75,-3.9) circle (0.5pt);
\filldraw (0.75,-4.1) circle (0.5pt);
\filldraw (0.75,-4.3) circle (0.5pt);
\filldraw (0,-5) circle (2pt) node[below] {$C_{1,1}$};
\filldraw (1.5,-5) circle (2pt) node[below] {$C_{1,1}$};
\filldraw(3.4,-5) circle (2pt) node[below] {$C_{1,1}$};
\draw[postaction={decorate}] (0,-5)
-- node[midway, above] {\tiny $\alpha_{n}^1=\alpha_{n}^2$}  
(1.5,-5);
\draw[postaction={decorate}](1.5,-5) -- node[midway, above] {\tiny $\alpha_{n}^2=\alpha_{n}^3$}  (3.4,-5);
\filldraw (3.7,-4.1) circle (1pt);
\filldraw (4,-4.1) circle (1pt);
\filldraw (4.3,-4.1) circle (1pt);
\filldraw 
(5,-3.5) circle (2pt) node[above] {$P_{2,1}$};
\filldraw
(7,-5) circle (2pt) node[right] {$P_{2,1}$};
\filldraw
(5,-5) circle (2pt) node[below] {$C_{1,1}$};
\draw[postaction={decorate}] 
(5,-3.5) 
-- node[midway, above, sloped] {\tiny $a^{n-1}b^{n-1}=a^{n}$} 
(7,-5);
\draw[postaction={decorate}] (5,-5)-- node[pos=0.4, above, sloped] {\tiny $\alpha^{n-1}_n=b^{n}$} 
(7,-5);
\end{tikzpicture}\end{center}

The above diagram (in the case $n=5$) can be illustrated of geometrically as:
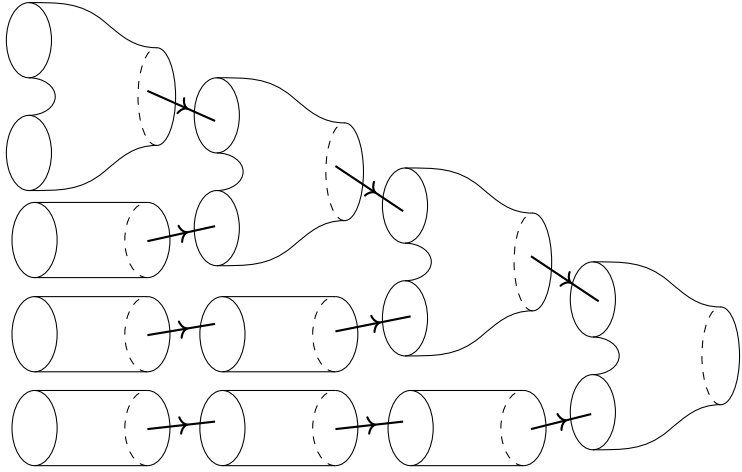
\begin{center} 
\begin{tikzpicture}[outer sep=auto, decoration={
    markings,
    mark=at position 0.6 with {\arrow{>}}}]
 \matrix (name) [matrix of nodes, column sep={2.5cm,between origins}, row sep={2.25cm,between origins}]{
     {\PantsL} &|[yshift=-1cm]|{\PantsL} &      \\ [-1.1cm]
    {\cyl} & \\ [-1cm] 
   {\cyl} &{\cyl} & |[yshift=0.2cm]|{\PantsL} & \\ [-1cm]
   {\cyl} &{\cyl}  &{\cyl}  & |[yshift=0.2cm]|{\PantsL}\\  };
   \draw[line width=0.8pt, postaction={decorate}] (-3,1.9)   -- (-2.1,1.5);
   \draw[line width=0.8pt, postaction={decorate}] (-3,-0.1)   -- (-2.1,0.1);
   \draw[line width=0.8pt, postaction={decorate}] (-3,-1.35)   -- (-2.1,-1.2);
   \draw[line width=0.8pt, postaction={decorate}] (-3,-2.6)   -- (-2.1,-2.5);
    \draw[line width=0.8pt, postaction={decorate}] (-0.5,0.9)   -- (0.4,0.3);
   \draw[line width=0.8pt, postaction={decorate}] (-0.5,-1.3)   -- (0.5,-1.1);
   \draw[line width=0.8pt, postaction={decorate}] (-0.5,-2.6)   -- (0.4,-2.5);
    \draw[line width=0.8pt, postaction={decorate}] (2.1,-0.3)   -- (3,-0.9);
   \draw[line width=0.8pt, postaction={decorate}] (2.1,-2.6)   -- (2.9,-2.4);
\end{tikzpicture}
\end{center}
So we get a connected cobordism with invariant $(0,0,n,1)$. A similar argument, with $P_{1,2}$ gives us a connected cobordism with invariant $(0,0,1,u)$. By gluing $g-1$ orientable tori: $T_{1,1}$, and then one torus orientable up to level $r$:  $T_{1,1}^{r}$, we get a connected cobordism with invariant $(g,r,1,1)$. Finally, gluing the cobordism with invariant $(0,0,n,1)$ to the cobordism with invariant $(g,r,1,1)$ and then to the cobordism with invariant $(0,0,1,u)$, we get the desired cobordism of invariant $(g,r,n,u)$. 
\end{proof}

To keep track of the twists coming from the choice of reduction to the skeleton we then have the following:
Suppose we have a connected cobordism: \[(G,\mathcal{N}_I,\mathcal{U}_J, \phi,\tau).\]
All boundaries are isomorphic to a collection of $\mathbb{Z}_p$. This cobordism is the composition of a family of twists indexed by $I$, i.e. a cobordism with connected components for $i\in I$, of the form: 
\[((\mathbb{Z}_p)_i,\{N_i\},\{(\mathbb{Z}_p)_i\}, \{\phi_i\},\id)\] 
followed by 
\[(G,\phi(\mathcal{N}):=\{\phi_i(N_i)\}_{i \in I}, \ \tau(\mathcal{U}):=\{\tau_j(U_j)\}_{j \in J}, \  \id, \ \id)\] followed by a family of twists indexed by $J$, where the connected component of the cobordism at $j\in J$:
\[((\mathbb{Z}_p)_j,\{(\mathbb{Z}_p)_j\},\{U_j\}, \id,\{\tau_j\}).\]
So one gets that by also adding in the generators $C^{\alpha}_{1,1}$ we get we can get any connected cobordism.
By also considering the cobordism $S$, we see we can reorder the boundaries in whatever order we want. Since a general cobordism is (up to reordering) a disjoint union of connected cobordisms, we get we can also generate all the not necessarily connected cobordisms.

\subsection{Relations}\label{subsect:Rels}
Now that we have generators for every cobordism, we would like to understand the relations among these generators. In other words, we want to know when two ways of gluing cobordisms, give isomorphic cobordisms (i.e. the same invariants $(g,r,n,u)$).
\begin{Rem}\label{Rem:SwapRelations}
    We will be omitting relations involving the swap, as there many of them. Most are the same as in the classical setting, and dealing with them follows exactly in a similar manner (this can be found in Section 1.4 of \cite{kock2004frobenius}, where they are refereed to as twist).
    The only new relation is of course, that for any $\alpha,\beta\in\mathbb{U}_p$ one has that:
    \[S\circ( C_{1,1}^\alpha\coprod C_{1,1}^{\beta})=( C_{1,1}^\beta\coprod C_{1,1}^{\alpha})\circ S\]
    We collectively refer to them as (RS).
\end{Rem}

We now proceed with listing the easier to compute relations, where one glues objects in the skeleton, along the identity map of $\mathbb{Z}_p$. These will later give us the structure of a Frobenius algebra:
\begin{Thm}\label{Thm:BasicRelations}
We have the following five relations among the cobordisms $C_{1,1},C_{0,1},C_{1,0},P_{2,1},P_{1,2}$:
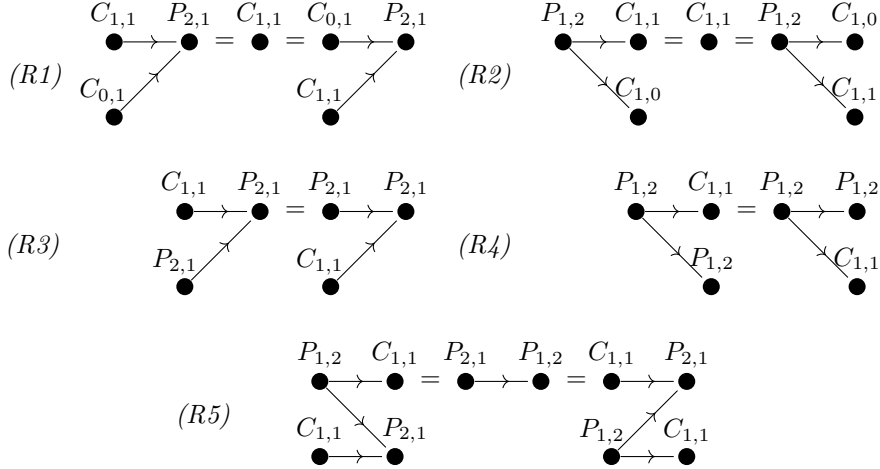
\begin{center} 
\begin{tikzpicture}[shorten >=1.5pt,  every edge/.style={draw,  postaction={nomorepostaction,decorate,  decoration={markings,mark=at position 0.6 with {\arrow{>}}}   }} ]
 \matrix [column sep=1mm,row sep=1mm]
{
\graph[nodes={circle,fill,inner sep=0.8mm},empty nodes]{1--2,3--2}; & \node {$=$}; &
\graph[nodes={circle,fill,inner sep=0.8mm},empty nodes]{4}; & \node {$=$}; & \graph[nodes={circle,fill,inner sep=0.8mm},empty nodes]{5--6,7--6};\\
};
\draw (1) node[above = 1.5pt]{$C_{1,1}$};
\draw (2) node[above = 1.5pt]{$P_{2,1}$};
\draw (3) node[above left = 1pt and -0.3cm]{$C_{0,1}$};
\draw (4) node[above = 1.5pt]{$C_{1,1}$};
\draw (5) node[above = 1.5pt]{$C_{0,1}$};
\draw (6) node[above = 1.5pt]{$P_{2,1}$};
\draw (7) node[above left = 1pt and -0.3cm]{$C_{1,1}$};
\draw (-3,0) node {(R1)};
\end{tikzpicture}
\begin{tikzpicture}[shorten >=1.5pt,  every edge/.style={draw,  postaction={nomorepostaction,decorate,  decoration={markings,mark=at position 0.6 with {\arrow{>}}}   }} ]
 \matrix [column sep=1mm,row sep=1mm]
{
\graph[nodes={circle,fill,inner sep=0.8mm},empty nodes]{1--{2,3}}; & \node {$=$}; &
\graph[nodes={circle,fill,inner sep=0.8mm},empty nodes]{4}; & \node {$=$}; & \graph[nodes={circle,fill,inner sep=0.8mm, branch right},empty nodes]{5--{6,7}};\\
};
\draw (1) node[above = 1.5pt]{$P_{1,2}$};
\draw (2) node[above = 1.5pt]{$C_{1,1}$};
\draw (3) node[above = 1.5pt]{$C_{1,0}$};
\draw (4) node[above = 1.5pt]{$C_{1,1}$};
\draw (5) node[above = 1.5pt]{$P_{1,2}$};
\draw (6) node[above = 1.5pt]{$C_{1,0}$};
\draw (7) node[above = 1.5pt]{$C_{1,1}$};
\draw (-3,0) node {(R2)};
\end{tikzpicture}
\end{center}
\begin{center} 
\begin{tikzpicture}[shorten >=1.5pt,  every edge/.style={draw,  postaction={nomorepostaction,decorate,  decoration={markings,mark=at position 0.6 with {\arrow{>}}}   }} ]
 \matrix [column sep=1mm,row sep=1mm]
{
\graph[nodes={circle,fill,inner sep=0.8mm},empty nodes]{1--2,3--2}; & \node {$=$}; &
 \graph[nodes={circle,fill,inner sep=0.8mm},empty nodes]{5--6,7--6};\\
};
\draw (1) node[above = 1.5pt]{$C_{1,1}$};
\draw (2) node[above = 1.5pt]{$P_{2,1}$};
\draw (3) node[above left = 1pt and -0.3cm]{$P_{2,1}$};
\draw (5) node[above = 1.5pt]{$P_{2,1}$};
\draw (6) node[above = 1.5pt]{$P_{2,1}$};
\draw (7) node[above left = 1pt and -0.3cm]{$C_{1,1}$};
\draw (-3.5,0) node {(R3)};
\end{tikzpicture}
\begin{tikzpicture}[shorten >=1.5pt,  every edge/.style={draw,  postaction={nomorepostaction,decorate,  decoration={markings,mark=at position 0.6 with {\arrow{>}}}   }} ]
 \matrix [column sep=1mm,row sep=1mm]
{
\graph[nodes={circle,fill,inner sep=0.8mm},empty nodes]{1--{2,3}};  & \node {$=$}; & \graph[nodes={circle,fill,inner sep=0.8mm, branch right},empty nodes]{5--{6,7}};\\
};
\draw (1) node[above = 1.5pt]{$P_{1,2}$};
\draw (2) node[above = 1.5pt]{$C_{1,1}$};
\draw (3) node[above = 1.5pt]{$P_{1,2}$};
\draw (5) node[above = 1.5pt]{$P_{1,2}$};
\draw (6) node[above = 1.5pt]{$P_{1,2}$};
\draw (7) node[above = 1.5pt]{$C_{1,1}$};
\draw (-3.5,0) node {(R4)};
\end{tikzpicture}
\end{center}
\begin{center}
\begin{tikzpicture}[shorten >=1.5pt,  every edge/.style={draw,  postaction={nomorepostaction,decorate,  decoration={markings,mark=at position 0.6 with {\arrow{>}}}   }} ]
 \matrix [column sep=1mm,row sep=1mm]
{
\graph[nodes={circle,fill,inner sep=0.8mm},empty nodes]{1--2,4--3,1--3}; & \node {$=$}; &
\graph[nodes={circle,fill,inner sep=0.8mm},empty nodes]{5--6}; & \node {$=$}; & \graph[nodes={circle,fill,inner sep=0.8mm},empty nodes]{7--8,9--{8,10}};\\
};
\draw (1) node[above = 1.5pt]{$P_{1,2}$};
\draw (2) node[above = 1.5pt]{$C_{1,1}$};
\draw (3) node[above right = 1.5pt and -0.3cm]{$P_{2,1}$};
\draw (4) node[above = 1.5pt]{$C_{1,1}$};
\draw (5) node[above = 1.5pt]{$P_{2,1}$};
\draw (6) node[above = 1.5pt]{$P_{1,2}$};
\draw (7) node[above = 1.5pt]{$C_{1,1}$};
\draw (8) node[above = 1.5pt]{$P_{2,1}$};
\draw (9) node[above left = 1pt and -0.3cm]{$P_{1,2}$};
\draw (10) node[above = 1pt]{$C_{1,1}$};
\draw (-4,0) node {(R5)};
\end{tikzpicture}
\end{center}

\end{Thm}

\begin{proof}
    By Lemma \ref{Lem:invariants}, it is enough to check for each relation, that the gluings give rise to the same invariants of cobordisms $(g,r,n,u)$, also as explained in subsection \ref{subsect:gluing}, we can compute the gluings one at a time.
    Also note that since all the objects we glue are orientable, and there are no loops in the diagram, (i.e. we only have free amalgamated products over $\mathbb{Z}_p$ and no HNN extensions), so the result is always orientable, and so $r=\infty$.
    
    We demonstrate relation (R3) as an example.
    Starting with the left hand side of the equality in (R3), notice that $C_{1,1}$ glued to one of the in components of the $P_{2,1}$ gives the cobordism \[(\langle a_1,b_1 \rangle ; \{\langle a_1 \rangle,\langle b_1 \rangle \},\{\langle a_1b_1 \rangle\}).\]
 Gluing the other $P_{2,1}$ from its out component to the above along, we get:
 \[(\langle a_1,b_2,a_2 \rangle ; \{\langle a_1 \rangle ,\langle b_2 \rangle,\langle a_2 \rangle\},\{\langle a_1 b_2 a_2 \rangle\})\]
so the invariant is $(g,r,n,u)=(0,\infty,3,1)$. 
    We now discuss the right hand side of the equality in (R3).
    We first glue $P_{2,1}$ from its out component to one of the in components of the other $P_{2,1}$ we get \[(\langle a_1,b_2,a_2 \rangle ; \{\langle a_1 \rangle,\langle b_2 \rangle,\langle a_2 \rangle\},\{\langle a_1 b_2 a_2 \rangle \}).\]
    Gluing \[C_{1,1}=(\langle c \rangle; \{\langle c \rangle \},\{\langle c \rangle\})\] to the other in component of the $P_{2,1}$ we get \[(\langle c,b_2,a_2 \rangle ; \{\langle c \rangle,\langle b_2 \rangle,\langle a_2 \rangle\},\{\langle c b_2 a_2 \rangle\}),\] which again has invariant $(g,r,n,u)=(0,\infty,3,1)$.
    So we get that these are the isomorphic cobordisms, proving the relation in (R3).
    The rest are similar and the computations are an easy exercise. 
\end{proof}

To understand the other relations, we first need the following proposition:
\begin{Prop}\label{Prop:KleinRelation}
    We have:
    \begin{enumerate}
        \item \[T^r_{1,1}*_{\mathbb{Z}_p}T^{r'}_{1,1}\cong T_{1,1}*_{\mathbb{Z}_p}T^{\min(r',r)}_{1,1}\]
        
        \item Let $\alpha,\beta\in\mathbb{U}_p$, such that $\alpha \beta^{-1}\in (1+p^r\mathbb{Z}_p)$ but $\alpha \beta^{-1}\not\in (1+p^{r+1}\mathbb{Z}_p)$. Suppose we glue the cobordisms \[P_{1,2}\cong (\langle a,b \rangle ; \{\langle ab \rangle\},\{\langle a \rangle,\langle b \rangle\})\] to \[P_{2,1}\cong(\langle c,d \rangle ; \{\langle c \rangle,\langle d \rangle \}, \{\langle cd \rangle\})\]  by isomorphisms sending $a$ to $c^\alpha$ and $b$  to $d^\beta$.  The result of such gluing is $T^r_{1,1}$.  
\end{enumerate}
\end{Prop}

\begin{proof}
The first statement follows from the explicit classification of $\mathbb{Z}_p$ splittings and gluings of pro-$p$ $\PD^2$ groups in \cite{GropperArxiv}. The main point is that given two relative pairs with invariants $(1,1,1,r)$ and $(1,1,1,r')$, the result of the gluing is $(2,1,1,s)$ for some orientation level.
By studying the orientation character one gets that $s\leq r$ and $s\leq r'$. On the other one also cannot have that $s<\min(r,r')$ as the character is determined by the characters on each component, and so overall one gets that the gluing $T^r_{1,1}*_{\mathbb{Z}_p}T^{r'}_{1,1}$ has invariant $(2,1,1,\min(r,r'))$, which also in particular the invariant of $T_{1,1}*_{\mathbb{Z}_p}T^{\min(r',r)}_{1,1}$. 
For the second statement, first note that by Theorem \ref{Thm:GropperSplit}, the condition that $\alpha=\beta$  mod $p$, gives us that the pro-$p$ gluing is the pro-$p$ completion of a discrete gluing. 
In other words, when we glue \[(\langle a,b \rangle; \{\langle ab \rangle \},\{\langle a \rangle ,\langle b \rangle \})\] to \[(\langle c,d \rangle ; \{\langle c \rangle,\langle d \rangle\}, \{\langle cd \rangle \})\]  by isomorphisms from $\langle a \rangle $ to $\langle c \rangle $ and $\langle b \rangle$  to $\langle d \rangle$, we know what the presentation of the resulting gluing is.
We can do this gluing in two stages, first an amalgamation, then an HNN. The isomorphism must send $a$ to $c^\alpha$, a $\mathbb{U}_{p}$ power, $\alpha$ of $c$, and similarly $b$ is sent to $d^\beta$. 
The result of the gluing of $a$ to $c^\alpha$ is the pair \[(\langle c,b,d \rangle ;\{\langle c^\alpha b \rangle,\langle b \rangle,\langle d \rangle,\langle cd \rangle\}).\]
Now gluing $b$ to $d^\beta$, we get the HNN relation, $b=td^\beta t^{-1}$ .
We get the pair \[(\langle c,b,d \rangle ; \{\langle c^\alpha td^\beta t^{-1} \rangle,\langle cd \rangle\}),\] after quotienting out by the boundary subgroups, we see that the abelianization is the group $\langle d \ | \ d^{\alpha}d^{\beta^{-1}} \rangle$, in other words, the orientability level, is the highest power of $p$  dividing $\alpha \beta^{-1} -1$, where $\alpha$  was the change in orientation of the first gluing, and $\beta$ is the change in orientation of the second gluing. In other words, the result of the gluing is isomorphic to $T^r_{1,1}$.

\end{proof}
\begin{Rem}
    The first part of the proposition is analogues to how the connected sum of two Klein bottles is homeomorphic to a connected sum of a torus and a Klein bottle. 
    The second part shows how we can glue ``pair of pants" in different ways to get $T^r_{1,1}$ for different $r$. This is similar to the topological story in which one can glue two pairs of pants in a compatible way to get a twice punctured torus, and in a non compatible way to get a twice punctured Klein bottle. 
    From this we also get that in fact we don't need the tori to generate the category, but rather information about gluings and changing orientation.
\end{Rem}

The last thing we want to understand what happens when we ``change the orientation" of the boundary components of such cobordisms. By this we mean gluing some $C_{1,1}^\alpha$, to a cobordism, which can also be thought of as twisting the inclusion maps $\phi, \tau$ of the cobordism.
Under the point of view of groups and subgroups, saying that composing $C_{1,1}^\alpha$ (twisting the out boundary by $\alpha$) with a cobordism of the triple $(G,\{N\},\{U\})$ does not change the cobordism (i.e. gives an isomorphic cobordism), is equivalent to having an automorphism of $G$, which sends $U$ to (a conjugate of) $U$, and $N$ to (a conjugate of) $N^\alpha$ (i.e. the automorphism of $G$ restricted to $N$ is the automorphism of taking $\alpha$ power).
\begin{Prop}\label{Prop:mobius}
    Let $\alpha\in \mathbb{U}_p$, we have the following 
    \begin{enumerate}
    \item There is an isomorphism of the free pro-$p$ group generated by $G:=\langle a, b\rangle$, which sends $a$ to a conjugate of $a^\alpha$, $b$ to a conjugate of $b^\alpha$, and $ab$ to a conjugate of $(ab)^\alpha$
    
        \item For any $r\in\mathbb{Z}_{>0}$ such that $\alpha\equiv 1 \pmod{p^r}$.
    Let $F=\langle x,y,s_1 \rangle$ be the free pro-$p$ group on three generators.
    There exists an automorphism of $F$ sending $s_1$ to a conjugate of itself, and $s_0=s_1 x^{p^r}[x,y]$ to a conjugate of $s_0^\alpha$. 
    \end{enumerate}
\end{Prop}
\begin{proof}
    For the first case we will use the descending $p$-central series of $G$, defined by $G_1:=G$, and $G_{i+1}:=G^p[G_i,G]$. We will show by induction that if up to a change by conjugate, one has that  $(ab)^\alpha\equiv a^\alpha \alpha$ mod $G_i$, then the same is true mod $G_{i+1}$.
    Since $\alpha\in(1+p\mathbb{Z}_p)$ then mod $G_2$ this is just $(ab)\equiv ab$.
    Now mod $p^2$, $\alpha$ is of the form $1+mp$ for some $m$. 
    From the definition of $G_3$ and the commutator we have that: $(ab)^{mp}\equiv a^{mp}b^{mp}[a,b]^{\binom{p}{2}}$ mod $G_3$.
     Since $p$ powers of a commutator is trivial mod $G_3$, we have that: \[(ab)^\alpha\equiv (ab)(ab)^{mp}\equiv aba^{mp}b^{mp}\]
    and since $p$ powers of elements are in the center of $G/G_3$, then $(ab)^\alpha\equiv a^\alpha b^\alpha$ mod $G_3$. Now that we have the desired property mod $G_3$, by following the argument in the proof of Theorem 3.3 in \cite{wilkes2020classification}, for the case of $n=0,b=2$, we get that up to conjugation $(ab)^\alpha\equiv a^\alpha b^\alpha$ mod $G_i$ for all $i$ and hence also for $G$.

   The second case is just slightly more complicated.
    Let $q=p^r$. To construct the  automorphism, we will look at the descending $q$-central series of $F$.
    This is a filtration 
    \[
    \cdots \subseteq F_3 \subseteq F_2 \subseteq F_1=F 
    \]
    on $F$ defined by the following:\[F_1:=F, \ \ F_{i+1}:=F_i^q[F_i,F].\]
    Since $[F_i,F_j]\subseteq F_{i+j}$ we have that $\gr(F)=\underset{i=0}{\overset{\infty}\bigoplus} F_i/F_{i+1}$ is a graded Lie algebra over $\mathbb{Z}/p\mathbb{Z}$.
    We will construct the automorphism by induction, showing that if we have the automorphism mod $F_i$ then we also have it mod $F_{i+1}$.
    First note that since mod $F_2$,  $s_0=s_1$ and $\alpha\equiv 1$ mod $q$, the desired automorphism mod $F_2$ is just the identity.
    We now move to compute mod $F_3$.
    First note we have that mod $F_3$ the power by $\alpha\equiv1+mq$ where $0\leq m < q$: 
    \begin{equation}
        \begin{split}
    s_0^\alpha& =(s_1 x^{p^r}[x,y])^{1+mq}=s_0(s_1 x^{p^r}[x,y])^{mq}=s_0s_1^{mq}(x^q[x,y])^{mq}[x^q[x,y],s_1]^{\binom{q}{2}} \\ & \equiv s_0s_1^{mq} .  
        \end{split}
    \end{equation}
    On the other hand, the basis change $x\mapsto xs_1^m$ gives us: 
   \begin{equation}
       \begin{split}
   s_0\mapsto & s_1(xs_1^m)^q[xs_1^m,y]\equiv s_1x^qs_1^{qm}[x,s_1]^{\binom{q}{2}}[x,y][s_1^m,y] \\ & \equiv s_1x^qs_1^{mq}[x,y][s_1,y^m] 
    \equiv s_1^{1+mq}x^q[x,y]s_1^{-1}y^{-m}s_1y^m \\ & =s_1^{mq}x^q[x,y]y^{-m}s_1y^m=y^{-m}s_1^{1+mq}x^q[x,y]y^m \\ & =y^{-m}s_0^\alpha y^m.
   \end{split}
   \end{equation}
 
    We used the fact the mod $F_3$ elements commute with $q$ powers and commutators, that $[a^m,b]\equiv[a,b^m]$, $(ab)^q\equiv a^qb^q[a,b]^{\binom{q}{2}}$ and that $p$ is odd.
    Now following the argument in the proof of Theorem 3.3 in \cite{wilkes2020classification}, and noting that all changes to $s_1$ are just by conjugation, we see that we can then lift this to any $F_n$. Thus, we can also lift it to all of $F$ as desired.
\end{proof}

To summarize we have added to the relations in Theorem \ref{Thm:BasicRelations}, the relation which give us the extended Frobenius structure: 
\begin{Thm}\label{Thm:ExtendedRelations} 
The following relations hold in $Cob^2_p$:
\begin{center}
    \begin{tikzpicture}[  every edge/.style={draw,  postaction={nomorepostaction,decorate,  decoration={markings,mark=at position 0.6 with {\arrow{>}}}   }} ]
 \matrix [column sep=0.8mm,row sep=0.8mm]
{
\graph[nodes={circle,fill,inner sep=0.8mm},empty nodes]{1}; & \node {$=$}; &
\graph[ nodes={circle,fill,inner sep=0.8mm},empty nodes]{2--3}; \\
};
\draw (1) node [above = 1.5pt] {$C_{0,1}$};
\draw (2) node[above = 1.5pt]{$C_{0,1}$};
\draw (3) node[above = 1.5pt ]{$C_{1,1}^\alpha$};

\draw (-4,0) node {$\alpha \in\mathbb{U}_p$};
\draw (-6.6,0) node {(R6)};
\end{tikzpicture}
\end{center}
\begin{center}
    \begin{tikzpicture}[  every edge/.style={draw,  postaction={nomorepostaction,decorate,  decoration={markings,mark=at position 0.6 with {\arrow{>}}}   }} ]
 \matrix [column sep=0.8mm,row sep=0.8mm]
{
\graph[nodes={circle,fill,inner sep=0.8mm},empty nodes]{1}; & \node {$=$}; &
\graph[ nodes={circle,fill,inner sep=0.8mm},empty nodes]{2--3}; \\
};
\draw (1) node [above = 1.5pt] {$C_{1,0}$};
\draw (2) node[above = 1.5pt]{$C_{1,1}^\alpha$};
\draw (3) node[above = 1.5pt ]{$C_{1,0}$};

\draw (-4,0) node {$\alpha \in\mathbb{U}_p$};
\draw (-6.6,0) node {(R7)};
\end{tikzpicture}
\end{center}
\begin{center} 
\begin{tikzpicture}[shorten >=1.5pt,  every edge/.style={draw,  postaction={nomorepostaction,decorate,  decoration={markings,mark=at position 0.6 with {\arrow{>}}}   }} ]
 \matrix [column sep=1mm,row sep=1mm]
{
\graph[nodes={circle,fill,inner sep=0.8mm},empty nodes]{1--2,3--2}; & \node {$=$}; &
 \graph[nodes={circle,fill,inner sep=0.8mm},empty nodes]{5--6};\\
};
\draw (1) node[above = 1.5pt]{$C_{1,1}^\alpha$};
\draw (2) node[above = 1.5pt]{$P_{2,1}$};
\draw (3) node[above left = 1pt and -0.3cm]{$C_{1,1}^\alpha$};
\draw (5) node[above = 1.5pt]{$P_{2,1}$};
\draw (6) node[above = 1.5pt]{$C_{1,1}^\alpha$};
\draw (-4,0) node {$\alpha \in\mathbb{U}_p$};
\draw (-6.3,0) node {(R8)};
\end{tikzpicture}
\end{center}
\begin{center}
\begin{tikzpicture}[shorten >=1.5pt,  every edge/.style={draw,  postaction={nomorepostaction,decorate,  decoration={markings,mark=at position 0.6 with {\arrow{>}}}   }} ]
 \matrix [column sep=1mm,row sep=1mm]
{ \graph[nodes={circle,fill,inner sep=0.8mm, branch right},empty nodes]{5--{6}};& \node {$=$}; &
\graph[nodes={circle,fill,inner sep=0.8mm},empty nodes]{1--{2,3}};  \\
};
\draw (1) node[above = 1.5pt]{$P_{1,2}$};
\draw (2) node[above = 1.5pt]{$C_{1,1}^\alpha$};
\draw (3) node[above = 1.5pt]{$C_{1,1}^\alpha$};
\draw (5) node[above = 1.5pt]{$C_{1,1}^\alpha$};
\draw (6) node[above = 1.5pt]{$P_{1,2}$};
\draw (-4,0) node {$\alpha \in\mathbb{U}_p$};
\draw (-6.3,0) node {(R9)};
\end{tikzpicture}
\end{center}
\begin{center}
\begin{tikzpicture}[  every edge/.style={draw,  postaction={nomorepostaction,decorate,  decoration={markings,mark=at position 0.6 with {\arrow{>}}}   }} ]
 \matrix [column sep=0.8mm,row sep=0.8mm]
{
\graph[nodes={circle,fill,inner sep=0.8mm},empty nodes]{1}; & \node {$=$}; &
\graph[no placement, nodes={circle,fill,inner sep=0.8mm},empty nodes]{2[at ={(0,0)}]--3[at ={(1,0.7)}]--5[at={(2,0)}],2--4[at={(1,-0.7)}]--5}; \\
};
\draw (1) node [above = 1.5pt] {$T_{1,1}^r$};
\draw (2) node[above = 1.5pt]{$P_{1,2}$};
\draw (3) node[above = 1.5pt ]{$C_{1,1}^\alpha$};
\draw (4) node[below =1.5pt]{$C_{1,1}^\beta$};
\draw (5) node[above = 1.5pt]{$P_{2,1}$};
\draw (-3.7,0.5) node {$\alpha \beta^{-1}\in(1+p^r\mathbb{Z}_p)$};
\draw (-3.7,-0.5) node {$\alpha \beta^{-1}\not\in(1+p^{r+1}\mathbb{Z}_p)$};
\draw (-6,0) node {(R10)};
\end{tikzpicture}
\end{center}
\begin{center}
\begin{tikzpicture}[  every edge/.style={draw,  postaction={nomorepostaction,decorate,  decoration={markings,mark=at position 0.6 with {\arrow{>}}}   }} ]
 \matrix [column sep=0.8mm,row sep=0.8mm]
{
\graph[nodes={circle,fill,inner sep=0.8mm},empty nodes]{1--2}; & \node {$=$}; &
\graph[nodes={circle,fill,inner sep=0.8mm},empty nodes]{3--4}; \\
};
\draw (1) node[above = 1.5pt]{$T_{1,1}^r$};
\draw (2) node[above = 1.5pt]{$T_{1,1}^{r'}$};
\draw (3) node[above = 1.5pt ]{$T_{1,1}$};
\draw (4) node[above right= 1.5pt and -0.3cm ]{$T_{1,1}^{\min(r,r')}$};

\draw (-5,0) node {(R11)};
\end{tikzpicture}
\end{center}

\begin{center}
    \begin{tikzpicture}[  every edge/.style={draw,  postaction={nomorepostaction,decorate,  decoration={markings,mark=at position 0.6 with {\arrow{>}}}   }} ]
 \matrix [column sep=0.8mm,row sep=0.8mm]
{
\graph[nodes={circle,fill,inner sep=0.8mm},empty nodes]{1}; & \node {$=$}; &
\graph[ nodes={circle,fill,inner sep=0.8mm},empty nodes]{2--3}; \\
};
\draw (1) node [above = 1.5pt] {$T_{1,1}^r$};
\draw (2) node[above = 1.5pt]{$C_{1,1}^\alpha$};
\draw (3) node[above = 1.5pt ]{$T_{1,1}^r$};

\draw (-4,0) node {$\alpha \in(1+p^r\mathbb{Z}_p)$};
\draw (-6.6,0) node {(R12)};
\end{tikzpicture}
\end{center}
\end{Thm}
\begin{proof}
    Relations (R6),(R7) are immediate from the definition of $C_{0,1},C_{1,0}$ and the definition of isomorphism of a group triple.
    One directly gets (R10) and (R11) from Proposition \ref{Prop:KleinRelation}.
    To get relations (R12), we apply the second part of Proposition \ref{Prop:mobius}. This gives us that for any $\alpha\equiv 1\pmod{p^{r}}$ there exists an automorphism \[
\psi:T^r_{1,1}\rightarrow T^r_{1,1}
\] of the group triple \[T^r_{1,1}:=(\{<s_1,x,y>,\{<s_1>\},\{<s_0>\})\] with $s_0:=s_1x^{p^r}[x,y]$ so that $\psi$ restricts on the boundaries to $s_1\mapsto s_1$, and $s_0\mapsto s_0^\alpha$ (up to conjugation). This exactly means that applying $C_{1,1}^\alpha$ to the out boundary of $T_{1,1}^r$ is equal to simply $T_{1,1}^r$
    In a similar way, by the first part of Proposition \ref{Prop:mobius}, the automorphism of the free pro-p group on two generators $\langle a,b\rangle$ which maps $a\mapsto a^\alpha$ $b\mapsto b^\alpha$, gives automorphisms of the group triples $P_{1,2}$, and $P_{2,1}$ which twists each boundary by $\alpha$, which is exactly (R8) and (R9).
\end{proof}
\begin{Rem}
    We can summarize the above theorem in terms of group theoretic gluing of cobordisms as follows:
    \begin{enumerate}
    
    \item
        For any $\alpha\in\mathbb{U}_p$
        \[C^{\alpha}_{1,1}  \circ C_{0,1} \cong C_{0,1}.\]
    \item
        For any $\alpha\in\mathbb{U}_p$
        \[ C_{1,0} \circ C^{\alpha}_{1,1}  \cong C_{1,0}.\]
        
     \item
        For any $\alpha\in\mathbb{U}_p$
        \[P_{2,1} \circ (C^{\alpha}_{1,1} \coprod C^{\alpha}_{1,1} )\cong C^{\alpha}_{1,1}  \circ P_{2,1}.\]

    \item For any $\alpha\in\mathbb{U}_p$ \[
        P_{1,2} \circ C^{\alpha}_{1,1}  \cong  (C^{\alpha}_{1,1} 
         \coprod C^{\alpha}_{1,1} ) \circ P_{1,2}.\]
    
    \item For $\alpha,\beta\in\mathbb{U}_p$, such that $\alpha \beta^{-1}\in (1+p^r\mathbb{Z}_p)$ but $\alpha \beta^{-1}\not\in (1+p^{r+1}\mathbb{Z}_p)$   \[
    T^r_{1,1} \cong (\langle a,b \rangle ; \{\langle ab \rangle\},\{\langle a \rangle,\langle b \rangle\}) \underset{\underset{b \mapsto d^\beta}{a\mapsto c^\alpha} }\circ (\langle c,d \rangle ; \{\langle c \rangle,\langle d \rangle\}, \{\langle cd \rangle \}. \]
    
      \item   One has that  \[ T^r_{1,1}\circ T^{r'}_{1,1}\cong T_{1,1} \circ T^{\min(r',r)}_{1,1}.  \]
    where the gluing is simply by the identity on the copies of $\mathbb{Z}_p$ in the boundaries in standard presentations. 
    
    \item For any $\alpha\in\mathbb{U}_p$ such that $\alpha\in(1+p^r\mathbb{Z}_p)$ \[C^{\alpha}_{1,1} \circ T_{1,1}^r  \cong T_{1,1}^r.\]
    
 \end{enumerate}
\end{Rem}

\begin{Thm}\label{Thm:allRel}
    The relations of Theorem \ref{Thm:BasicRelations}, Theorem (so all together (R1)-(R12)) and (RS) are all the relations of $\Cob_p^2$.
\end{Thm}
\begin{proof}
    First, if a relation does not involve any non orientable things (i.e. no $T_{1,1}^r$ or $C_{1,1}^\alpha$), then the same argument as in \cite{kock2004frobenius} 1.4.36 and 1.4.37 will bring it to a normal form.
    Now given a relation for some connected cobordism with invariant $(g,r,n,u)$ that does have some $C_{1,1}^\alpha$ but has no switch $S$.
    We can use use a similar process as done by Kock of moving to normal form, which will result in a relation of the following form:\\
    Some collection of $C_{1,1}^{\alpha_i}$ $1\leq i\leq n$, followed by the normal form cobordism $(0,0,n,1)$ from Theorem \ref{Thm:generators}, followed by $g$ successive gluings of
    
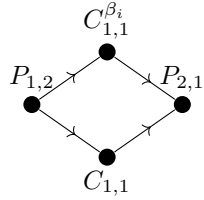
\begin{center}
\begin{tikzpicture}[  every edge/.style={draw,  postaction={nomorepostaction,decorate,  decoration={markings,mark=at position 0.6 with {\arrow{>}}}   }} ]

\graph[no placement, nodes={circle,fill,inner sep=0.8mm},empty nodes]{2[at ={(0,0)}]--3[at ={(1,0.7)}]--5[at={(2,0)}],2--4[at={(1,-0.7)}]--5}; 

\draw (2) node[above = 1.5pt]{$P_{1,2}$};
\draw (3) node[above = 1.5pt ]{$C_{1,1}^{\beta_i}$};
\draw (4) node[below =1.5pt]{$C_{1,1}$};
\draw (5) node[above = 1.5pt]{$P_{2,1}$};
\end{tikzpicture}
\end{center}
for some $\beta_j$, $1\leq j\leq g$,
and then the normal form cobordism $(0,0,1,u)$ of Theorem  \ref{Thm:generators} and then some collection of $C_{1,1}^{\gamma_k}$, $1\leq k\leq u$. Using relations (R10) and (R11) we get that the middle part can be moved to the form of $(g-1)$ gluings of $T_{1,1}$, followed by a single $T_{1,1}^r$. We have arrived at the normal form of a general cobordism. 
To complete the proof one would like to deal with non connected cobordisms, and relations which involve a switch $S$, the elimination of the $S$ when moving to a normal form is done exactly as in \cite{kock2004frobenius} 1.4.39 and 1.4.40, using the relations (RS), and also (R8), (R9) when moving by a $C_{1,1}^\alpha$.
We conclude that the relations we listed are indeed a sufficient list.
\end{proof}
\section{Equivalence of categories}\label{sect:Equiv}
We are now ready to prove the main theorem of this section 
\begin{proof}[Proof of Theorem \ref{Thm:Main}]
First, given a R-TQFT $(A, \zeta)$ over $p$, we will define a $\mathbb{U}_p$-extended Frobenius algebra in a functorial way.
The module $A$ will be $A_{\mathbb{Z}_p}$.
The cobordisms $C_{1,0}$ and $C_{0,1}$ give us maps $\epsilon: A_{\mathbb{Z}_p}\rightarrow A_\emptyset=R$ and  $\iota: R=A_\emptyset\rightarrow A_{\mathbb{Z}_p}$ resp. which are the unit and co-unit of the Frobenius algebra.
Given a connected $(1+1)$  cobordism $(G,N,U)$ and a choice of standard form basis with $N_i=<s_i>,$ $U_j=<s_j>$,  we have a canonical way to identify $A_N$ to $\underset{i=1, \dots, n}\bigotimes A_{\mathbb{Z}_{p}}$ via isomorphisms $\psi_i:<s_i>\rightarrow \mathbb{Z}_p$ sending generator $s_i$  to $1\in\mathbb{Z}_p$ and similarly for $A_U$. 
So using $\zeta(G):A_N \to A_U$ and the compositions $(\psi_j)_{\#} \circ \zeta(G) \circ (\psi^{-1}_i)_{\#}$ we get maps \[\underset{i=1, \dots, n}\bigotimes A_{\mathbb{Z}_{p}}\rightarrow \underset{j=1, \dots, u}\bigotimes A_{\mathbb{Z}_{p}}. \]

Axiom (T4) in the definition of a TQFT, gives us that these maps are independent of the choice of basis.
We thus have from $P_{2,1}$ the map $m$ and from $P_{1,2}$  the map $\Delta$.
 The relations in (RS) is exactly saying that the symmetric monoidal structure on $Cob^2_p$ is compatible with the one of $Proj(R)$.
By relations (R1) to (R5) in Theorem \ref{Thm:BasicRelations} and (RS) we get that the above is a Frobenius algebra (i.e. satisfies (F1)-(F5)+(FS)).

We now want to also check the $\mathbb{U}_p$-extended structure.\\
As explained in Remark \ref{Rem:twist}, for any $\alpha \in\mathbb{U}_p$ we have the isomorphism $f_\alpha :\mathbb{Z}_{p}\rightarrow \mathbb{Z}_{p}$ given by the generator $1$ to the generator $\alpha$ which defines
\[\phi_\alpha:=(f_\alpha)_{\#}:A_{\mathbb{Z}_{p}}\rightarrow A_{\mathbb{Z}_{p}}.\]

By \ref{Thm:ExtendedRelations} (R6),(R7) and axiom (T4) we have that $\phi_\alpha\circ \iota=\iota$ and $\epsilon\circ\phi_\alpha=\epsilon$ (which correspond to (F6),(F7)).

From relations (R8) and axiom (T4) the following commutative diagram:
\[
\xymatrix@C=7pc{A_{\langle a \rangle}\otimes_R A_{\langle b \rangle}\ar[d]_{( \ )^\alpha_{\#}\otimes_R( \ )^\alpha_{\#}}\ar[r]^{\zeta(P_{2,1})} &A_{\langle ab \rangle}\ar[d]^{( \ )^\alpha_{\#}} \\
A_{\langle a \rangle}\otimes_R A_{\langle b \rangle}\ar[r]_{\zeta(P_{2,1})} &  A_{\langle ab \rangle}}
\]
and therefore
$\phi_\alpha\circ m=m\circ (\phi_\alpha\otimes \phi_\alpha)$ which is axiom (F8) in the definition of an $\mathbb{U}_p$-extended Frobenius algebra and similarly from relation (R9) one has
\[
\xymatrix@C=7pc{A_{\langle ab \rangle}\ar[d]_{( \ )^\alpha_{\#}}\ar[r]^{\zeta(P_{1,2})} & A_{\langle a \rangle}\otimes_R A_{\langle b \rangle}\ar[d]^{( \ )^\alpha_{\#}\otimes_R( \ )^\alpha_{\#}}\\
A_{\langle ab \rangle}\ar[r]_{\zeta(P_{1,2})} & A_{\langle a \rangle}\otimes_R A_{\langle b \rangle}}
\]
and so 
$(\phi_\alpha\otimes \phi_\alpha)\circ  \Delta=\Delta \circ \phi_\alpha$ which is axiom (F9) in the definition of an $\mathbb{U}_p$-extended Frobenius algebra.
Using (T7), and Since $m\circ\Delta$ corresponds to a twice punctured torus we get that from (R10) of Theorem \ref{Thm:ExtendedRelations} correspond to relation (F10) of Frobenius algebras: 
\[m \circ (\phi_\alpha\otimes\phi_\beta) \circ \Delta \circ \iota =\kappa_r,\] where $r$ the highest power of $p$  dividing $\alpha \beta^{-1} -1$.
    Similarly, one gets from axiom (T7), and from (R11) of Theorem \ref{Thm:ExtendedRelations} axiom (F11) of Frobenius algebras:\[m \circ (\kappa_r \otimes \kappa_{r'})= m \circ \Delta \circ \kappa_{\min(r,r')}.\] 

We are left with showing axiom (F12). First note that (R12) and (T4) give us:
\[
\xymatrix{A_{\langle s_1 \rangle}\ar[d]_{\text{id}_{\#}}\ar[r]^{\zeta(T_{1,1}^r)} & A_{\langle s_0 \rangle}\ar[d]^{( \ )^\alpha_{\#}}\\
A_{\langle s_1 \rangle}\ar[r]_{\zeta(T_{1,1}^r)} & A_{\langle s_0 \rangle}}
\]
 Since the map on $A$ coming from $\zeta(T_{1,1}^r)$ is exactly the one sending $v$ to $m\circ (\kappa_r\otimes v)$, we get that
 \[\phi_\alpha \circ m\circ (\kappa_r\otimes v) =m\circ (\kappa_r\otimes v) \]
 which is (F12).
 
And so we get that any $R$-TQFT over $p$, defines a $\mathbb{U}_p$-extended Frobenius algebra in a functorial way.

Now given a $\mathbb{U}_p$-extended Frobenius algebra $(V,\phi_\alpha,\kappa_r)$ we assign to a one dimensional connected object $M$
\[
A_M=\{(\psi,v) \ | \ \psi:M\rightarrow \mathbb{Z}_p, \ v\in V\}/\sim
\]
where $(\psi_1,v_1)\sim (\psi_2,v_2)$ if $\alpha\circ\psi_1=\psi_2$ and $\phi_\alpha(v_1)=v_2$ for some $\alpha\in \Aut(\mathbb{Z}_p)$.
For a family of $PD^1$ groups one simply takes the tensor product of these vector spaces.
\\
Let FG be the full subcategory of groupoids of the form $\underset{1,2,\dots, n}\coprod B\mathbb{Z}_{p}$, for $n=0,1,\dots$. To any groupoid $\mathfrak{G}$ equivalent to a groupoid of this form, assign the $R$-module of all triples \[\{\psi:\mathfrak{G} \to \mathfrak{F}, \mathfrak{F} \in FG, \zeta \in V^{\otimes_{R} |\pi_{0}(\mathfrak{G})|} \ | \ \psi \ \text{is an equivalence of categories}\}/\sim.\] 

Now similar to the classical setting (For example in \cite{kock2004frobenius}), we define the value of the TQFT on the generators of $\Cob^2_p$ (i.e. on a pro-p $PD^2$ triple or an isomorphism of $PD^1$ groups) to be as follows:
\begin{center}
\begin{tabular}{|c|c|}
\hline 
$\Cob_{p}^n$ generator & $\mathbb{U}_{p}$-extended Frobenius algebras\tabularnewline
\hline 
\hline 
$P_{2,1}$ & $m$\tabularnewline
\hline 
$P_{1,2}$ & $\Delta$\tabularnewline
\hline 
$C_{1,0}$ & $\epsilon$\tabularnewline
\hline 
$C_{0,1}$ & $\iota$\tabularnewline
\hline 
$C_{1,1}^{\alpha}$ & $\phi_{\alpha}$\tabularnewline
\hline 
$S$ & $\sigma$ \tabularnewline
\hline 
\end{tabular}
\end{center}

Given a cobordism, we can decompose it into the generators of Theorem \ref{Thm:generators}. By definition of a $\mathbb{U}_p$-extended Frobenius algebra over $R$, all the relations of Theorem \ref{Thm:allRel} are satisfied and so we get a well a symmetric monoidal functor from $\Cob_p^n$ to $\Proj(R)$. 

It is easy to see that applying the previous functor to the resulting TQFT, gives us an equivalent $\mathbb{U}_p$-extended Frobenius algebra over $R$.
\end{proof}
\newpage
\section{Pro-\texorpdfstring{$p$}{p} Dijkgraaf-Witten theory}\label{sect:DW}
\subsection{Defining the TQFT}\label{subsect:DefDW}
In this section we will construct a TQFT of dimension $1+1$, which is an analogue/generalization of Dijkgraaf-Witten theory (which can be thought of as a Chern-Simons theory with a finite gauge group). As explained in \cite{freed1993lectures, freed1993chern}, given a finite group $\Theta$, the Dijkgraaf-Witten theory is defined as follows:

First one has that the space of fields on a manifold $X$ is $\text{Prin}_\Theta(X)$ the space of principal $\Theta$-bundles on $X$. One can take the trivial action, and the mass of (or measure on) a field $P\in\text{Prin}_\Theta(X)$ to be $\mu_X(P)=\frac{1}{|\Aut(P)|}$. \\
The TQFT associates to a 1-dimensional manifold $Y$ the space $L^2(\overline{\text{Prin}_\Theta(Y)},\mu_Y)$, where $\overline{\text{Prin}_\Theta(Y)}$ is the set of equivalence classes of principal $\Theta$-bundles on $Y$.
For a 2-manifold $X$ with boundary $Y$, the TQFT takes $Q$ to the volume of $\text{Prin}_\Theta(X)(Q)$, the set of $\Theta$-bundles on $X$ whose restriction to $Y$ is $Q$, up to isomorphisms of bundles, which are the identity on $Y$.
\begin{Rem}
    More generally one can define such a theory for any cohomology class $\alpha\in H^{n+1}(B\Theta,\mathbb{C}^\times)$, by associating to any $n$-dimensional manifold $Y$ the global sections of a certain complex line bundle defined by $\alpha$.
\end{Rem}
For a connected manifold $M$, there a natural group theoretic identification of the set of equivalence classes of principal $\Theta$ bundles on $M$, namely:
\[\text{Prin}_\Theta (M)\cong \Hom(\pi_1(M),\Theta)/\sim\]
where the equivalence is by the conjugation action of $G$.
One then has that for $\alpha=0$ the space of global sections is the $\mathbb{C}$-vector space generated by the $\Hom(\pi_1(M),\Theta)/\sim$. This is finite since $\Theta$ is finite, and $\pi_1(M)$ is finitely generated.

This leads us to make the following definition of a pro-$p$ Dijkgraaf-Witten theory.
\begin{Def}\label{Def:DWTheory}
Fix $\Gamma$ a finite $p$ group.
Define the $(1+1)$ dimensional pro-$p$ Dijkgraaf-Witten theory with gauge group $\Gamma$ to be the following TQFT:
\begin{itemize}
    \item 
    For a finite collection of pro-$p$ $\PD^1$ groups $\mathcal{N}=\{N_i\}_{i\in I}$ we assign $Z(\mathcal{N})$ the $\mathbb{C}$ vector space of maps into $\mathbb{C}$, from the set $\underset{i\in I}\coprod\Hom(N_i,\Gamma)/\sim$.
    \item For a pro-$p$ 2-cobordism $(G,\mathcal{N}_I,\mathcal{U}_J, \phi_I,\tau_J)$ define the morphism associated to it to be 
    \[Z(G):Z(\mathcal{N})\mapsto Z(\mathcal{U})\]
    \[Z(G)(f)(p_2)=\sum_{p_1\in \Hom(\mathcal{N},\Gamma)/\sim} \ \sum_{p\in P_G(p_1,p_2)}\frac{|\Aut(p_2)|}{|\Aut(p)|}f(p_1)\]
    where \[P_G(p_1,p_2)=\{\Hom(G,\Gamma)/\sim \ | \ p_1=p\circ\phi_I, \ p_2=p\circ\tau_J\},\] $f\in\text{Map}(\Hom(\mathcal{N},\mathbb{C})$ and $\Aut(p)$ is the stabilizer under conjugation by $\Gamma$ of a representative homomorphism of $p$ (i.e. the size of the centralizer in $\Gamma$ of the homomorphism).
   \end{itemize}
\end{Def}
\begin{Rem}
    Since we are working with categories of pro-$p$ groups and all amalgamations are in such categories, we need $\Gamma$ to be a $p$-group for this to have a chance to define a TQFT.
\end{Rem}

In order to see more easily that the construction of a similar picture in our case is functorial, we will turn to look at the above in the language of cospans and groupoids. 
This way of rephrasing, also allows to establish a future connection to categories of Calabi-Yau cospans, a concept developed in \cite{bozec2024topologicalfieldtheoriesassociated}, \cite{MR4319769}, \cite{gorsky2024countingcalabiyaucategoriesapplications} as well as in \cite{MR3678002} and \cite{MR3381472}. 
We will use a sequence of symmetric monoidal functors 
\[
\Cob_p^2 \longrightarrow \Cospan(\pGrpd) \longrightarrow \Span(\fGrpd) \longrightarrow \Proj(\mathbb{C})
\]
relating the structures 
\[
\coprod \longrightarrow \coprod \longrightarrow \prod \longrightarrow \otimes_{\mathbb{C}}.
\]
The functor 
\[
\Cospan(\pGrpd) \longrightarrow \Span(\fGrpd)
\]
is given by $\Fun(-, B\Gamma)$.
Any morphism 
\[
\mathfrak{G}_1 \longrightarrow \mathfrak{G}_{12} \longleftarrow \mathfrak{G}_2
\]
in the cospan category $\Cospan(\pGrpd)$ gives rise to the span diagram 
\begin{equation}\label{eqn:SpanDiag}
\Fun(\mathfrak{G}_1, B\Gamma) \stackrel{r_1}\longleftarrow \Fun(\mathfrak{G}_{12}, B\Gamma) \stackrel{r_2}\longrightarrow \Fun(\mathfrak{G}_2, B\Gamma)
\end{equation}
of finite groupoids which we think of as a morphism from $\Fun(\mathfrak{G}_1, B\Gamma)$ to $\Fun(\mathfrak{G}_2, B\Gamma)$. Such a diagram can be composed with 
\begin{equation*}\label{eqn:FunDiag1}
\Fun(\mathfrak{G}_2, B\Gamma) \stackrel{s_2}\longleftarrow \Fun(\mathfrak{G}_{23}, B\Gamma) \stackrel{s_3}\longrightarrow \Fun(\mathfrak{G}_3, B\Gamma)
\end{equation*}
thought of as a morphism from $\Fun(\mathfrak{G}_2, B\Gamma)$ to $\Fun(\mathfrak{G}_3, B\Gamma)$. 
Indeed if we define 
\[\mathfrak{G}_{13}= \mathfrak{G}_{12} \underset{\mathfrak{G}_2}\coprod \mathfrak{G}_{23}
\]
to give 
using \[\Fun(\mathfrak{G}_{13}, B\Gamma) \cong \Fun(\mathfrak{G}_{12}, B\Gamma) \times_{\Fun(\mathfrak{G}_2, B\Gamma)}  \Fun(\mathfrak{G}_{23}, B\Gamma)\]
the following composition of spans of finite groupoids where the square is a pullback:
\begin{equation}\label{eqn:pullback}
\xymatrix@C=.57em{
 & & \Fun(\mathfrak{G}_{13}, B\Gamma) \ar[dr]^{\pi_3} \ar[dl]_{\pi_1}& & \\
 & \Fun(\mathfrak{G}_{12}, B\Gamma) \ar[dr]^{r_2} \ar[dl]_{r_1} & & \Fun(\mathfrak{G}_{23}, B\Gamma) \ar[dr]^{s_3} \ar[dl]_{s_2} & \\ 
\Fun(\mathfrak{G}_1, B\Gamma) & & \Fun(\mathfrak{G}_2, B\Gamma) & & \Fun(\mathfrak{G}_3, B\Gamma) 
}
\end{equation}
\[
\Fun(\mathfrak{G}_1, B\Gamma) \stackrel{r_1 \circ \pi_1}\longleftarrow \Fun(\mathfrak{G}_{13}, B\Gamma) \stackrel{s_3 \circ \pi_3}\longrightarrow \Fun(\mathfrak{G}_3, B\Gamma).
\]
When working with connected cobordisms, for which the incoming and outgoing boundaries are subgroups, this simplifies to saying that:
Let $G_{12}$ have boundary subgroups $G_1, G_2$, and $G_{23}$ have boundary subgroups $G_2, G_3$ (where we identify the two copies of the $G_2$), denote by $G_{13}$ the amalgamation of $G_{12}$ and $G_{23}$ over $G_2$, then:
\[\Hom(G_{13},\Gamma)\cong\{(\varphi_1,\varphi_2)\ | \ \begin{array}{c}
\varphi_1\in\Hom(G_{12},\Gamma)\\
\varphi_2\in\Hom(G_{23},\Gamma)\\
\end{array}, \varphi_1(g_2)=\varphi_2(g_2)  \ \forall g_2 \in G_2 \}.\]
For simplicity we use the notation 
$P_{\Gamma}(\mathfrak{G})=\Fun(\mathfrak{G}, B\Gamma)$ for the groupoid of functors, and call them the principal $\Gamma$ bundles over $\mathfrak{G}$.
We now consider functions on the set of equivalence classes on these groupoids. 
When $\mathfrak{G}=B\mathbb{Z}_p$, since $\Gamma$ is a $p$-group we have 
\[
P_{\Gamma}(\mathfrak{G})= 
P_{\Gamma}(B\mathbb{Z}_p) \cong[\Hom(\mathbb{Z}_p, \Gamma) / \Gamma]=[\Gamma / \Gamma]
\]
and therefore 
\[\pi_0(P_{\Gamma}(B\mathbb{Z}_p))=\Gamma/\sim,\] the set of conjugacy classes of $\Gamma$.

For any set $S$ we will write $\mathcal{O}(S)$ for the complex vector space of functions from $S$ to $\mathbb{C}$. For any groupoid $\mathfrak{G}$ we write $\mathcal{O}(\mathfrak{G})$ to mean $\mathcal{O}(\pi_0(\mathfrak{G}))$.
We thus have for $\mathfrak{G}=B\mathbb{Z}_p$:
\[
Z(B\mathbb{Z}_p)=\mathcal{O}(\Gamma/\sim) \cong \mathcal{O}(\Gamma)^\Gamma\cong Z(\mathbb{C}\Gamma).
\]
the center of the group algebra of $\Gamma$.

For any two groupoids $\mathfrak{G}$ and $\mathfrak{H}$, denote by $\mu_{\mathfrak{G}, \mathfrak{H}}$ the obvious isomorphism 
\[
\mathcal{O}(P_{\Gamma}(\mathfrak{G}))\otimes_{\mathbb{C}} \mathcal{O}(P_{\Gamma}(\mathfrak{G})) \longrightarrow \mathcal{O}(P_{\Gamma}(\mathfrak{G} \coprod \mathfrak{H}))
\]
We will construct for any cospan 
\begin{equation*}\label{eqnfirstcospan}
\mathfrak{G}_1 \stackrel{R_1}\longleftarrow \mathfrak{G}_{12} \stackrel{R_2}\longrightarrow \mathfrak{G}_2
\end{equation*}
maps of the form
\begin{equation}\label{eqn:lin}
\mathcal{O}(P_{\Gamma}(\mathfrak{G}_1))\stackrel{r_{1}^*}\longrightarrow \mathcal{O}(P_{\Gamma}(\mathfrak{G}_{12}) \stackrel{r_{2*}}\longrightarrow \mathcal{O}(P_{\Gamma}(\mathfrak{G}_2)).
\end{equation}

We now define the maps $r_1^*$ and  $r_{2*}$ from Equation \ref{eqn:lin}. The map $r_1^*$ just pulls back functions via the maps of sets given by applying $\pi_0$ to Equation \ref{eqn:SpanDiag}. More explicitly, for a subgroup $G_{1} \subseteq G_{12}$ this is the pullback (of the pullback) on functions on $\Gamma$-conjugacy classes of group homomorphisms

\[\mathcal{O}(\Hom(G_{1}, \Gamma)/\sim) \longrightarrow \mathcal{O}(\Hom(G_{12}, \Gamma)/\sim).
\]
 The second map $r_{2*}$ sends $F \in \mathcal{O}(P_{\Gamma}(\mathfrak{G}_{12}))$ to the function in $\mathcal{O}(P_{\Gamma}(\mathfrak{G}_{2}))$ mapping $P \in \pi_0(P_{\Gamma}(\mathfrak{G}_{2}))$ to \[\underset{\{ \ Q \in \pi_0(P_{\Gamma}(\mathfrak{G}_{12})) \ | \ \pi_0(r_2)(Q) = P \ \}}\sum \frac{|\Aut(P)|}{|\Aut(Q)|} F(Q).\]

The composition 
\[Z(\mathfrak{G}_{12})= r_{2*} \circ r_1^*\] of these maps sends $F \in \mathcal{O}(P_{\Gamma}(\mathfrak{G}_1))$ to the function in $\mathcal{O}(P_{\Gamma}(\mathfrak{G}_2))$ mapping $P \in \pi_0(P_{\Gamma}(\mathfrak{G}_2))$ to
\[
\underset{  \{ \ Q \in \pi_0(P_{\Gamma}(\mathfrak{G}_{12}))\  | \ \pi_0(r_2)(Q) = P \ \}}\sum \frac{|\Aut(P)|}{|\Aut(Q)|} F(\pi_0(r_1)(Q)).
\]
The $\mathbb{C}-$linear map 
\[\mathcal{O}(P_{\Gamma}(\mathfrak{G}_1)) \stackrel{Z(\mathfrak{G}_{12})}\longrightarrow \mathcal{O}(P_{\Gamma}(\mathfrak{G}_2))
\]
can also be expressed as a double sum:
\[
\underset{R \in \pi_0(P_{\Gamma}(\mathfrak{G}_1))}\sum \ 
\underset{\{\ Q \in \pi_0(P_{\Gamma}(\mathfrak{G}_{12}))\ | \ \pi_0(r_2)(Q) = P \ , \ \pi_0(r_1)(Q) = R \ \}}\sum \frac{|\Aut(P)|}{|\Aut(Q)|} F(R).
\]
Given another cospan 
\begin{equation*}\label{eqnsecondcospan}
\mathfrak{H}_1 \stackrel{S_1}\longleftarrow \mathfrak{H}_{12} \stackrel{S_2}\longrightarrow \mathfrak{H}_2,
\end{equation*}
the disjoint union cospan 
\begin{equation*}\label{eqndisuncospan}
\mathfrak{G}_1 \coprod \mathfrak{H}_1 \stackrel{R_1 \coprod S_1}\longleftarrow \mathfrak{G}_{12} \coprod \mathfrak{H}_{12} \stackrel{R_2 \coprod S_2}\longrightarrow \mathfrak{G}_2 \coprod \mathfrak{H}_2,
\end{equation*}
gives rise to arrows on the bottom row  of a commutative diagram: 
\begin{footnotesize}
\[
\xymatrix@C=1.9em{\mathcal{O}(P_{\Gamma}(\mathfrak{G}_1))\otimes_{\mathbb{C}}\mathcal{O}(P_{\Gamma}(\mathfrak{H}_1))\ar[r]^{r^{*}_{1} \otimes s^{*}_{1}} \ar[d]_{\mu_{\mathfrak{G}_1, \mathfrak{H}_1}} & \mathcal{O}(P_{\Gamma}(\mathfrak{G}_{12}))\otimes_{\mathbb{C}}\mathcal{O}(P_{\Gamma}(\mathfrak{H}_{12}))\ar[r]^{r_{2*} \otimes s_{2*}} \ar[d]_{\mu_{\mathfrak{G}_{12}, \mathfrak{H}_{12}}} & \mathcal{O}(P_{\Gamma}(\mathfrak{G}_2))\otimes_{\mathbb{C}}\mathcal{O}(P_{\Gamma}(\mathfrak{H}_2)) \ar[d]_{\mu_{\mathfrak{G}_{2}, \mathfrak{H}_{2}}} \\
\mathcal{O}(P_{\Gamma}(\mathfrak{G}_1 \coprod \mathfrak{H}_1))\ar[r]  &\mathcal{O}(P_{\Gamma}(\mathfrak{G}_{12} \coprod \mathfrak{H}_{12}))\ar[r]  & \mathcal{O}(P_{\Gamma}(\mathfrak{G}_2 \coprod \mathfrak{H}_2)).
}
\]
\end{footnotesize}
Being a theory with transfer as explained in Proposition 8.2.5 of \cite{MR3970975}, we can use Diagram \ref{eqn:pullback} to see that $s_2^* \circ r_{2*} = \pi_{3*} \circ \pi^{*}_1$ and so the composition satisfies \[(s_{3*} \circ s_2^*) \circ (r_{2*} \circ r_1^*) = (s_3 \circ \pi_3)_* \circ (r_1 \circ \pi_1)^*.\] We therefore have the gluing (functoriality)  
\[
Z(\mathfrak{G}_{23}) \circ Z(\mathfrak{G}_{12}) = Z(\mathfrak{G}_{13}).
\]
If we set $Z(\mathfrak{G}_{i})=\mathcal{O}(P_{\Gamma}(\mathfrak{G}_i))$ then we have defined a functor 
\begin{equation}
Z:\Cospan(\pGrpd) \to \Proj(\mathbb{C}).
\end{equation}
It is easy to see this is the TQFT we defined at the start of this section.

\subsection{Explicit formulas for finite \texorpdfstring{$p$}{p} gauge groups}\label{subsect:DWComp}
As discussed in the previous subsection, since $\Gamma$ is a (finite) $p$-group, we have
\[
P_{\Gamma}(B\mathbb{Z}_p) \cong [\Gamma/\Gamma]
\]
and therefore \[\pi_0(P_{\Gamma}(B\mathbb{Z}_p))=\Gamma/\sim,\] the set of conjugacy classes of $\Gamma$. We also have
\[
Z(B\mathbb{Z}_p)=\mathcal{O}(\Gamma/\sim) =\mathcal{O}(\Gamma)^{\Gamma} \cong Z(\mathbb{C}\Gamma).
\]

The automorphism group of some $T \in \pi_0(P_{\Gamma}(B\mathbb{Z}_p))$ is the stabilizer subgroup of any representative $t$ of $T$ in $\Gamma$ which we write as $Z_{\Gamma}(T) =C_{\Gamma}(t)$, the centralizer of $t$ in $\Gamma$. Different representatives give conjugate and hence isomorphic subgroups.

To understand the TQFT we constructed in the previous subsection we will calculate its values on the different generators of $Cob_p^2$.
\begin{Rem}
    For the orientable generators, many of the computations here are classical. We bring here full details of these computations since we could not find a source which gave sufficient details. 
    We mainly follow the flow of the arguments in the lecture notes of Penning \cite{Penning11}, with a slight correction for the computation of the counit, and the new cases of the unoriented pro-p objects.
\end{Rem}
All generators have at most 2 incoming/outgoing boundaries, and so they all define maps between the following vector spaces:
\begin{align*}
    Z(\emptyset)\cong\text{Map}(\{1\},\mathbb{C})^\Gamma\cong\mathbb{C} &&Z(\mathbb{Z}_p)\cong \text{Map}(\Gamma,\mathbb{C})^\Gamma\cong Z(\mathbb{C}\Gamma)\end{align*}
    \begin{align*}
     Z(\mathbb{Z}_p\coprod\mathbb{Z}_p)&\cong\text{Map}(\Gamma,\mathbb{C})^\Gamma\otimes \text{Map}(\Gamma,\mathbb{C})^\Gamma\cong Z(\mathbb{C}\Gamma)\otimes Z(\mathbb{C}\Gamma) 
\end{align*}
 
\begin{enumerate}
    \item Consider first the cup $C_{1,0}:=(\{1\}; \{\mathbb{Z}_p\},\emptyset)$. By Definition \ref{Def:DWTheory} we have that:
    \[Z(C_{1,0}):Z(\mathbb{C}\Gamma)\rightarrow\mathbb{C}\]
    defined by 
    \[
        Z(C_{1,0})(f)(p_2)=\sum_{p_1\in\Hom(\mathbb{Z}_p,\Gamma)/\sim} \ \sum_{p\in P_{\{1\}}(p_1,p_2)}\frac{|\Aut(p_2)|}{|\Aut(p)|}f(p_1).
    \]
    
    But for the trivial group $P_{\{1\}}(p_1,p_2)=\emptyset$ if $p_1\neq e$ and $P_{\{1\}}(p_1,p_2)=\{1\}$ otherwise.
    and so it corresponds to the evaluation of the map to $\mathbb{C}$ at $e_\Gamma$ the identity of $\Gamma$.
    One also has that $|\Aut(p_2)|=1$ and $|\Aut(p)|=|C_\Gamma(e_\Gamma)|=|\Gamma|$ so this defines the counit:
    \begin{equation}\label{eqn:counit}
        \epsilon(\sum_{g\in \Gamma}\lambda_{g}g)=\frac{\lambda_{e_\Gamma}}{|\Gamma|}
    \end{equation}

\item For $C_{0,1}:=(\{1\}; \emptyset,\{\mathbb{Z}_p\})$, similar reasoning to the previous gives us the unit:
\[
\iota:=Z(C_{0,1}): \mathbb{C} \to Z(\mathbb{C}\Gamma).
\]
\begin{equation}\label{eqn:unit}
    \iota(\lambda)=\lambda e_\Gamma
\end{equation}

\item Next we would like to understand the multiplication comes from the pair of pants $P_{2,1}:=(\langle a,b \rangle; \{\langle a \rangle,\langle b \rangle \},\{\langle ab \rangle\})$. It will be a map:
 \[Z(P_{2,1}):Z(\mathbb{C}\Gamma)\otimes_\mathbb{C} Z(\mathbb{C}\Gamma)\rightarrow Z(\mathbb{C}\Gamma).\]
 To understand this map, we need to understand what are $P_{\langle a, b \rangle}(p_1,p_2)$ for different $p_1\in [
 (\Hom(\langle a\rangle, \Gamma) \times \Hom(\langle b\rangle, \Gamma))
 /\Gamma]$ and $p_2\in[\Hom(\langle ab \rangle,\Gamma)/\Gamma]$.
 
 We get from the definition that for a given $p_2\in[\Hom(\langle ab \rangle,\Gamma)/\Gamma]\cong[\Gamma/\Gamma]$ the set $P_{\langle a,b \rangle}(p_1,p_2)$ is empty when $p_1(a)p_1(b)\neq p_2(ab)$ and has a single element when $p_1(a)p_1(b)=p_2(ab)$.
 We get that for a fixed $p_2$ which corresponds to a conjugacy class of some $g\in \Gamma$:

\[(Z(P_{2,1})(f)([g]) =\underset{[\gamma_1], [\gamma_2]\in[\Gamma/\Gamma]}\sum
\underset{\underset{g=hk,[h]=[\gamma_1],[k]=[\gamma_2]}{(h,k) \in [\Gamma \times \Gamma/\Gamma]}}
\sum \frac{|C_\Gamma(\gamma)|}{|C_\Gamma(\{\gamma,\gamma_1,\gamma_2\})|}f([\gamma_1],[\gamma_2]) \] 
Using the orbit-stabilizer theorem we get that
\[Z(P_{2,1})(f)([\gamma])=\underset{\{(\gamma_1,\gamma_2) \in \Gamma \times \Gamma \ | \ \gamma=\gamma_1\gamma_2\}} 
\sum f(\gamma_1,\gamma_2).\]
  And so the product 
\[m:=Z(P_{2,1}):\mathcal{O}(\Gamma)^{\Gamma} \otimes_{\mathbb{C}} \mathcal{O}(\Gamma)^{\Gamma}\rightarrow \mathcal{O}(\Gamma)^{\Gamma} 
\]
is given by convolution:
\begin{equation}\label{eqn:convolution}
    m(a,b) \ (g) = \underset{\{(h,k) \in \Gamma \times  \Gamma \ | \ g=hk\}}\sum a(h)b(k).
\end{equation}
If we think of this in terms of a product on $Z(\mathbb{C}\Gamma)$ we can write it as
\begin{equation}\label{eqn:mult.}
    m(\sum_{g_1\in \Gamma}\lambda_{g_1}g_1\otimes\sum_{g_2\in \Gamma}\mu_{g_2}g_2)=\sum_{g\in \Gamma}(\sum_{{h\in\Gamma}}\lambda_{h}\mu_{gh^{-1}})g.
\end{equation}

\item 
    We now turn to compute the co-multiplication, coming from $P_{1,2}:=(\langle a,b \rangle;\{\langle ab \rangle\}, \{\langle a \rangle,\langle b \rangle \}).$
    By (R5) and (R1) we have that:
    \[P_{1,2}=P_{1,2}\circ P_{2,1}\circ(C_{0,1}\coprod C_{1,1})=(C_{1,1}\coprod P_{2,1})\circ (P_{1,2}\coprod C_{1,1})\circ(C_{0,1}\coprod C_{1,1}).\] 
    This corresponds geometrically to:
    
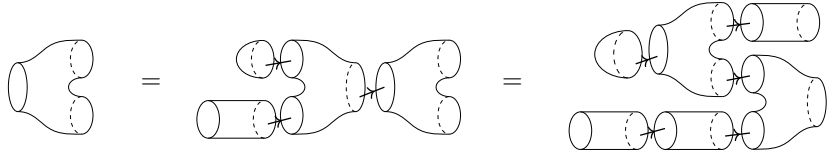
\begin{center} 
\begin{tikzpicture}[outer sep=auto, decoration={
    markings,
    mark=at position 0.6 with { \arrow{>[length=0.8mm]}}}]
 \matrix (name) [matrix of nodes, column sep={1.2cm,between origins}, row sep={1.1cm,between origins}]{ |[xshift=-0.5,yshift=-0.75cm]|{\scalebox{0.5}\PantsR} &&
    |[xshift=0.3cm]| {\scalebox{0.5}\cupB} &|[yshift=-0.75cm]|{\scalebox{0.5}\PantsL} &  |[xshift=0.07cm,yshift=-0.75cm]|{\scalebox{0.5}\PantsR} & &  |[xshift=0.33cm,yshift=0cm]| {\scalebox{0.63}\cupB} &|[xshift=0.1cm, yshift=-0.25cm]|{\scalebox{0.5}\PantsR} &  |[xshift=0.08cm,yshift=0.49cm]|{\scalebox{0.5}\cyl} &  \\ [-0.3cm]
   & &|[xshift=0.05cm]|{\scalebox{0.5}\cyl} & & & & |[xshift=0.2cm,yshift=-0.15cm]|{\scalebox{0.5}\cyl} &|[xshift=0.13cm,yshift=-0.15cm]|{\scalebox{0.5}\cyl}&|[xshift=0.13cm,yshift=-0.15cm]|{\scalebox{0.5}\PantsL}\\  };
   \draw (-3.7,-0.1) node{$=$}; 
    \draw (1.1,-0.1) node{$=$}; 
   \draw[line width=0.5pt, postaction={decorate}] (-2.18,0.14)   -- (-1.84,0.2);
    \draw[line width=0.5pt, postaction={decorate}] (-2.15,-0.64)   -- (-1.84,-0.56);
   \draw[line width=0.5pt, postaction={decorate}] (-0.93,-0.28)   -- (-0.6,-0.15);
   \draw[line width=0.5pt, postaction={decorate}] (2.74,0.16)   -- (3.03,0.26);
    \draw[line width=0.5pt, postaction={decorate}] (2.8,-0.79)   -- (3.13,-0.7);
    \draw[line width=0.5pt, postaction={decorate}] (3.95,0.66)   -- (4.26,0.7);
   \draw[line width=0.5pt, postaction={decorate}] (3.95,-0.06)   -- (4.26,0);
   \draw[line width=0.5pt, postaction={decorate}] (3.93,-0.8)   -- (4.26,-0.76);
\end{tikzpicture}
\end{center}

    In other words, if we define $P_{0,2}:=P_{1,2}\circ C_{0,1}$ then as in Equation \ref{eqn:FrobCons2} we have:
    \[Z(P_{1,2})=(\id\otimes_\mathbb{C}m)(Z(P_{0,2})(1)\otimes_\mathbb{C}\id).\]
    So to understand the co-multiplication, we will first calculate the co-pairing map, which comes from $P_{0,2}$.
    From Equations \ref{eqn:counit},\ref{eqn:mult.} we easily get the value of the pairing map coming from $P_{2,0}:=C_{1,0}\circ P_{2,1}$, which we denote by $\langle-,-\rangle$:
    \[Z(P_{2,0}):Z(\mathbb{C}\Gamma)\otimes_\mathbb{C}Z(\mathbb{C}\Gamma)\rightarrow\mathbb{C}\]
    \begin{equation}\label{eqn:pairing}
        \langle\sum_{g_1\in \Gamma}\lambda_{g_1}g_1,\sum_{g_2\in \Gamma}\mu_{g_2}g_2\rangle=\frac{1}{|\Gamma|}\sum_{g\in \Gamma}\lambda_g\mu_{g^{-1}}
    \end{equation}
    To calculate the value of $Z(P_{0,2})$ we have the following snake relation which follows from (R5), (R2) and (R1) (or simply by a direct computation of the invariant of the cobordism):
    \[C_{1,1}=(C_{1,1}\coprod P_{2,0})\circ(P_{0,2}\coprod C_{1,1})\]
    which geometrically is:
    
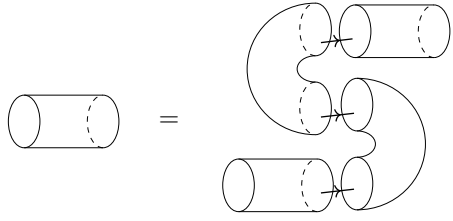
\begin{center}
    \begin{tikzpicture}[outer sep=auto, decoration={
    markings,
    mark=at position 0.6 with { \arrow{>[length=0.8mm]}}}]
    \draw (0,1) node{\scalebox{0.7}\CoPairing};
 \draw (-0.15,-0.55) node{\scalebox{0.7}\cyl};
 \draw (1.25,0) node{\scalebox{0.7}\Pairing};
 \draw (1.4,1.5) node{\scalebox{0.7}\cyl};
  \draw (-3,0.3) node{\scalebox{0.7}\cyl};
   \draw (-1.6,0.3) node{$=$}; 
    \draw[line width=0.5pt, postaction={decorate}] (0.42,1.35)   -- (0.84,1.4);
   \draw[line width=0.5pt, postaction={decorate}] (0.42,0.36)   -- (0.86,0.41);
   \draw[line width=0.5pt, postaction={decorate}] (0.42,-0.65)   -- (0.86,-0.6);
\end{tikzpicture}
\end{center}
    Now from the representation theory of finite groups, we know that the irreducible characters of $\Gamma$ are an orthonormal basis for $Z(\mathbb{C}\Gamma)$ for the pairing in Equation \ref{eqn:pairing}.
    Denote the set of irreducible representations of $\Gamma$ by $\Irr(\Gamma)$, and for a representation $\rho$ of $\Gamma$, denote its character by $\chi_\rho$
    In other words we get that for the snake relation to hold one has that the co-pairing is:
    \begin{equation}\label{eqn:copairing}
        Z(P_{0,2})=\sum_{\rho\in\Irr(\Gamma)}\sum_{g,h\in\Gamma}\chi_\rho(g)g\otimes\chi_\rho(h)h
    \end{equation}
    Combining all of the above we get that the co-multiplication map is:
    \begin{equation}\label{eqn:coproduct}
        Z(P_{1,2})(\sum_{k\in\Gamma}\lambda_k k)=\sum_{\rho\in\Irr(\Gamma)}\sum_{g,h,k\in\Gamma}\chi_\rho(g)\chi_\rho(h^{-1})g\otimes\lambda_{kh}k
    \end{equation}
   For an irreducible representation $\rho$, let $b_\rho:=\underset{g\in\Gamma}\sum\chi_\rho(g)g$.
   Let us calculate the co-multiplication on each of the $b_\rho$:
     \begin{equation*}
        Z(P_{1,2})(b_{\rho_i})=\sum_{\rho\in\Irr(\Gamma)}(\sum_{g\in\Gamma}\chi_\rho(g)g)\otimes \sum_{k\in\Gamma}(\sum_{h\in\Gamma}\chi_\rho(h)\chi_{\rho_i}(kh^{-1}))k
    \end{equation*}
    But $\underset{h\in\Gamma}\sum \chi_\rho(h)\chi_{\rho_i}(kh^{-1})$ equals $\frac{|\Gamma|}{\dim(\rho_i)}\chi_{\rho_i}(k)$ if $\rho=\rho_i$ and is $0$ otherwise.
    Thus:
    \begin{equation}\label{eqn:comultBasis}
        Z(P_{1,2})(b_{\rho_i})=\frac{|\Gamma|}{\dim(\rho_i)}\sum_{g\in\Gamma}\chi_{\rho_i}(g)g\otimes b_{\rho_i}
    \end{equation}
\item 
    Finally, we have the twisting maps, $\phi_{\alpha}= Z(C^\alpha_{1,1})$: \[
\phi_{\alpha}: \mathcal{O}(\Gamma)^{\Gamma} \to \mathcal{O}(\Gamma)^{\Gamma}
\]
Recall that , and so from the definition of $P_G(p_1,p_2)$ it is empty if $p_1(a)\neq p_2(a^{-\alpha})$, and has a unique element if $p_1(a)\neq p_2(a^{-\alpha})$, and so we get:
\begin{equation}
(\phi_{\alpha} F)(g) = F(g^{\alpha^{-1}})    
\end{equation}
for any $g\in \Gamma$.
\end{enumerate}

Using the above formulas for the generators we can easily compute the values on $T_{1,1}$ and $T_{1,1}^r$ (by using relation (R10)).
First since $T_{1,1}=P_{2,1}\circ P_{1,2}$ and by using similar reasoning to the way we computed equation \ref{eqn:comultBasis} we get that:
\begin{equation}
    Z(T_{1,1})(b_{\rho_i})=\left(\frac{|\Gamma|}{\dim(\rho_i)}\right)^2b_{\rho_i}.
\end{equation}

Next, recall that 
\[T^r_{1,1}=(\langle s, x, y \rangle,\{\langle s \rangle\},\{\langle x^{p^r}[x,y]s \rangle\}).\] 
Taking $\alpha:=1+\underset{n\geq1}\sum (p^r)^n$, relation (R10) gives us:
\[T_{1,1}^r=P_{2,1}\circ(C_{1,1}^{\alpha}\coprod C_{1,1})\circ P_{1,2}\]
We have that \[Z(T_{1,1}^r):Z(\mathbb{C}\Gamma)\rightarrow Z(\mathbb{C}\Gamma)\] is the map 
\begin{equation}
\begin{split}
    Z(T_{1,1}^r)(b_{\rho_i}) & =Z(P_{2,1})\left(\frac{|\Gamma|}{\dim(\rho_i)}\sum_{g\in\Gamma}\chi_{\rho_i}(g^{1-p^r})g\otimes b_{\rho_i}\right) \\ & =\frac{|\Gamma|}{\dim(\rho_i)}\sum_{g\in \Gamma}\left(\sum_{h\in\Gamma}\chi_{\rho_i}\left((gh^{-1})^{1-p^r}\right)\chi_{\rho_i}(h)\right)g
    \end{split}
\end{equation}
Where we used that $\alpha^{-1}=1-p^r$.
Now let 
\[
G_{n,r}=\langle \ x_1,y_1,...,x_n,y_n  | \ x_1^{p^r}[x_1,y_1]\cdots[x_n,y_n] \  \rangle
\]
be the pro-$p$ group without boundary subgroups from the statement of Theorem \ref{Thm:WilkesPDn} where $r\in\mathbb{N}_{\geq 0}\cup\{\infty\}$ if $n\neq 0$, and $r=\infty$ when $n=0$.
On the one hand, the usual argument shows that if $G_{n,r}$ is any pro-$p$ $\PD^2$ group then as a map from $\mathbb{C}$ to itself we have 
\begin{equation}
\label{eqn:fin}
Z(BG_{n,r})(1)= \frac{|\Hom(G_{n,r},\Gamma)|}{|\Gamma|}.
\end{equation}
 on the other hand, by the normal form of $G_{n,r}$ from Theorem \ref{Thm:generators}, we get that $G_{n,r}$ is given by \[C_{1,0}\circ T_{1,1}^r\circ \underset{(n-1)\text{ times}}{T_{1,1}\circ\dots\circ T_{1,1}}\circ C_{0,1}.\]
 One has from the Schur orthogonality relations that:
 \begin{equation}     \sum_{\rho\in\Irr(\Gamma)}\dim(\rho)b_\rho=\sum_{g\in\Gamma}\left(\sum_{\rho\in\Irr(\Gamma)}\chi_{\rho}(g)\chi_\rho(e_\Gamma)\right)g=\sum_{\rho\in\Irr(\Gamma)}\chi_\rho(e_\Gamma)^2e_\Gamma=|\Gamma|e_\Gamma
 \end{equation}
 
 And so finally, we get that get that:

\begin{equation}
\begin{split}
Z(BG_{n,r})(1)&= Z(C_{1,0}\circ T_{1,1}^r)\circ Z(T_{1,1})^{n-1}(e_\Gamma) \\
&=Z(C_{1,0}\circ T_{1,1}^r)\circ Z(T_{1,1})^{n-1}\left( \frac{1}{|\Gamma|}\sum_{\rho\in\Irr(\Gamma)}\dim(\rho)b_\rho \right)\\
&=Z(C_{1,0}\circ T_{1,1}^r)\circ Z(T_{1,1})^{n-1}\left( \frac{1}{|\Gamma|}\sum_{\rho\in\Irr(\Gamma)}\dim(\rho)b_\rho\right)
\\
&=Z(C_{1,0}\circ T_{1,1}^r)\left(\sum_{\rho\in\Irr(\Gamma)}\left(\frac{|\Gamma|}{\dim(\rho)}\right)^{2n-3}b_\rho \right)\\
&=Z(C_{1,0})\left(\sum_{\rho\in\Irr(\Gamma)}\left(\frac{|\Gamma|}{\dim(\rho)}\right)^{2n-2}\sum_{g\in \Gamma}\left(\sum_{h\in\Gamma}\chi_{\rho}\left(\left(gh^{-1}\right)^{1-p^r}\right)\chi_{\rho}(h)\right)g \right)\\
&=|\Gamma|^{2n-3} \underset{\rho \in \Irr(\Gamma)}\sum \frac{1}{\dim(\rho)^{2n-2}} \underset{g\in \Gamma}\sum \chi_\rho(g^{p^r-1})\chi_\rho(g).
\end{split}
\end{equation}

Overall using Equation \ref{eqn:fin} we get the following:
\begin{Thm}\label{Thm:HomCount}
    Let $G_{n,r}$ be a Demushkin group with $2n$ generators and orientability level $r$, and let $\Gamma$ be a finite $p$-group. We have the following formula:
    \begin{equation}\label{eqn:HomCount}
        |\Hom(G_{n,r},\Gamma)|=|\Gamma|^{2n-2} \underset{\rho \in \Irr(\Gamma)}\sum \frac{1}{\dim(\rho)^{2n-2}} \underset{g\in \Gamma}\sum \chi_\rho(g^{p^r-1})\chi_\rho(g).
    \end{equation}
\end{Thm}

In the special case of a $p$-adic field $K$, containing $p$-th roots of unity, we re-derive Yamagishi's formula \cite{Yama}:
\begin{Cor}
    Let $K$ be a $p$-adic field, with degree $[K:\mathbb{Q}_p]=n$, and $\mu_{p^r}\subseteq K, \mu_{p^{r+1}}\not\subseteq K$, we that $G_K(p)\cong G_{\frac{n}{2}+1,r}$.
Therefore we get as in 
\begin{equation}\label{eqn:Yama}
|\Hom(G_K(p),\Gamma)|=|\Gamma|^{n} \underset{\rho \in \Irr(\Gamma)}\sum \frac{1}{\dim(\rho)^{n}} \underset{g\in \Gamma}\sum \chi_\rho(g^{p^r-1})\chi_\rho(g).
\end{equation}
where $\Irr(\Gamma)$ is the set of isomorphism classes of irreducible finite dimensional complex representations of $\Gamma$.
\end{Cor}

Thus, as explained in \cite{Yama}, we can count the number of extensions $L$ of a field $K$, with $\Gal(L/K)\cong\Gamma$ by
\[\frac{1}{|\Aut(\Gamma)|}\sum_{H\leq\Gamma}\mu(H)|\Hom(G_{K},H)|.\]

The above equation comes from Hall's inversion formula \cite{hall1936eulerian}. Hall's formula is a formula for finding the number of normal subgroups N of a finitely generated group G with a given finite quotient $G/N\cong\Gamma$, by a group theoretic version of the M\"{o}bius function $\mu$.
\subsubsection{Computations with general gauge groups}
We finish by noting that one can use the TQFT above to compute the homotopy cardinality of the stack of $H$ bundles on the classifying space of $G_{n,r}$ for more general groups $H$.
That is one looks at a fixed $p$-Sylow, subgroup of $H$ and counts homomorphisms to it using Theorem \ref{Thm:HomCount}.
One then multiplies the result by the amount of $p$-Sylow subgroups, and then subtracts homomorphisms which were counted twice from double intersections, adds back homomorphisms in triple intersections etc. 
Alternatively, one uses the Hall inversion formula to count only surjective homomorphisms, and applies it to every conjugacy class of subgroup of the $p$-Sylow, and then multiplies by the size of orbit under conjugation.
The above method is of course complicated in general, but using known character tables nice $p$ groups these are computable.
For example in cases where one has that all $p$-Sylow subgroups pairwise intersect at the identity, the computations reduce to knowing how many $p$-Sylow subgroups one has and applying Theorem \ref{Thm:HomCount}.\\
E.g. in the simplest example of $H=GL_{2}(\mathbb{F}_p)$, there are $p+1$ $p$-Sylow subgroups which are cyclic groups on $p$ elements. Denote one such $p$-Sylow by $S$. The group $S$ has $p$ irreducible representation, all of dimension $1$, and so equation \ref{eqn:HomCount} is simply:
\begin{equation*}
    \begin{split}
        |\Hom(G_{n,S},S)|&=|S|^{2n-2} \underset{\rho \in \Irr(S)}\sum \frac{1}{\dim(\rho)^{2n-2}} \underset{g\in S}\sum \chi_\rho(g^{p^r-1})\chi_\rho(g)\\
&=|S|^{2n-2}\underset{\rho \in \Irr(S)}\sum \frac{1}{1^{2n-2}} \underset{g\in S}\sum \chi_\rho(g^{-1})\chi_\rho(g)\\
&=|S|^{2n-2}\underset{\rho \in \Irr(S)}\sum|S|=p^{2n}.
    \end{split}
\end{equation*}
Now since every homomorphism lands in some Sylow, and its image is the whole Sylow or trivial, overall the homotopy cardinality of the stack of $GL_{2}(\mathbb{F}_p)$ bundles on the classifying space of $G_{n,r}$ is
\[
\frac{|\Hom(G_{n,r}, GL_{2}(\mathbb{F}_p))|}{|GL_{2}(\mathbb{F}_p)|}=\frac{(p+1)p^{2n}-(p+1)+1}{p(p-1)^2(p+1)}=\frac{p^{2n}+p^{2n-1}-1}{(p-1)^2(p+1)}.\]
\begin{Rem}
    The above example is of course very simple and can be easily computed by hand, in more complicated cases, we think the above method outlined will help simplify the computations. We hope to further explore this direction in future work.
\end{Rem}
\newpage
\appendix
\section{Appendix}\label{Appendix}
\subsection{The categories of groups and groupoids}\label{App:Cats}
The $1$-category of (small) pro-$p$ groupoids $\mathfrak{G}$, and functors $\Hom(\mathfrak{G}, \mathfrak{H})$ between them, equipped with the obvious cartesian product $\prod$ is a cartesian closed category whose internal hom we write $\Fun$. The functor 
\[\pGrp \longrightarrow \pGrpd
\]
maps every pro-$p$ group $G$ to $BG$, the category with one object and a group $G$ as automorphisms. 
For any $G\in \pGrp$ and for any $\mathfrak{H} \in \pGrpd$ there is a non-canonical equivalence of categories 
\[
\Fun(BG, \mathfrak{H}) \cong \underset{x \in \pi_0(\mathfrak{H})}\coprod [\Hom(G, \Aut(x))/\Aut(x)].
\]
The initial pro-$p$ groupoid will be written as $\emptyset$ and the final pro-$p$ groupoid as $\bullet$.
The functor $B$ has an explicit left adjoint $U$ described in Example 1 on page 319 of \cite{TG}. It's easy to see that $U \circ B$ is isomorphic to the identity functor. The category of pro-$p$ groupoids is actually a cartesian closed model category. Here, the weak equivalences are equivalences as categories and the cofibrations are functors which are injective on objects. This model category also happens to be  combinatorial, proper, and simplicial. The coproduct $\coprod$ of pro-$p$ groupoids equips this category with the structure of a symmetric monoidal category. Notice that any pro-$p$ groupoid $\mathfrak{H}$ is non-canonically equivalent to a coproduct of pro-$p$ groupoids of the form $BH_i$ and so for any pro-$p$ group $G$, $\Fun(\mathfrak{H}, BG)$ is equivalent to a product of categories $C_i$ whose objects are continuous group homomorphisms $H_i \to G$ and whose morphisms are elements of $G$ conjugating one morphism to another. 
 
Constructing limits in the category of pro-$p$ groupoids is very easy, one just takes the limit of the objects and the limit of the morphisms. The existence of $U$, the left adjoint of $B$ shows that $B$ commutes with limits. Colimits of pro-$p$ groupoids are more difficult to define. According to \cite{MR0210125}, to construct the colimit of a functor 
\[
F: I \to \pGrpd 
\]
we should first define $S=\underset{i\in I}\colim \  \ob F(i)$ and define $A(s,t)$ to be the set of paths from $s$ to $t$ built out of edges in $E(x,y)=\underset{i\in I}\colim \Hom_{F(i)}(x, y)$ for $x,y \in S$. In particular we have the paths on no edges given by $\text{id}_{E,x} \in E(x,x)$. We have maps of sets 
\[f_i: \ob (F(i)) \to S
\]
For any $x,y \in \ob(F(i))$ there is a map of sets 
\[
m_i(x,y): \Hom_{F(i)}(x, y) \to E(f_i(x),f_i(y)).
\]
Consider the intersection $\sim$ of all equivalence relations $\sim_R$ on $\underset{s,t}\coprod A(s,t)$ which preserve domain and range and satisfy: \[m_i(x,x)(\id_x)\sim_R \id_{E,m_i(x)}, \]
\[m_i(y,z)(f)\circ m_i(x,y)(g)\sim_R m_i(x,z)(f\circ g),\]
and for formal composition of arrows $f\sim_R g$ if and only for all $h,k$ we have if $k \circ f\circ h \sim_R k \circ g \circ h$. Finally, the colimit is the pro-$p$ groupoid $\mathfrak{F}$ with $\ob(\mathfrak{F})=S$ and $\Hom_{\mathfrak{F}}(s,t)=A(s,t)/\sim$ and the obvious functors $F(i)\to \mathfrak{F}$. We claim that $B$ also commutes with connected colimits, not only limits. For any pro-$p$ groupoid $\mathfrak{H}$ and any functor from a {\it connected} small category $I$ to pro-$p$ groups sending $i$ to $G_i$ we have 
\[
\Fun(B(\underset{i\in I}\colim   G_i), \mathfrak{H}) \cong \Fun(\underset{i\in I}\colim B G_i, \mathfrak{H})
\]
for any $\mathfrak{H}$. Therefore for any connected diagram of pro-$p$ groups defined by a functor from a connected category to the category $\pGrp$ we have 
\[
B(\underset{i\in I}\colim   G_i) \cong \underset{i\in I}\colim B G_i.
\]
The category $\pGrp$ has pushouts, given a diagram 
\[
H_1 \longleftarrow G
\longrightarrow 
H_2 
\]
we denote the pushout in this category by $H_1 *_{G} H_2$.

The category $\pGrpd$ also has pushouts and given a diagram of pro-$p$ groups 
\[
H_1 \longleftarrow G
\longrightarrow 
H_2 
\]
we have,
\[BH_1 \coprod_{BG} BH_2=B(H_1 *_{G} H_2).\]
Due to the fact that the left hand side is the colimit of a diagram of cofibrations, we in fact can say 
\begin{equation}
BH_1 \coprod^{h}_{BG} BH_2=B(H_1 *_{G} H_2).
\end{equation}
Consider now a pair of group monomorphisms $f:C \to A$ and $g:C \to A$. The colimit of the diagram $BC \rightrightarrows BA$ is just $B(A/(N\{f(c)g(c)^{-1} \ | \ c\in C\}))$. On the other hand the homotopy colimit of $BC \rightrightarrows BA$ is the same as the homotopy colimit of the diagram 
\[BC \stackrel{\triangledown}\longleftarrow BC \coprod BC \stackrel{(Bf, Bg)}\longrightarrow BA
.
\]
The homotopy colimit of this diagram is not the same as the colimit which is not surprising as the maps are not cofibrations. However, there is a standard factorization of $\triangledown = w \circ \tilde{\triangledown}$ where $w$ is an equivalence of categories and 
\[ BC \coprod BC = BC \prod \{0,1\} \stackrel{\tilde{\triangledown}}\longrightarrow BC \times \mathfrak{I}
\]
is induced the inclusion of the objects $0,1$ with only identity morphisms into the groupoid $\mathfrak{I}$ which just has these two objects, their identity morphisms and also has a morphism from $0$ to $1$ and its inverse. The homotopy colimit is then the ordinary colimit of the diagram 
\[BC \times \mathfrak{I} \stackrel{\tilde{\triangledown}}\longleftarrow BC \coprod BC \stackrel{(Bf, Bg)}\longrightarrow BA
.
\]
According to page 337 of \cite{TG} this pro-$p$ groupoid can be written in terms of the HNN extension described in Definition \ref{Def: pro-p HNN} as $B(A*_{C})$ and so 
\begin{equation}
\hcolim(BC \rightrightarrows BA) \cong B(A*_{C}).
\end{equation}
The category of compactly generated weak Hausdorff spaces is a complete and cocomplete cartesian closed model category using the Quillen model structure where the weak equivalences are homotopy equivalences, and the fibrations are Serre fibrations. It is a proper model category. The fundamental groupoid functor 
\[CGWH \longrightarrow
\Grpd
\]
written on objects as $X \mapsto \pi_{\leq 1}(X)$ where the objects of $\pi_{\leq 1}(X)$ are points of $X$ and the morphisms  of $\pi_{\leq 1}(X)$ are paths. Since cofibrations in CGWH are necessarily injective, $\pi_{\leq 1}$ takes cofibrations to cofibrations and in fact $\pi_{\leq 1}$ is part of a Quillen adjunction. The left adjoint of $\pi_{\leq 1}$ associates to a groupoid $\mathfrak{G}$, its Eilenberg-Mac Lane space $|\mathfrak{G}|$ which has a vertex for each object, an edge for each morphism, a triangle for each composable morphism, and so on.
\subsection{The symmetric monoidal category of cospans of pro-\texorpdfstring{$p$}{p} groupoids}\label{App:Cospan}

In this subsection, we consider how our cobordism categories can be described using groupoids as opposed to groups. We show that the category $\Cob_p^2$ defined in Section \ref{subsect:Cob} is equivalent to a full symmetric monoidal subcategory of the category of cospans of pro-$p$ groupoids $\Cospan(\pGrpd)$. We also explain why the groupoid of symmetric monoidal functors from $\Cob_p^2$ to $\Proj(R)$ is equivalent to the groupoid $\TQFT^{2}_{p}(R)$ defined in Section \ref{subsect:TQFT}. Combining these gives us a functor 
\[
\Fun^{\otimes}(\Cospan(\pGrpd), \Proj(R))\longrightarrow \Fun^{\otimes}(\Cob_p^2, \Proj(R))\cong \TQFT^{2}_{p}(R).
\]
This gives a convenient way of defining objects of $\TQFT^{2}_{p}(R)$. In fact, the example in section is an object in $\TQFT^{2}_{p}(\mathbb{C})$ which is defined this way.
In order to describe the cobordism category we need to start with defining a category of cospans of pro-$p$ groupoids before placing restrictions on what kind of pro-$p$ groupoids and maps appear in the cospans. Consider the category of cospans of pro-$p$ groupoids: the objects are pro-$p$ groupoids. The morphisms $\wHom(\mathfrak{G},\mathfrak{H})$ in this category  are equivalence classes of diagrams of functors 
\[
\mathfrak{G} \stackrel{f}\longrightarrow \mathfrak{K} \stackrel{d}\longleftarrow \mathfrak{H}.
\]
 We sometimes refer to this cospan as $C(\mathfrak{G}, f, \mathfrak{K},d, \mathfrak{H})$. Here, diagrams $\mathfrak{G} \rightarrow \mathfrak{K} \leftarrow \mathfrak{H}$ and $\mathfrak{G} \rightarrow \mathfrak{L} \leftarrow \mathfrak{H}$ are equivalent if there exists an equivalence of categories $\mathfrak{K} \to \mathfrak{L}$ making the diagrams
 \[
 \xymatrix@C=.57em{
  & & \mathfrak{K} \ar[dd]  \\
\mathfrak{G} \ar[urr] \ar[drr]_{}
 & & & &\mathfrak{H}  \ar[ull]_{}\ar[dll]^{}
 \\
  &&  \mathfrak{L}}\]
 commute up to natural isomorphism. If we consider the example that $S,T,U$ are finite sets and $\mathfrak{G}=\underset{s\in S}\coprod B G_{s}$, $\mathfrak{K}=\underset{t\in T}\coprod B K_{t}$
$\mathfrak{H}=\underset{u\in U}\coprod B H_{u}$ then $\wHom(\mathfrak{G},\mathfrak{H})$ consists of equivalence classes 
\[
\wHom(\mathfrak{G},\mathfrak{H})=\{
\mu:S\to T, \nu:U \to T, f_s:G_s\to K_{\mu(s)}, d_u:H_u \to K_{\nu(u)}\}/\sim,
\]
where collections
\[
\{
\mu:S\to T, \nu:U \to T, f_s:G_s\to K_{\mu(s)}, d_u:H_u \to K_{\nu(u)}\}
\]
and 
\[
\{
\mu':S\to T', \nu':U \to T', f'_s:G_s\to L_{\mu'(s)}, d'_u:H_u \to L_{\nu'(u)}\}
\]
are equivalent if there exist a collection of isomorphisms 
\[
\{\gamma:T\to T', m_t:K_t \to L_{\gamma(t)} \}
\]
and a collection of elements $l_{v} \in \mathbb{L}_{v}$ such that $\gamma \circ \mu =\mu', \ \gamma \circ \nu =\nu'$ and using $C(l)$ to denote conjugation by $l$, for every $u\in U, s\in S$ we have 

\[\ m_{\mu(s)}\circ f_s=C(l_{\mu'(s)}) \circ f'_{s}, \ m_{\nu(u)} \circ d_u = C(l_{\nu'(u)}) \circ d'_{u}.
\]
The composition of a cospan $D$ given by 
\[
\mathfrak{G}_1 \longrightarrow \mathfrak{G}_{12} \longleftarrow \mathfrak{G}_2
\]
and a cospan $C$ given by 
\[
\mathfrak{G}_2 \longrightarrow \mathfrak{G}_{23} \longleftarrow \mathfrak{G}_3
\]
is the cospan $C\circ D$ given by 
\[
\mathfrak{G}_1 \longrightarrow \mathfrak{G}_{12} \underset{\mathfrak{G}_2}\coprod \mathfrak{G}_{23} \longleftarrow \mathfrak{G}_3.
\]
The associativity for composition of morphisms uses the fact that given the above diagrams as well as the cospan $E$ given by 
\[
\mathfrak{G}_3 \longrightarrow \mathfrak{G}_{34} \longleftarrow \mathfrak{G}_4
\]
there exists an equivalence of categories 
\[
(\mathfrak{G}_{12} \underset{\mathfrak{G}_2}\coprod \mathfrak{G}_{23})\coprod_{\mathfrak{G}_3}\mathfrak{G}_{34} \cong \mathfrak{G}_{12} \underset{\mathfrak{G}_2}\coprod (\mathfrak{G}_{23}\coprod_{\mathfrak{G}_3}\mathfrak{G}_{34} )
\]
making the required functor diagrams commute up to natural isomorphism. Therefore \[(E \circ D) \circ C = E \circ (D \circ C).\]  The identity in $\wHom(\mathfrak{G},\mathfrak{G})$ is given by 
\begin{equation}\label{eqn:identity}
\mathfrak{G} \stackrel{\text{id}}\longrightarrow \mathfrak{G} \stackrel{\text{id}}\longleftarrow \mathfrak{G}
\end{equation}
where both maps are the identity from \ref{TT} Axiom T6.
We define a symmetric monoidal structure on this category. On objects it is just given by the coproduct of groupoids. On morphisms it is given by $C\coprod D$ defined by 
\begin{equation}\label{eqn:coprodmors}
    \begin{split}
& (\mathfrak{G}_1 \rightarrow \mathfrak{G}_{12} \leftarrow \mathfrak{G}_2)\coprod(\mathfrak{H}_1 \rightarrow \mathfrak{H}_{12} \leftarrow \mathfrak{H}_2) \\ & =(\mathfrak{G}_1 \coprod \mathfrak{H}_1) \rightarrow (\mathfrak{G}_{12} \coprod \mathfrak{H}_{12})\leftarrow (\mathfrak{G}_2 \coprod \mathfrak{H}_2) .    
    \end{split}
\end{equation}
The symmetry morphism is the obvious map $C\coprod D\to D\coprod C$ and the associators are also obvious. The monoidal unit is given by cospan \[\emptyset \longrightarrow \emptyset \longleftarrow \emptyset.
\] Consider the full monoidal subcategory of cospans whose objects are cospans of the form 
\begin{equation}\label{eqn:simplecob}
\mathfrak{G} \stackrel{g}\longrightarrow \mathfrak{H} \stackrel{\text{id}}\longleftarrow \mathfrak{H}.
\end{equation}
where $g$ is an equivalence of categories. We sometimes write this cobordism as $C(\mathfrak{G}, g, \mathfrak{H})$. It is easy to see by considering 
\[
\mathfrak{G} \stackrel{d}\rightarrow \mathfrak{E} \stackrel{\text{id}}\leftarrow 
\mathfrak{E}\stackrel{f}\rightarrow  \mathfrak{K} \stackrel{c}\leftarrow \mathfrak{H} \stackrel{b}\leftarrow \mathfrak{L} \stackrel{\text{id}}\leftarrow \mathfrak{L} 
\]
that
\begin{equation}\label{eqn:manycomp}
C(\mathfrak{H}, b, \mathfrak{L}) \circ C(\mathfrak{G}, f\circ d, \mathfrak{K},c, \mathfrak{H}) = C(\mathfrak{E}, f, \mathfrak{K}, b \circ c, \mathfrak{L}) \circ C(\mathfrak{G}, d, \mathfrak{E}).
\end{equation}
This equation shows that axiom T4 in \ref{TT} is a consequence of having a functoriality. If we consider an equivalence $b$ and consider 
\[
\mathfrak{E}\stackrel{f}\rightarrow  \mathfrak{K} \stackrel{c}\leftarrow \mathfrak{H} \stackrel{b}\rightarrow \mathfrak{G} \stackrel{\text{id}}\leftarrow \mathfrak{G} \stackrel{h}\rightarrow  \mathfrak{L} \stackrel{g}\leftarrow \mathfrak{M}
\]
then it is easy to see that if we let $\mathfrak{T}=\mathfrak{K} \coprod_{\mathfrak{H}}\mathfrak{L} \cong \mathfrak{K} \coprod_{\mathfrak{G}}\mathfrak{L} $ we have
\begin{equation}\label{eqn:otherthing}
C(\mathfrak{G}, h, \mathfrak{L}, g, \mathfrak{M}) \circ C(\mathfrak{H}, b, \mathfrak{G}) \circ C(\mathfrak{E}, f, \mathfrak{K}, c, \mathfrak{H})=C(\mathfrak{E}, f, \mathfrak{T}, g, \mathfrak{M}).
\end{equation}
This equation shows that axiom T7 in \ref{TT} is a consequence of having a symmetric monoidal functor.
A typical example for a group $G$ is
\begin{equation}\label{eqn:switch}
BG \coprod BG\stackrel{\sigma}\longrightarrow BG \coprod BG\stackrel{\text{id}}\longleftarrow BG \coprod BG
\end{equation}
where $\sigma$ is the switch map.
Another important example is for any automorphism $a$ of a group $G$ we have the cospan 
\begin{equation}\label{eqn:aut}
BG \stackrel{Ba}\longrightarrow BG  \stackrel{\text{id}}\longleftarrow BG .
\end{equation}

Notice that if we post-compose the cospan in Equation \ref{eqn:simplecob} with the cospan 
\[
\mathfrak{H} \stackrel{f}\longrightarrow \mathfrak{K} \stackrel{\text{id}}\longleftarrow \mathfrak{K}.
\]
we get a cobordism equivalent to 
\[
\mathfrak{G} \stackrel{f \circ g}\longrightarrow \mathfrak{K} \stackrel{\text{id}}\longleftarrow \mathfrak{K}.
\]
In other words $C(\mathfrak{H}, f, \mathfrak{K}) \circ C(\mathfrak{G}, g, \mathfrak{H})= C(\mathfrak{G}, f \circ g, \mathfrak{K})$, corresponding to \ref{TT} axiom T3. Therefore the monoidal category whose objects are groupoids and whose morphisms are (natural isomorphism) equivalence classes of functors has a symmetric monoidal functor into the category of groupoid cospans. 
A symmetric lax monoidal functor from a symmetric monoidal subcategory $S$ of this category of cospans of groupoids to the category of R-modules with its usual tensor product then consists of a finitely presented projective $R$-module $Z(\mathfrak{G})=V_\mathfrak{G}$ for any groupoid which is an object in $S$, for any cospan $C$
\[
\mathfrak{G}_1 \longrightarrow \mathfrak{G}_{12} \longleftarrow \mathfrak{G}_2.
\]
in $S$ a $R$-linear map of $R$-modules 
\[
V_{\mathfrak{G}_1} \stackrel{Z(C)}\longrightarrow V_{\mathfrak{G}_2}
\]
and composition of cobordisms go to compositions of  $R$-linear maps:
\[
Z(C \circ D)= Z(C) \circ Z(D),
\]
this is axiom T2 of \ref{TT}.
We have $Z(\text{id})= \text{id}$. The empty groupoid maps to $R$ and the cobordism (\ref{eqn:identity}) maps to the identity of $V_{\mathfrak{G}}$.  In the situation of (\ref{eqn:coprodmors}) we have
\[
Z(C)\otimes_{R} Z(D)= Z(C \coprod D)
\]
as maps \[Z(\mathfrak{G}_{1})\otimes_{R} Z(\mathfrak{H}_{1})\longrightarrow Z(\mathfrak{G}_{2})\otimes_{R} Z(\mathfrak{H}_{2})
\]
\bibliography{Arithmetic field theory via pro-p duality groups}	
\end{document}